\documentclass{amsart}
\usepackage{pdfpages}
\usepackage{amscd}
\usepackage[mathscr]{euscript}
\RequirePackage{xspace}
\usepackage{enumerate}        
\usepackage{color}
\numberwithin{equation}{section}
\usepackage[all]{xy}
\usepackage{verbatim}
\usepackage{amssymb}

\let\cal\mathcal
\def\Ascr{{\cal A}}
\def\Bscr{{\cal B}}
\def\Cscr{{\cal C}}
\def\Dscr{{\cal D}}
\def\Escr{{\cal E}}
\def\Fscr{{\cal F}}

\def\Hscr{{\cal H}}
\def\Iscr{{\cal I}}

\def\Kscr{{\cal K}}
\def\Lscr{{\cal L}}
\def\Mscr{{\cal M}}
\def\Nscr{{\cal N}}
\def\Oscr{{\cal O}}

\def\Rscr{{\cal R}}

\def\Tscr{{\cal T}}

\def\Xscr{{\cal X}}
\def\Yscr{{\cal Y}}

\let\blb\mathbb
\def\CC{{\blb C}}

\def\FF{{\blb F}} 
\def\QQ{{\blb Q}}

\def \PP{{\blb P}}

\def \ZZ{{\blb Z}}
\def \TT{{\blb T}}

\def \HH{{\blb H}}

\def\proj{\operatorname{proj}}
\def\id{\text{id}}
\def\Id{\operatorname{id}}
\def\pr{\mathop{\text{pr}}\nolimits}

\def\Bimod{\operatorname{Bimod}}
\def\Ab{\mathbb{Ab}}

\def\Lotimes{\overset{L}{\otimes}}

\def\Mod{\operatorname{Mod}}
\def\mod{\operatorname{mod}}

\def\gr{\operatorname{gr}}

\def\ch{\mathop{\text{Ch}}}

\def\Qch{\operatorname{Qch}}
\def\coh{\mathop{\text{\upshape{coh}}}}

\def\gr{\operatorname {gr}}
\def\Spec{\operatorname {Spec}}

\def\Ext{\operatorname {Ext}}
\def\Hom{\operatorname {Hom}}

\def\End{\operatorname {End}}
\def\RHom{\operatorname {RHom}}
\def\REnd{\operatorname {REnd}}
\def\uRHom{\operatorname {R\mathcal{H}\mathit{om}}}

\def\cd{\operatorname {cd}}

\def\coker{\operatorname {coker}}
\def\ker{\operatorname {ker}}

\def\Tor{\operatorname {Tor}}
\def\End{\operatorname {End}}

\def\id{{\operatorname {id}}}

\def\Tot{\operatorname {Tot}}

\def\gldim{\operatorname {gl\,dim}}

\def\r{\rightarrow}

\DeclareMathOperator{\Proj}{Proj}
\DeclareMathOperator{\Pro}{Pro}

\DeclareMathOperator{\HHom}{\mathcal{H}\mathit{om}}

\DeclareMathOperator{\cor}{cor}

\newtheorem{lemma}{Lemma}[section]
\newtheorem{proposition}[lemma]{Proposition}
\newtheorem{theorem}[lemma]{Theorem}
\newtheorem{theoremdefinition}[lemma]{Theorem-Definition}
\newtheorem{corollary}[lemma]{Corollary}

\newtheorem{conclusion}[lemma]{Conclusion}

\newtheorem{lemmas}{Lemma}[subsection]
\newtheorem{propositions}[lemmas]{Proposition}

\newtheorem{corollarys}[lemmas]{Corollary}

\newtheorem{observations}[lemmas]{Observation}

\theoremstyle{definition}

\newtheorem{example}[lemma]{Example}
\newtheorem{definition}[lemma]{Definition}

\newtheorem{examples}[lemmas]{Example}

{

\newtheorem{case}{Case}

\newtheorem{hypothesis}{Hypothesis}
}

\theoremstyle{remark}

\newtheorem{remark}[lemma]{Remark}
\newtheorem{remarks}[lemmas]{Remark}
\newtheorem{notation}{Notation}

\newdimen\uboxsep \uboxsep=1ex
\def\uboxn#1{\vtop to 0pt{\hrule height 0pt depth 0pt\vskip\uboxsep
\hbox to 0pt{\hss #1\hss}\vss}}

\def\uboxs#1{\vbox to 0pt{\vss\hbox to 0pt{\hss #1\hss}
\vskip\uboxsep\hrule height 0pt depth 0pt}}

\def\Sp{\operatorname{Sp}}
\def\HH{\operatorname{HH}}

\def\Inj{\operatorname{Inj}}
\def\Ob{\operatorname{Ob}}
\def\Tw{\operatorname{Tw}}
\def\aa{\mathfrak{a}}
\def\bb{\mathfrak{b}}
\def\cc{\mathfrak{c}}
\def\dd{\mathfrak{d}}

\def\Ab{\mathbf{Ab}}

\let\oldmarginpar\marginpar
\def\marginpar#1{\oldmarginpar{\tiny #1}}
\def\cone{\operatorname{cone}}
\def\BB{\mathbb{B}}
\def\Free{\operatorname{Free}}

\def\CC{\mathbf{C}}
\def\HKR{\operatorname{HKR}}

\def\Perf{\operatorname{Perf}}
\def\dg{\operatorname{dg}}

\makeatletter\let\@wraptoccontribs\wraptoccontribs\makeatother

\title[An example of a non-Fourier-Mukai functor]{An example of a non-Fourier-Mukai functor between derived categories of coherent sheaves}
\author{Alice Rizzardo}
\address[Alice Rizzardo]{School of Mathematics\\The University of Edinburgh\\James Clerk Maxwell Building\\The King's Buildings\\Peter Guthrie Tait Road\\Edinburgh, EH9 3FD\\Scotland, UK}
\email{alice.rizzardo@ed.ac.uk}
\author{Michel Van den Bergh}
\address[Michel Van den Bergh]{Universiteit Hasselt\\ Universitaire Campus\\ 3590 Diepenbeek\\Belgium}
\email{michel.vandenbergh@uhasselt.be}
\contrib[With an appendix by]{Amnon Neeman}
\address[Amnon Neeman]{Centre for Mathematics and its Applications \\
        Mathematical Sciences Institute\\
        John Dedman Building\\
        The Australian National University\\
        Canberra, ACT 0200\\
        AUSTRALIA}
\email{Amnon.Neeman@anu.edu.au}

\thanks{The first author is a Postdoctoral Research Fellow at the University of Edinburgh. The second author is a senior researcher at the FWO}
\thanks{This research started at the Mathematical Sciences Research
Institute in 2013 with the support of the National Science Foundation under Grant
No.\ 0932078\,000. A number of results were obtained during research visits of the first author to the University of Hasselt,
supported by
ESF Exchange Grant 4498  in the framework of the project ``Interactions of Low-Dimensional Topology and Geometry with Mathematical Physics (ITGP)'' and by the FWO grant 1503512N ``Non-commutative algebraic geometry''
and by the second author to the International School for Advanced Studies (SISSA) at Trieste.
The
authors are very grateful to the NSF, the MSRI, the ESF, the steering committee of the ITGP project, the FWO and SISSA for their financial and moral support.}
\thanks{The research by Amnon Neeman was partly supported by the Australian Research Council}

\def\Ho{\operatorname{Ho}}
\begin{document}\begin{abstract}
Orlov's
 famous representability theorem asserts that any fully
faithful exact functor between the bounded derived categories of
coherent sheaves on smooth projective varieties is a Fourier-Mukai
functor.  In this paper we show that this result is false without the
fully faithfulness hypothesis. 
We also show that our functor
does not lift to the homotopy category of spectral categories if the ground field is $\QQ$.
\end{abstract}

\maketitle
\setcounter{tocdepth}{1}
\tableofcontents
\section{Introduction}
Throughout $k$ is a field of characteristic zero. All objects and constructions are assumed to be $k$-linear. 

Recall the following seminal theorem proved by 
Orlov almost 20 years ago:
\begin{theorem} \cite[Thm 2.2]{Orlov4}
\label{ref-1.1-0}
Let $X/k$, $Y/k$ be smooth projective schemes.
Then every fully faithful exact functor $\Psi:D^b(\coh(X))\r
  D^b(\coh(Y))$ is isomorphic to  a Fourier-Mukai functor associated to an
object in $D^b(\coh(X\times_k Y))$.
\end{theorem}
Although this theorem was generally believed to be false without the hypothesis
that $\Psi$ is fully faithful, no counterexamples were known. \emph{In the current paper
we fill this gap by constructing such a counterexample.} 

\medskip

Another example of a functor that is not Fourier-Mukai was obtained by Vologodsky \cite{Vologodsky} shortly after this paper appeared on the arXiv. While our base field has characteristic zero, the example in loc.\ cit.\ 
is over $\FF_p$.
Vologodsky's
functor is in fact a composition of derived functors and while it is not a Fourier-Mukai functor over $\FF_p$,
it is still represented by a morphism in the homotopy category of $\ZZ$-linear DG-categories (see Lemma \ref{lem:vg} below).
In contrast we will show in Appendix \ref{stable infinity}  that if $k=\QQ$ then our functor does not even have a lift to the homotopy category of spectral categories.

\medskip

Our counterexample is presented in Theorem \ref{ref-11.1-160} below, but
we will first discuss the underlying ideas on which it is based.
Some initial progress toward the construction of non-Fourier-Mukai functors had already been made in \cite{RizzardoVdB} where we systematically analyzed
 functors whose source category is the derived category of a field.
Leveraging this theory, we were able to construct a non-Fourier-Mukai
functor $D^b(\coh(X))\r D^b(\Qch(Y))$
by factoring through the localization at the generic point of
$X$. Unfortunately, such methods do not allow one to replace $\Qch(Y)$
by $\coh(Y)$.

\medskip

A highly nontrivial topological example of an exact functor which is not Fourier-Mukai in an appropriate sense is given in the beautiful paper \cite{Neeman6}.
Let $\Ho(\Sp)$ be the homotopy category of spectra. In \cite{Neeman6}, Neeman constructs an exact functor\footnote{The actual result proved
in loc.\ cit.\ is for $D(\ZZ[1/2])$.}
\begin{equation}
\label{ref-1.1-1}
D^b(\ZZ[1/2])\r \Ho(\Sp)[1/2]
\end{equation}
which sends $\ZZ[1/2]$ to the sphere spectrum $S^0$. 

\medskip

Now note that
$
\Ho(\Sp)^{-1}(S^0,S^0)=\pi_1(S^0)=\ZZ/2\ZZ
$.
So  
\[
\Ho(\Sp)[1/2]^{-1}(S^0,S^0)=0
\] and Neeman's proof strongly
suggests that it is precisely
this gap in the negative $\Ext$'s  that makes this example work.

\medskip

In the first version of this paper which appeared on the arXiv the authors
proved a generalization of this result, valid for
more complicated categories, at the cost of requiring the vanishing of more negative $\Ext$-groups.
Our proof was for triangulated categories of algebraic nature so it did not recover Neeman's original result.
However Amnon Neeman succeeded in finding a yet more general
argument 
which is valid for triangulated categories satisfying Neeman's 
more restrictive axioms \cite{Neeman} 
(i.e.\ all those that occur in nature)
so it encompasses both our algebraic result and Neeman's result for $\Ho(\Sp)[1/2]$. 
Amnon Neeman's proof is included as Appendix~\ref{sec:amnon} in the current paper.
Its main results are summarized in \S\ref{sec:construction_} and in particular
the now simple construction of the functor \eqref{ref-1.1-1} is presented
in Example \ref{ex:original_example}.

\medskip

For the purpose of this introduction we state a simple corollary of these results:
\begin{theorem} \label{ref-1.2-2} (see Appendix \ref{ref-B-182}) 
\begin{enumerate}
\item Let $B$ be a DG-algebra and let $R\r H^0(B)$ be a $k$-algebra
  morphism. Assume that $\gldim R=m$ and $H^i(B)=0$ for
  $i=-1,\ldots,-m$. Then the natural functor
  \[
R\r D(B):R\mapsto B
\]
($R$ considered as a one-object category)
extends to an exact functor
\begin{equation}
\label{ref-1.2-3}
{L}:D^b(R)\r D(B).
\end{equation}
\item
\label{one}
If ${L}$ is  isomorphic to a functor of the form
$U\Lotimes_R-$, for $U$ a complex of $k$-central $B{-}R$-bimodules, then the graded $k$-algebra morphism
$
R\r H^\ast(B)
$
may be lifted to an $A_\infty$-morphism $R\r B$.
\end{enumerate}
\end{theorem}
As expected, if $R$ is a field  Theorem \ref{ref-1.2-2}(1) imposes no conditions on~$B$  and hence this theorem
may be regarded as an extension of the basic principle underlying the constructions in \cite{RizzardoVdB}.

\medskip

In order do be able to apply part (2) of Theorem \ref{ref-1.2-2} we note that there are very concrete obstructions against the lifting of a
morphism of graded rings to an $A_\infty$-morphism (see
\S\ref{ref-7.2-82} below). In the setting of Theorem \ref{ref-1.2-2}, these obstructions
take values in the Hochschild cohomology groups
\begin{equation}
\label{ref-1.3-4}
\HH^{i}(R,H^{2-i}(B))
\end{equation}
for $i\ge m+3$.  Since such obstructions are easily controlled,
Theorem \ref{ref-1.2-2}(1) immediately gives a supply of functors
which are non-Fourier-Mukai in an appropriate sense. The most basic
case is the following.
\begin{proposition} \label{ref-1.3-5} (see Appendix \ref{ref-C.2-184}).
Let $R$ be a $k$-algebra of global dimension~$m$ and let $M$ be a $k$-central $R$-bimodule. Let
$\eta\in\HH^n(R,M)$ be a non-zero class in Hochschild cohomology, with $n\ge m+3$. Let $R_\eta$ be the
$A_\infty$-algebra $R\oplus \Sigma^{n-2} M\epsilon$, $\epsilon^2=0$, whose multiplication is twisted by $\eta$ (i.e.\ $m_{R_\eta,n}$ is given by $\eta:R^{\otimes n}\r M$). 
Finally, let $R^{\dg}_\eta$ be the DG-hull of $R_\eta$ (see Appendix \ref{ref-C.1-183} below). Then the functor~$L$ in Theorem \ref{ref-1.2-2}(1), with $B=R^{\dg}_\eta$ is not isomorphic to a functor of the form
$U\Lotimes_R-$ for $U$ a  $R^{\dg}_\eta{-}R$-DG-bimodule.
\end{proposition}
In this proposition we have introduced $R^{\dg}_\eta$ to stay within the world of DG-algebras but in fact
the distinction between $R_\eta$ and $R^{\dg}_\eta$ is only of technical relevance and we will ignore it in the rest of this introduction.

\medskip

The functors constructed via Proposition \ref{ref-1.3-5} may be called ``non-Fourier-Mukai'' in a generalized sense.
Unfortunately they are
essentially non-geometric because:
\begin{enumerate}
\item They involve DG-algebras.
\item  If $R$ is commutative and finitely generated over an algebraically closed 
field,  the Hochschild dimension of $R$ is equal to the global
dimension. This implies for example that the inequality $n\ge m+3$ cannot be satisfied if we want
$\Spec R$ to be an affine variety. Note however that the equality can be
satisfied with $R$ being \emph{essentially} of finite type (e.g. take
$R=M=k(x,y,z)$).
\end{enumerate}
To get rid of the first problem  in the ring case, one may consider the situation where there is some $k$-algebra morphism $f:S\r R$ such that\footnote{Here $f_\ast$ is shorthand for
  the composition ${}_SS_S\xrightarrow{f} {}_SR_S\xrightarrow{{}_S \eta_S} {}_SM_S[n]$ in $D(S\otimes_k S^{\circ})$.
}
 $0=f_\ast(\eta)\in \HH^n(S,M)$.
In that case one may show that there is a commutative diagram  of $A_\infty$-algebras (see Proposition \ref{ref-7.2.6-95} below)
\begin{equation}
\label{ref-1.4-6}
\xymatrix{
R_\eta\ar[dr]&&\ar@{.>}[ll]_{\tilde{f}} S\ar[dl]^{f}\\
&R
}
\end{equation}
Starting from   \eqref{ref-1.4-6} one could hope that in some cases the composition
\begin{equation}
\label{ref-1.5-7}
\Psi:D^b(R)\xrightarrow{\eqref{ref-1.2-3}}  D(R_\eta)\xrightarrow{\tilde{f}_\ast} D(S)
\end{equation}
is not given by tensoring with a complex of bimodules. We have not really investigated how well this method works for contructing non-Fourier-Mukai functors between derived categories of rings
but the underlying idea is used in the geometric case we consider below. 
By a stroke of luck the geometric construction will also be seen to yield an example of a non-Fourier-Mukai functor between derived categories of finite dimensional algebras.
See Corollary \ref{ref-1.5-11} below.

\medskip

To deal with the second problem note that, if $X$ is a smooth projective variety of dimension $m$, then $\gldim \Qch(X)=m$, whereas the Hochschild dimension of $X$ is $2m$
(see e.g. \S\ref{ref-9.6-144} below).  So in that case there is no problem satisfying the inequality $n\ge m+3$ when $m\ge 3$.

\medskip

As a thought experiment we will consider the following situation: $X$ is as in the previous paragraph, $n\ge m+3$, $M$ is an $\Oscr_X$-module. Let $0\neq \eta
\in \HH^{n}(X,M)\overset{\text{def}}{=}\Ext^n_{X\times_k X}(i_{\Delta,\ast}\Oscr_X,i_{\Delta,\ast} M)$ ($i_{\Delta}:X\r X\times X$ being the diagonal).
Then we expect there to be some kind of derived deformation (or infinitesimal thickening) $X\r X_\eta$ corresponding to $\eta$. We will discuss this further below,
but for now assume $X_\eta$ exists.
Assume furthermore there is some morphism $f:X\r Y$ with $Y$ smooth such that $f_\ast(\eta)=0$. Then, as in the ring case, we \emph{expect} there to be a diagram of the type
\begin{equation}
\label{eq:geometric}
\xymatrix{
X_\eta\ar@{.>}[rr]^{\tilde{f}}&& Y\\
&X\ar[ul]\ar[ur]_f
}
\end{equation}
(the arrows are reversed with respect to \eqref{ref-1.4-6} by the usual algebra-geometry duality) which should in principle allow us to define a functor as in \eqref{ref-1.5-7}.

\medskip

Now we have to deal with the question: what is $X_\eta$? One canonical answer is to use $\eta$ to deform an enhancement of $D(\Qch(X))$ \cite{Kuznetsov1,OrlovLunts,LuntsSchnurer}. But then one hits the so-called 
``curvature'' problem: the result will in general only be a $\operatorname{cA}_\infty$-category \cite{DeDekenLowen,LowenVdBFormal}, i.e.\ roughly speaking it will not satisfy $d^2= 0$.
Homological algebra over $\operatorname{cA}_\infty$-categories is possible \cite{DeDekenLowen,Positselski1} but presents rather serious technical difficulties.

Another approach is to view $X_\eta$ as a kind of DG-gerbe on $X$. However in our examples $\eta$ will be very non-local, so the ``higher gluing'' 
required to understand $X_\eta$ will be necessarily subtle.

In this paper we have
 opted for a third approach (based on \cite{lowenvdb2}) which is much cheaper but nonetheless sufficient for our purposes. The idea is to embed $\Qch(X)$
into a \emph{category of presheaves} associated to an affine covering of $X$. Such presheaves form a module category so we can directly apply
the algebraic constructions discussed above. In particular there is no curvature problem.

Let $X=\bigcup_{i=1}^n
U_i$ be an affine covering. For $I\subset \{1,\ldots,n\}$ put
$U_I=\bigcap_{i\in I} U_i$. 
Let $\Iscr$ be the set $\{I\subset\{1,\ldots,n\}\mid I\neq \emptyset\}$
  and let $\Xscr$ be the category with objects
 $\Iscr$ and $\Hom$-sets
\begin{equation}
\label{ref-1.6-8}
\Xscr(I,J)=
\begin{cases}
\Oscr_X(U_J)&\text{$I\subset J$}\\
0&\text{otherwise.}
\end{cases}
\end{equation}
In other words $\Xscr$ is the subcategory of $\Qch(X)$ spanned by the objects
$(i_{U_I,\ast} \Oscr_{U_I}))_{I\in\Iscr}$ where we only allow maps $I\r J$ when $I\subset J$. Since $\Iscr$ is finite one may even think of $\Xscr$ as an actual ring $\bigoplus_{I,J\in\Iscr}\Xscr(I,J)$.

It is easy to see that $\Mod(\Xscr)$ is the category of modules over the presheaf of rings $(I,\Gamma(U_I,\Oscr_{U_I}))_{I\in \Iscr}$.
In particular $\Mod(\Xscr)$ contains $\Qch(X)$ as a full
subcategory. Furthermore by the ``Special Cohomology Comparison
Theorem'' \cite{GS1,lowenvdb3} one has
$\HH^\ast(\Xscr,\Mscr)=\HH^\ast(X,M)$ (see \S\ref{ref-8.2-103} below) where
$\Mscr$ is the $\Xscr-\Xscr$-bimodule associated to $M$ defined by a
similar formula as \eqref{ref-1.6-8}. It follows that we may define and $A_\infty$-category
$\Xscr_\eta$ in exactly the same way as we defined $R_\eta$.

\medskip

Assume now that $f,X,Y,\eta$ are as above and that in addition $f:X\r Y$ is a closed immersion, and start with an affine covering of $Y$. By giving $X$ the induced covering, we may then construct a diagram of $A_\infty$-categories and functors
\[
\xymatrix{
\Xscr_\eta\ar[dr]&&\ar@{.>}[ll]_{\tilde{f}} \Yscr\ar[dl]^{f}\\
&\Xscr
}
\]
(the arrows are again reversed with respect to \eqref{eq:geometric} since we are now back in an algebraic framework as in \eqref{ref-1.4-6}) 
and we may construct an exact functor
\begin{equation}
\label{ref-1.7-9}
\Psi: D^b(\Qch(X))\xrightarrow{L} D_{\Qch}(\Xscr_\eta) \xrightarrow{\tilde{f}_\ast} D_{\Qch}(\Yscr)\cong D(\Qch(Y))\, ,
\end{equation}
where the first functor is a geometric version of \eqref{ref-1.2-3} and
where $D_{\Qch}(-)$ means that we only consider complexes with quasi-coherent cohomology, through the embedding $\Qch(X)\subset \Mod(\Xscr)$. In fact, since $\Qch(X)$ has enough
injectives but not projectives, the first functor
is given by a  construction dual to the one presented in Theorem \ref{ref-1.2-2}(1). See Remark \ref{ref-10.2-152} below.

Now we may state our main theorem.
\begin{theorem} \label{ref-11.1-160} (see \S\ref{ref-11-159} below) Let $X$
  be a smooth quadric in $Y=\PP^{4}$ whose defining equation has maximal isotropy index\footnote{This condition is only relevant for a non-algebraically closed base field.
One may take  $x_0^2+x_1x_2+x_3x_4=0$.
}
and let $f:X\r
  Y$ be the inclusion. Let $M=\omega_X^{\otimes 2}$ and let $0\neq \eta\in\HH^{6}(X,\omega_X^{\otimes 2})\cong k$. Then
  $f_\ast\eta\in \HH^{6}(Y,f_\ast (\omega_X^{\otimes 2}))$ is zero. The functor
  $\Psi$ in \eqref{ref-1.7-9} restricts to an exact functor
\[
\Psi:D^b(\coh(X))\r D^b(\coh(Y))
\]
which is not a Fourier-Mukai functor. 
\end{theorem}
Recall \cite{CS2,Rizzardo1} that even for a non-Fourier-Mukai functor one may still define sheaves $\Hscr^i$ on $X\times_k Y$ which would be the cohomology of the kernel - if the latter existed. In our case
we have (see \eqref{ref-10.6-158} below)
\begin{equation}
\label{eq:virtual}
\Hscr^i=
\begin{cases}
\Oscr_{\Gamma_f}&\text{if $i=0$}\\
\omega_{\Gamma_f}^{\otimes -2}&\text{if $i=4$}\\
0&\text{otherwise.}
\end{cases}
\end{equation}
where $\Gamma_f\subset X\times_k Y$ is the graph of $f$.

\medskip

To conclude this introduction, let us indicate how we prove that $\Psi$ in Theorem \ref{ref-11.1-160} 
is not Fourier-Mukai (see \S\ref{ref-11-159} below). The technical details are somewhat involved
but the underlying idea is the following. The basic feature of a Fourier-Mukai functor is 
that it is compatible with \emph{base change}. In fact, if $\Psi:D^b(\coh(X))\r D^b(\coh(Y))$ 
is a Fourier-Mukai functor and $Z$ is a smooth proper scheme over $k$, then by extending 
the kernel of $\Psi$ we obtain a Fourier-Mukai functor $D^b(\coh(X\times Z))\r D^b(\coh(Y\times Z))$. If $Z=X$ then the kernel of $\Psi$ is the image via this functor of the structure sheaf $\Oscr_\Delta$ of the diagonal $\Delta\subset X\times X$.

Because of this last fact we do not expect base change to hold for non-Fourier-Mukai functors suggesting
 a possible method to identify them.
Unfortunately it seems not so obvious to give a workable definition of base change in this more general setting. There is one situation
however which is easier to handle and which applies to our example. Assume that $X$ has a tilting bundle $T$ and let $\Gamma=\End_X(T)$. Then
$D^b(\coh(X\times X))$ is equivalent to $D^b(\coh(X)_\Gamma)$ where $\coh(X)$ is the abelian category of
coherent sheaves on $X$ equipped with a $\Gamma$-action and under this equivalence $\Oscr_\Delta$
correponds to $T$ (which is indeed a coherent sheaf on $X$, naturally equipped with a $\Gamma$-action).
So it is a natural idea to replace 
$D^b(\coh(X\times X))$ by the more algebraic $D^b(\coh(X)_\Gamma)$.  

If $\Psi$ is a Fourier Mukai functor then since the standard constructions of derived pullback, tensor product and pushforward are compatible with the action of $\Gamma$, the kernel for $\Psi$
defines at the same time a ``$\Gamma$-equivariant lift'' $\Psi_\Gamma:D^b(\coh(X)_\Gamma)\r D^b(\coh(Y)_\Gamma)$ of $\Psi$, i.e.
a functor which behaves as $\Psi$ if we ignore the $\Gamma$-action. So we should try to prove that
such a $\Gamma$-equivariant lift of $\Psi$ does not exist. The above discussion suggests we should only consider the object $\Psi_\Gamma(T)$ since it represents the (would be) kernel of $\Psi$ under the equivalence
$D^b(\coh(X\times Y))\cong D^b(\coh(Y)_\Gamma)$.

\medskip

We are now ready to describe the key argument. By functoriality $\Psi(T)$ is an
object in $D(\coh(Y))_\Gamma$ (that is: the category of objects in $D(\coh(Y))$ equipped with a
$\Gamma$-action) and as explained above, if $\Psi$ is Fourier-Mukai then $\Psi(T)$ lifts to the object $\Psi_\Gamma(T)$ in $D^b(\coh(Y)_\Gamma)$.
It turns out that in our setting there is a homological obstruction against lifting $\Psi(T)$ under the obvious functor
\begin{equation}
\label{eq:lift}
D(\coh(Y)_\Gamma)\r D(\coh(Y))_\Gamma
\end{equation}
and we have to prove it is non-zero. Now it is more or less a tautology that $L(T)\in D^b(\Xscr_\eta)_\Gamma$ (with $L$ as in \eqref{ref-1.7-9}) does not lift to $D(\Xscr_\eta\otimes
\Gamma)$ (as $L$ is a non-Fourier-Mukai functor). We show in \S\ref{sec:obstructions} that by general principles the obstruction to the latter lift is sent by a suitable
incarnation of $f_\ast$ to the corresponding
obstruction against the lift \eqref{eq:lift}. Now all we have to do is to show in our example that $f_\ast$ induces an isomorphism between the cohomology groups which are the targets for the obstruction.
This is carried out \S\ref{ref-11-159}.

\medskip

Valery Lunts suggested adding the following immediate corollary of Theorem \ref{ref-11.1-160} for finite dimensional algebras.
 Recall that Rickard proved in \cite[Theorem 3.3]{Ri2} that if  $\Gamma$ and $\Lambda$ are derived equivalent $k$-algebras, then there is an equivalence of the form
\[
U\stackrel{L}{\otimes}_{\Gamma}-: D^b(\mod(\Gamma))\to D^b(\mod(\Lambda))
\]
where $U$ is a complex of $\Lambda-\Gamma$ bimodules. However, in contrast with the geometric situation (see Theorem \ref{ref-1.1-0}), it is unknown whether all derived equivalences between rings are of this form. The next corollary states that there
do indeed exist exact functors between derived categories of finite dimensional algebras that are not given by
 tensoring with a bimodule. They are however not derived equivalences.
\begin{corollary}\label{ref-1.5-11}
With the notation of Theorem \ref{ref-11.1-160}, let $\Gamma=\End_X(T)$ where $T$ is the tilting bundle on $X$ described in Theorem \ref{ref-11-159},
$\Lambda=\End_{\PP^{4}}\allowbreak\left(\oplus_{i=0}^{4}\Oscr(i)\right)$. Then the functor $\Psi$ induces a functor 
\begin{equation}
\label{eq:Phifin}
\Phi: D^b(\mod(\Gamma))\to D^b(\mod(\Lambda))
\end{equation}
which is not of the form $U\Lotimes_{\Gamma}-$ for $U$ a complex of $\Lambda-\Gamma$-bimodules.
\end{corollary}
\begin{proof} We have $D^b(\coh(X))\cong D^b(\mod(\Gamma))$, $D^b(\coh(Y)\cong D^b(\mod(\Lambda))$ so that we may indeed define the functor $\Phi$ as in \eqref{eq:Phifin}. If $\Phi$ were of
the form $U\Lotimes_{\Gamma}-$ then it would be induced from a DG-functor and hence the same would be true for $\Phi$ as in Theorem \ref{ref-11.1-160}. But then the latter would be a Fourier-Mukai
functor by \cite[Thm 8.15]{Toen}.
\end{proof}

A number of extensions and variants of Orlov's theorem are known. See e.g.\ \cite{Ballard,COV,CS4,CS3,CS2,Kawamata2,OrlovLunts,Rizzardo1,Rizzardo}. For excellent surveys on the current state of knowledge see~\cite{CS1,CanonacoStellari}.
\section{Outline}
The paper consists of a number of parts which are  independent of each other.
\begin{itemize}
\item In \S\ref{sec:maintechnical} we discuss our main technical result 
(a dual version of Theorem \ref{ref-1.2-2})
which is at the heart
of our construction of a non-Fourier-Mukai functor. The proof reduces quickly to the general result by Amnon Neeman contained in Appendix~\ref{sec:amnon}.
\item In \S\ref{sec:deformations},\S\ref{sec:obstructions} we discuss the main facts concerning $A_\infty$-categories that we will need in the rest of the paper. This culminates in \S\ref{ref-7.3-97} where we discuss the obstructions for lifting objects under the
functor $D(\bb\otimes_k \Gamma)\r D(\bb)_\Gamma$. The relevance of this has been explained in the introduction.
\item In \S\ref{sec:sheaves} we  relate the properties of a quasi-compact separated scheme $X$ to
similar properties
of the corresponding category $\Xscr$ defined in the introduction. This material is necessary as we use $\Xscr$ to
deform the derived category of quasi-coherent sheaves on $X$.
\item In \S\ref{sec:divisors} we discuss the behavior of Hochschild cohomology under restriction to a smooth hypersurface. This is used to
verify that the quadruple $(X,Y,f,\eta)$ in Theorem \ref{ref-11.1-160} indeed has the property $\eta\neq 0$, $f_\ast \eta=0$.
\item In \S\ref{sec:construction}, \S\ref{ref-11-159} we give the proof of Theorem \ref{ref-11.1-160}.
\item We include several appendices. In Appendix \ref{sec:virtual} we prove \eqref{eq:virtual}. In Appendix \ref{stable infinity} we prove that $\Psi$ does not lift to a functor of spectral categories (if $k=\QQ$).
In Appendices \ref{ref-B-182}, \ref{sec:ref-1.3-5}                 we prove Theorem \ref{ref-1.2-2} and Proposition \ref{ref-1.3-5} from the introduction. Finally Appendix \ref{sec:amnon} is written by Amnon Neeman. In contrast to the other appendices, this one is essential for our paper!
\end{itemize}
\section{Acknowledgements}
The authors thank Greg Stevenson and Adam-Christiaan Van Roosmalen for
a number of interesting discussions which reinforced our interest in the
problem.

The first author would like to thank the University of Hasselt for its
hospitality during  two recent visits to Belgium. She would also like to
thank MSRI for providing a welcoming environment during the
Noncommutative Geometry and Representation Theory semester, which
encouraged many fruitful conversations, some of which got this project
started. She is grateful to Jon Pridham for providing very useful ideas and guidance for Appendix \ref{stable infinity}, and to Julian Holstein for help with understanding stable infinity categories. Finally, she would like to thank her advisor, Johan de Jong,
for originally suggesting she look into exact functors between derived categories in the non fully faithful case.

The second author thanks  the International School for Advanced Studies (SISSA) in the beautiful city  of Trieste for the wonderful working conditions it provides.

The authors thank Andrei C\u{a}ld\u{a}raru for useful discussions concerning derived self intersections and Fourier-Mukai functors. They  are deeply grateful to Wendy Lowen for allowing the use of some ideas developed during  an ongoing cooperation with the second author.

Finally the authors thank Amnon Neeman for providing them with an alternative and much more general proof for the extension of \eqref{ref-1.1-1}.
\section{Notation and preliminaries}
Throughout $k$ is a field of characteristic zero and all constructions will be $k$-linear. 
\subsection{Modules and bimodules over categories}
Let $\aa$ be a pre-additive category. A left $\aa$-module $M$ is a covariant
additive functor $M:\aa\r \Ab$. We view it as a collection of
abelian groups $M(A)_{A\in \Ob(\aa)}$ depending covariantly on $A$. 
We write $\Mod(\aa)$ for the additive category of left $\aa$-modules. 

A right $\aa$-module is an object in $\Mod(\aa^\circ)$. If $\aa$,
$\bb$ are $k$-linear categories, a ($k$-linear) $\aa$-$\bb$-bimodule
is an object in $\Mod(\aa\otimes_{k} \bb^\circ)$. We view it as a 
collection of $k$-vector spaces $M(B,A)_{B\in \Ob(\bb),A\in\Ob(\aa)}$ depending
contravariantly on $B$ and covariantly on $A$. We will sometimes write
$\Bimod_k(\aa,\bb)$ for $\Mod(\aa\otimes_k \bb^\circ)$. Given $M\in \Mod(\bb)$ and a functor $\bb\to \aa$, we will write $\aa\otimes_{\bb} M$ to denote the tensor product of $M$ with the $\aa$-$\aa$ bimodule $\aa(-,-)$ viewed as an $\aa$-$\bb$-bimodule.

If $\aa$, $\bb$ are DG-categories, the above notions have obvious generalizations
to DG-modules, DG-bimodules, etc. We write $\underline{\Mod}(\aa)$ for the
category of left DG-$\aa$-modules.
\subsection{$A_\infty$-notions}
If $\aa$ is a DG-graph, we denote by $\BB \aa$ its 
bar-cocategory. I.e.\ $\Ob(\BB\aa)=\Ob(\aa)$ and
\[
\BB\aa(A,B)=\bigoplus_{A_1,\ldots,A_{i-1}\in \Ob(\aa)}\Sigma \aa(A_{i-1},B)\otimes \ldots \otimes \Sigma \aa(A,A_1)
\]
An $A_\infty$-structure on $\aa$ is a codifferential $b_\aa$ on $\BB\aa$.
Similarly, an $A_\infty$-functor between $A_\infty$-categories $\aa$, $\bb$
is a cofunctor $f:\BB\aa\r\BB\bb$ 
such that $b_{\bb}\circ f=f\circ b_{\aa}$.  

As usual, we describe $A_\infty$-structures and morphisms via their
Taylor coefficients: $(b_{\aa,j})_j$, $(f_j)_j$ which may be evaluated
on sequences of $j$ composable maps.

\emph{All $A_\infty$-constructions will always be implicitly assumed to be  strictly unital}. Note that any reasonable $A_\infty$-construction can be strictified, which is ultimately due
to the fact that Hochschild cohomology may be computed using normalized cocycles. See
\cite[Ch.\ 3]{Lefevre}. We routinely apply standard constructions for $A_\infty$-algebras to $A_\infty$-categories. This simply means
that operations are only applied to composable arrows.

If $\bb$ is an $A_\infty$-category, we will denote by $C_\infty^u(\bb)$ the category of  (strictly unital) $A_\infty$-$\bb$-modules with $A_\infty$-morphisms.  Let $D_\infty(\bb)$
be obtained from $C_\infty^u(\bb)$ by identifying homotopic maps. This is one of several equivalent constructions for the derived category
of an $A_\infty$-category. See \cite{Lefevre} for details. 

\section{Construction of a functor}
\label{sec:construction_}
We briefly summarize the results we will need 
from Appendix \ref{sec:amnon} written by Amnon Neeman 
and also present a relevant example.

If $\Hscr$ is a full subcategory of a triangulated category $\Tscr$
then we denote by $\Hscr^\ast$ the \emph{extension closure} of $\Hscr$, i.e.
the smallest full subcategory of $\Tscr$ which contains 
$\Hscr$ and which has the property that if there
is a distinguished triangle  $x\r y\r z$ with $x,z\in\Hscr^\ast$ 
then $y\in \Hscr^\ast$. We say that $\Hscr$ is \emph{extension closed}
if $\Hscr^\ast=\Hscr$.

\begin{definition}
Let $H:\Rscr\r\Tscr$ be an exact functor between triangulated
categories. The pair of full subcategories $(\Ascr\subset\Rscr,\Bscr\subset\Rscr)$ is called a
\emph{good couple with respect to $H$} if
\begin{enumerate}
\item
$\Sigma^{-1}\Ascr\subset\Ascr$ and $\Sigma\Bscr\subset\Bscr$.
\item
The map $\Rscr(a,b)\r\Tscr(Ha,Hb)$ is an isomorphism if $a\in\Ascr$ and
$b\in\Bscr$, and is surjective if $a\in\Ascr$ and $b\in\Sigma^{-1}\Bscr$.
\end{enumerate}
A good couple $(\Ascr,\Bscr)$ is called \emph{excellent} if $\Ascr$,
$\Bscr$ are extension closed.
\end{definition}
\begin{remark}
\label{rem:digest}
 If $(\Ascr,\Bscr)$ is a good couple
 for $H$, then it is clear that the restriction of $H$ to $\Ascr\cap\Bscr\subset\Rscr$ is
 fully faithful.
\end{remark}
\begin{proposition}[See Corollaries \ref{C1.5}, \ref{C22.5} and their proofs]
\label{prop:digest}
If $(\Ascr,\Bscr)$ is a good couple then $(\Ascr^\ast,\Bscr^\ast)$ is
an excellent couple. Moreover if $\Cscr=H(\Ascr\cap \Bscr)$ is
the essential image of $\Ascr\cap \Bscr$ then $\Cscr^\ast\subset H(\Ascr^\ast
\cap \Bscr^\ast)$.
\end{proposition}
The following result is a version of Theorem \ref{T1.13} of Appendix \ref{sec:amnon}. We have slightly
altered the notation so as to be more compatible with the body of the paper.
\begin{theorem}
\label{thm:digest}
Let $H:\Rscr\r\Tscr$ be an exact functor. Assume the
category $\Rscr$ satisfies the axioms of 
the article \cite{Neeman91}. 
Suppose further that
$\Tscr$ has a non-degenerate
$t$-structure with heart $\Tscr^\heartsuit$, let $\Hscr:\Tscr\r\Tscr^\heartsuit$
be the
standard homological functor from $\Tscr$ to the heart, and
let $\Dscr\subset\Tscr^\heartsuit$ be a
full, abelian subcategory closed under extensions and define
\[
\Tscr^b_{\Dscr}=\left\{t\in\Tscr\left|
\begin{array}{c}
\Hscr^i(t)=0\text{ for all but finitely many }i\in\ZZ\\
\Hscr^i(t)\in\Dscr\text{ for every }i\in\ZZ  
\end{array}
\right.
\right\}
\]
Assume now that $(\Ascr,\Bscr)$ is an excellent couple in $\Rscr$ such that $\Dscr\subset H(\Ascr\cap\Bscr)$.
Then there exists an exact functor $G:\Tscr^b_{\Dscr}\r\Rscr$ which fits in a
commutative diagram
\vspace*{2mm}
\begin{equation}
\label{eq:basicdiagram}
\xymatrix@=4em{
\Dscr\ar@{^(->}@/^1.5em/[rr]\ar@{_(.>}[r]_-{G\mid \Dscr}\ar@{^(->}[d] & \Ascr\cap \Bscr\ar@{^(->}[d]\ar[r]^{\cong}_-{H\mid \Ascr\cap \Bscr}&H(\Ascr\cap\Bscr)\ar@{^(->}[d]\\
\Tscr^b_{\Dscr}\ar@{_(->}@/_1.5em/[rr]\ar@{.>}[r]^G & \Rscr \ar[r]^H & \Tscr
}
\vspace*{3mm}
\end{equation}
\end{theorem}
\begin{example}[Neeman]
\label{ex:original_example}
Suppose $\Rscr=\Ho(\Sp)[1/2]$ is the homotopy category of spectra
with 2 inverted, let $\Tscr=D(\ZZ[1/2])$, and let
$H:\Rscr\r\Tscr$ be the functor taking a spectrum to its
singular chain complex. 
We define $\Free(S^0)$ to be the full subcategory
of $\Rscr$ whose objects are bouquets of zero-spheres.
Set $\Ascr=\{\Sigma^n P\mid P\in \Free(S^0),\,n\leq 1\}$
and $\Bscr=\{\Sigma^n P\mid P\in \Free(S^0),\,n\geq 0\}$.
That is the objects of $\Ascr$
and $\Bscr$ are just shifts of bouquets of the zero-sphere
$S^0$, with the shifts
as prescribed. 
We claim that $(\Ascr,\Bscr)$ is a good couple. This comes down to
$\Rscr(S^0,S^0)=\pi_0(S^0)[1/2]=\ZZ[1/2]$, $\Rscr(S^0,\Sigma^{-1}S^0)=\pi_1(S^0)[1/2]=0$.

We define $\Free(\ZZ[1/2])$ to be
the category of free $\ZZ[1/2]$--modules, viewed
as objects of $\Tscr$ concentrated in degree 0.
Let $\Cscr$ be the essential image of $\Ascr\cap\Bscr$. Then $\Cscr$
contains $\{\Sigma^n P\mid P\in \Free(\ZZ[1/2]),\,n=0\text{ or }1\}$,
and it's easy to see that  $\Cscr^*$ contains the heart $\Tscr^{\heartsuit}$ of the
standard $t$-structure on $\Tscr$ (a single extension is enough as $\ZZ[1/2]$ is
hereditary).
Proposition \ref{prop:digest} now informs us that $(\Ascr^\ast,\Bscr^\ast)$, is an
excellent couple and that the essential image of $\Ascr^\ast\cap\Bscr^\ast$ contains 
$\Tscr^{\heartsuit}$. Applying Theorem \ref{thm:digest} with $\Dscr=\Tscr^{\heartsuit}$ we
obtain a functor $G:D^b(\ZZ[1/2])\r \Ho(\Sp)$ such that the composition
$HG$ is the inclusion $D^b(\ZZ[1/2])\r D(\ZZ[1/2])$.
\end{example}
Our main application of Theorem \ref{thm:digest} will be Proposition \ref{ref-5.3.1-69} below. Another
application will be given in Appendix \ref{ref-B-182}.
\section{The main technical result}
\label{sec:maintechnical}
\subsection{Derived injectives}
\label{ref-5.1-60}
This is part of ongoing work of the second author 
with Francesco Genovese and Wendy Lowen. Assume that $\Tscr$ is a well generated triangulated category
  equipped with a
  t-structure with heart $\Tscr^{\heartsuit}$, such that $H^0(-)$ respects coproducts. If $I\in \Inj\Tscr^{\heartsuit}$, by Brown
  representability \cite{Neeman11}  the cohomological functor $\Tscr\mapsto\Ab$ given
  by $\Tscr^{\heartsuit}(H^0(-),I)$ is representable. Denote the representing
  object by ${L}(I)$. So we have that for $X\in \Tscr$
\begin{equation}
\label{ref-5.1-61}
\Tscr^{\heartsuit}(H^0(X),I)=\Tscr(X,{L}(I)).
\end{equation}
We will call ${L}(I)$ the \emph{derived injective}
  associated to $I$. 
\begin{remarks} If $\Tscr=\Ho(\Sp)$ is the homotopy category of spectra with the standard $t$-structure with heart
$\Ab$ and $I=\QQ/\ZZ$, then ${L}(I)$ is the Brown-Comenetz dual of the sphere
spectrum.
\end{remarks}
The following properties are easily verified:
\begin{equation}
{L}(I)\in \Tscr_{\ge 0},
\end{equation}
\begin{equation}
\label{ref-5.3-62}
 H^0({L}(I))=I.
\end{equation}
For $I,J\in \Inj \Tscr^{\heartsuit}$, one has
\begin{equation}
\label{ref-5.4-63}
\Tscr({L}(I),\Sigma^i {L}(J))=
\begin{cases}
\Tscr^{\heartsuit}(I,J)&\text{if $i=0$}\\
0&\text{if $i>0$}
\end{cases}
\end{equation}
Thus, in particular, we have a fully faithful functor
\begin{equation}
\label{ref-5.5-64}
{L}:\Inj\Tscr^{\heartsuit}\r \Tscr:I\mapsto {L}(I).
\end{equation}
\subsection{Derived injectives in a DG-category}
Now assume that $\Dscr$ is $D(\cc)$ with $\cc$ a DG-category concentrated in
degrees $\le 0$.
Equip $\Dscr$ with the standard t-structure \cite{Keller} with heart $\Tscr^{\heartsuit}=\Mod(H^0(\cc))$. One verifies
for $I\in \Inj \Tscr^{\heartsuit}$
\begin{equation}
\label{ref-5.6-65}
H^\ast({L}(I))=\Hom_{H^0(\cc)}(H^{-\ast}(\cc),I)
\end{equation}
as graded $H^\ast(\cc)$-modules.
We also find
\begin{equation}
\label{ref-5.7-66}
\Hom^i_{D(\cc)}({L}(I),{L}(J))=\Hom_{H^0(\cc)}(H^{-i}({L}(I)),J)=\Hom_{H^0(\cc)}(\Hom_{H^0(\cc)}(H^{i}(\cc),I),J)
\end{equation}
Finally note the following
\begin{lemmas} 
\label{ref-5.2.1-67}
There is a commutative diagram 
\[
\xymatrix{
\Inj \Mod(H^0(\cc))\ar[r]^-L\ar@{_(->}[dr] & D(\cc)\ar[d]^{\RHom_{\cc}(H^0(\cc),-)}\\
&
 D(H^0(\cc))
}
\]
were the horizontal map is obtained from \eqref{ref-5.5-64}.
\end{lemmas}
\begin{proof} 
Let $I\in \Inj \Mod(H^0(\cc))$. We have to contruct a natural isomorphism
\begin{equation}
\label{ref-5.8-68}
\RHom_{\cc}(H^0(\cc),L(I))\cong I
\end{equation}
in $D(H^0(\cc))$.
Let $Y\in D(H^0(\cc))$. We have 
\begin{multline*}
\label{eq:yoneda}
\Hom_{H^0(\cc)}(Y,\RHom_{\cc}(H^0(\cc),L(I)))=\Hom_{\cc}(Y,L(I))\\=\Hom_{H^0(\cc)}(H^0(Y),I)=\Hom_{H^0(\cc)}(Y,I)
\end{multline*}
The first equality is ``change of rings'', the second equality is \eqref{ref-5.1-61} and the
last equality is because $I$ is injective. The isomorphism \eqref{ref-5.8-68} now follows
by Yoneda's lemma.
\end{proof}
\subsection{The main technical result}
\begin{propositions} 
\label{ref-5.3.1-69}
Let $\cc$ be a DG-category concentrated in degree $\le 0$, satisfying in addition
\[
H^i(\cc)=0\qquad \text{for $i=-1,\ldots,-m$}
\]
and assume that there is
a full abelian subcategory $\Dscr\subset \Mod(H^0(\cc))$ such that
\begin{itemize}
\item $\Dscr$ has enough injectives and $\gldim \Dscr\le m$.
\item $\Inj\Dscr\subset \Inj\Mod H^0(\cc)$ (and hence $D^b(\Dscr)=D^b_{\Dscr}(H^0(\cc))\subset D^b(H^0(\cc))$, in particular $\Dscr$ is extension closed).
\end{itemize}
Then there is a commutative diagram of additive functors where the top arrow is the functor $L$ introduced in \eqref{ref-5.5-64} (with $\Dscr=D(\cc)$).
\begin{equation}
\label{ref-5.9-70}
\xymatrix{
\Inj \Dscr\ar@{^(->}[r]^-{L}\ar@{^(->}[d]_{
}&D(\cc)\ar@{=}[d]\\
D^b(\Dscr)\ar@{.>}[r]_{\text{exact}}^-{L}\ar@{_(->}@/_2em/[rrr]_{
}& D(\cc)\ar[rr]^-{\RHom_{\cc}(H^0(\cc),-)}&& D(H^0(\cc))
}
\end{equation}
\end{propositions}

\begin{proof}
We will use the notation of \S\ref{sec:construction_}. Let $\Rscr=D(\cc)$, $\Tscr=D(H^0(\cc))$ and let $H:\Rscr\r\Tscr$
be given by $\RHom_{\cc}(H^0(\cc),-)$.
Let 
\begin{align*}
\Ascr &= \{\Sigma^n I | I\in L(\Inj \Dscr),n\leq 0\} \ \ \ \,\subset \Rscr\\
\Bscr &= \{\Sigma^n I | I\in L(\Inj \Dscr),n\geq -m\}  \subset\Rscr
\end{align*}
Then we claim
 that $(\Ascr,\Bscr)$ is a good couple with respect to $H$. In fact by \eqref{ref-5.8-68} we have $\RHom_{\cc}(H^0(\cc),L(I))\cong I$
and for $I,J\in \Inj\Dscr$ we have vanishing $\Hom_{D^b(H^0(\cc))}\allowbreak(\Sigma^n I,\Sigma^{n'} J)=0$ for $n\neq n'$ and moreover
\begin{itemize}
\item $\Hom_{D(\cc)}(\Sigma^n L(I),\Sigma^{n'}L(J))=0$ for $n\leq 0, n'\geq -m, n\neq n'$: in fact
\begin{itemize}
\item[(i)] $\Hom_{D(\cc)}(\Sigma^n L(I),\Sigma^{n'} L(J))=0$ if $n'-n>0$  by \eqref{ref-5.4-63};
\item[(ii)] $\Hom_{D(\cc)}(\Sigma^n L(I),\Sigma^{n'} L(J))=0$ if $-m\leq n'-n<0$ by \eqref{ref-5.7-66} since $H^i(\cc)=0$ for $i=-1,\ldots,-m$.
\end{itemize}
\item $\Hom_{D(\cc)}(\Sigma^n L(I),\Sigma^n L(J))=\Hom_{D^b(H^0(\cc))}(\Sigma^n I,
\Sigma^n J)$ by \eqref{ref-5.4-63}.
\end{itemize}
The rest of the proof follows by observing that the essential image
$\Cscr$ of $\Ascr\cap\Bscr$ under the functor $\RHom_\cc(H^0(\cc),-)$
contains all the objects
$\{I[n] \, | \, I \in \Inj\Dscr, -m \leq n \leq 0\}$. Since
$\Inj \dim \Dscr \leq m$, it is clear that $\Cscr^\ast$ contains all
of the category $\Dscr \subset \Mod–H^0(\cc)$. 
Proposition \ref{prop:digest} now informs us that 
$(\Ascr^\ast,\Bscr^\ast)$ is  an excellent couple
 such that the essential image of $\Ascr^\ast\cap\Bscr^\ast$ contains $\Dscr \subset \Mod H^0(\cc)=\Tscr^{\heartsuit}$.  
Applying Theorem \ref{thm:digest}  we find an exact functor
$G:D^b_{\Dscr}(H^0(\cc))\r D(\cc)$ such that the composition $HG$ is the inclusion
$D^b_{\Dscr}(H^0(\cc))\r D(\cc)$.  We now put $L=G$.
\end{proof}
\begin{comment}
We have used the following lemma.
\begin{lemma} 
\label{lem:extension} Let $\Dscr$ be an abelian subcategory of an abelian category $\Escr$. Let
$\Dscr^\ast$ be the extension closure $\Dscr$ in $\Escr$, i.e.\ the full subcategory
of all objects of $\Escr$ which admit filtrations with subquotients in $\Dscr$. Then
$\Dscr^\ast$ is an abelian subcategory of $\Escr$.
\end{lemma}
\begin{proof}
Let $f:D\r D'$ be a morphism in $\Dscr^\ast$. We have to show that $\ker f$ and $\coker f$ are in $\Dscr^\ast$.
By hypothesis $D$ and $D'$ are equipped with finite ascending filtrations $F$, $F'$ such that $\gr_F D$, $\gr_{F'}D'$ are
in $\Dscr$. We may assume that $F$, $F'$ are $\ZZ$-indexed and by suitably shifting $F'$ we may assume that
$f$ is a filtered map. If we now consider $C^\bullet=(D\xrightarrow{f} D')$ as a 2-term complex then the spectral sequence for filtered complexes
\[
H^\ast(\gr C^\bullet)\Rightarrow H^\ast(C^\bullet)
\]
(whose first page is in $\Dscr$)  yields a filtration on $H^\ast(C^\bullet)$ with subquotients
in $\Dscr$, finishing the proof. 
\end{proof}
\end{comment}
We will need the following technical property of $L$ later.
\begin{corollarys} Let $L$ be as in  the lower row of \eqref{ref-5.9-70}.
If $X\in D^b(\Dscr)$ and $N\in D^b(H^0(\cc)\otimes H^0(\cc)^\circ)$
then 
\begin{equation}
\label{ref-5.11-72}
\RHom_{\cc}(N,{L}(X))\cong \RHom_{H^0(\cc)}(N,X)
\end{equation}
in $D^b(H^0(\cc))$.
\end{corollarys}
\begin{proof}
 By the standard ``change of rings'' identity we have
\[
\RHom_{\cc}(N,{L}(X))=\RHom_{H^0(\cc)}(N,\RHom_{\cc}(H^0(\cc),L(X))).
\]
Using the lower row in \eqref{ref-5.9-70}, we obtain
\[
\RHom_{\cc}(H^0(\cc),L(X))\cong X.
\]
This finishes the proof.
\end{proof}
\section{Deformations}
\label{sec:deformations}
\subsection{$A_\infty$-deformations of linear categories}
\label{ref-6.1-74}
If $\aa$ is a $k$-linear category and $M$ is a $k$-central $\aa$-bimodule, we write $\CC^\bullet(\aa,M)$ for the Hochschild complex of $M$ and $\bar{\CC}^\bullet(\aa,M)$ for its subcomplex of normalized cochains
(i.e.\ those cochains vanishing on identity morphisms). The
inclusion $\bar{\CC}^\bullet(\aa,M)\r \CC^\bullet(\aa,M)$ is a quasi-isomorphism
by \cite[\S1.5.7]{Loday1}. We write $\HH^\bullet(\aa,M)$ for the corresponding
cohomology. Note that we have
\[
\HH^n(\aa,M)=\Ext_{\aa\otimes_k\aa^\circ}^n(\aa,M).
\]

Now let $\eta\in Z^n\bar{\CC}^\bullet(\aa,M)$. Let $\tilde{\aa}$ be the DG-category $\aa\oplus \Sigma^{n-2}M$: its objects are the objects of $\aa$, morphisms are given by $\aa(-,-)\oplus \Sigma^{n-2}M(-,-)$ and composition is coming from the composition in $\aa$ and the action of $\aa$ on $M$. 

We denote by $\aa_\eta$ the $A_\infty$-category 
$\tilde{\aa}$,
with deformed $A_\infty$-structure given by
\[
b_{\aa_\eta}=b_{\tilde{\aa}}+\eta,
\]
where we view $\eta$ as a map of degree one $(\Sigma \aa)^{\otimes n}\r \Sigma (\Sigma^{n-2}M)$ and extend it 
to a map $\eta:(\Sigma \aa_\eta)^{\otimes n}\r \Sigma \aa_\eta$ of degree one by making the unspecified components zero.  Clearly we have
$H^\ast(\aa_\eta)=\tilde{\aa}$.
Furthermore, since $\eta$ is normalized, it is clear that $\aa_\eta$ is strictly
unital. 
\begin{lemmas} If $\eta$, $\eta'$ represent the same element of 
$\HH^n(\aa,M)$ then there is an $A_\infty$-isomorphism $f:\aa_\eta\r\aa_{\eta'}$
whose only non-trivial component is of the form 
\[
f_{n-1}:(\Sigma \aa)^{\otimes n-1}\r \Sigma (\Sigma^{n-2}M)
\]
\end{lemmas}
\begin{proof} This is an easy and standard verification.
\end{proof}
Because the construction of $\aa_{\eta}$ only depends on the cohomology class of $\eta$, we will often write $\aa_{\bar{\eta}}$ with $\bar{\eta}\in \HH^n(\aa,M)$ to denote $\aa_{\eta}$, where $\eta$ is a lift of $\bar{\eta}$ to  $Z^n\bar{\CC}^\bullet(\aa,M)$.

\subsection{Tensoring with a DG-category}
\label{ref-6.2-75}
Let $\aa$, $\bb$ be $k$-linear DG-categories, and let
$M$ be a $k$-central $\aa$-bimodule. Then then there is a morphism of
complexes
\[
\CC(\aa,M)\r \CC(\aa\otimes_k \bb,M\otimes_k \bb):\eta\mapsto \eta\cup 1
\]
where $\eta\cup 1$ is defined by (for suitable composable arrows)
\[
(\eta\cup 1)(a_1\otimes b_1\otimes\cdots \otimes a_n\otimes b_n)=\pm\eta(a_1,\ldots,a_n)\otimes b_1\cdots b_n
\]
with the sign given by the Koszul convention.
It is easy to see that on the level of cohomology $\eta\cup 1$ has the usual interpretation as a map
\[
\Ext^\ast_{\aa\otimes_k \bb^\circ}(\aa,M)\r \Ext^\ast_{\aa\otimes_k \bb\otimes_k \aa^\circ\otimes_k \bb^\circ}(\aa\otimes_k\bb,M\otimes_k\bb)
\]
where $1$ now refers to the identity element of $\Hom_{\bb\otimes_k\bb^\circ}(\bb,\bb)$.

\medskip

It is nontrivial to construct the tensor product of two $A_\infty$-categories. However, no difficulty arises when one of the $A_\infty$-categories is a DG-category: this is a special case of tensoring an 
algebra over an asymmetric operad with a DG-algebra, where again only suitable composable arrows should be considered. Specifically, assume that $\aa$ is a $k$-$A_\infty$-category and $\bb$
is a $k$-DG-category. Then we define 
\begin{align*}
&b^n_{\aa\otimes_k\bb}(s(a_1\otimes b_1),\ldots,s(a_n\otimes b_n))
=\pm b^n_{\aa}(sa_1,\ldots,sa_n)\otimes b_1\cdots b_n \\
&b^1_{\aa\otimes_k\bb}(s(a\otimes b))=b^1(sa)\otimes b + (-1)^{|sa|} sa\otimes d(b)
\end{align*}
(for suitably composable arrows)
where the sign is given by the Koszul convention after making the identification $s(a_i\otimes_k b_i)=(sa_i)\otimes_k b_i$.

With this definition is easy to see that, if $\eta\in Z^n \bar{\CC}(\aa,M)$ and
$\eta\cup 1\in \bar{\CC}(\aa\otimes_k \bb,M\otimes_k \bb)$ is the extended cocycle, then
\[
\aa_\eta\otimes\bb=(\aa\otimes\bb)_{\eta\cup 1}.
\]
\subsection{The characteristic morphism}
\label{ref-6.3-76} Assume again that $\aa$ is a $k$-linear category
and let $N$ be an $\aa$-module.
Then there is a so-called characteristic map \cite{lowen6}
\[
c_N:\HH^n(\aa,M)\r \Ext^n_{\aa}(N,M\Lotimes_\aa N),
\]
which may be constructed by interpreting $\eta\in \HH^n(\aa,M)$ as a map
$\aa\r \Sigma^n M$ in $D(\aa\otimes_k \aa^\circ)$. Applying the functor $-\Lotimes_\aa N$ to  $\eta$ we
obtain a map $N\r \Sigma^nM\Lotimes_\aa N$ which is $c_N(\eta)$.

There is a dual characteristic map
\[
c_N:\HH^n(\aa,M)\r \Ext^n_{\aa}(\RHom_{\aa}(M,N),N),
\]
obtained by applying $\RHom_{\aa}(-,N)$ to $\eta$. For the sequel we note the following obvious fact.
\begin{lemmas} \label{ref-6.3.1-77}
Assume that $M$ is an invertible $\aa$-bimodule. In that case we have a commutative
diagram
\[
\xymatrix{
&  \Ext^n_{\aa}(N,M\Lotimes_\aa N)\ar[dd]_{\cong}^{M^\ast\Lotimes_{\aa}-}\\
\HH^n(\aa,M)\ar[ru]^{c_N(\eta)}\ar[rd]_{c^\ast_N(\eta)}\\
& \Ext^n_{\aa}(\RHom_{\aa}(M,N),N)
}
\]
\end{lemmas}
In other words in the context of Lemma \ref{ref-6.3.1-77} we do not have to make 
a distinction between the two characteristic maps. 

\medskip

Now assume that $M$ is right flat over $\aa$. It is well-known 
that in that case $c_N$ can be constructed directly on the level of complexes.
One starts with the identification
\[
\Ext^n_{\aa}(N,M\otimes_\aa N)=\HH^n(\aa,\underline{\Hom}_k(N,M\otimes_\aa N)).
\]
With this identification $c_N$ is obtained by passing to cohomology
from the map of complexes
\[
c_N:\CC^\bullet(\aa,M)\r \CC^\bullet(\aa,\underline{\Hom}_k(N,M\otimes_\aa N)),
\]
which is  obtained from the obvious map of $\aa$-bimodules 
\begin{equation}
\label{ref-6.1-78}
M\r \underline{\Hom}_k(N,M\otimes_\aa N).
\end{equation}
A similar statement holds for $c^\ast_N$. In this case \eqref{ref-6.1-78} is replaced
by the equally obvious map of $\aa$-bimodules
\[
M\r \underline{\Hom}_k(\Hom(M,N),N).
\]

\subsection{Deformation of objects}
\label{ref-6.4-79}
Let the notation be as in \S\ref{ref-6.1-74}, but to simplify things we
will restrict to the case $n\ge 3$. 

Assume that $M$ is right flat over
$\aa$.  Let $U\in \Mod(\aa)$. A lift of $U$ to $\aa_\eta$ 
is a pair $(V,\phi)$ where $V$ is in $D_\infty(\aa_\eta)$ and $\phi$ is an isomorphism of graded $H^\ast(\aa_\eta)$-modules
$H^\ast(V)\cong H^\ast(\aa_\eta)\otimes_{\aa} U$. 

Similarly, if $M$ is left projective, then a colift  of $U$ to $\aa_\eta$ 
is a pair $(V,\phi)$, where $V$ is in $D_\infty(\aa_\eta)$ and $\phi$ is an isomorphism of graded $H^\ast(\aa_\eta)$-modules
$H^\ast(V)\cong \Hom_{\aa}(H^\ast(\aa_\eta),U)$. 

We recall the following well-known fact. 
\begin{lemmas} \label{ref-6.4.1-80}
The object $U\in \Mod(\aa)$ has a lift to $\aa_\eta$ if and only if $c_U(\bar{\eta})=0$. 
It has a colift if and only if $c^\ast_U(\bar{\eta})=0$. 
\end{lemmas}
\begin{proof}
Both cases are similar, so we will consider the case of a lift. Thus in that case
we assume that $M$ is right flat.
Let $U'$ be the graded $H^\ast(\aa_\eta)$-module $U\oplus \Sigma^{n-2}M\otimes_\aa U$.
If $V$ is an $\aa_\eta$-lifting of $U$ then we may assume that $V$ is represented by a ``minimal model'' object $V=(U',b_V)$ with $b_{V,1}=0$.
We now have a graded functor between graded categories
\[
\aa\oplus\Sigma^{n-2}M\xrightarrow{f}
\Lambda:=\begin{pmatrix}
\End_k(U)&0\\
\Hom_k(U,\Sigma^{n-2}M\otimes_{\aa} U) &\End_k(U)
\end{pmatrix}
\]
representing the action of $H^\ast(\aa_\eta)$ on $U'$, 
and we have to change it to an $A_\infty$-morphism
\[
(\aa\oplus\Sigma^{n-2}M,b_{\aa_\eta})\xrightarrow{f+\xi} 
\begin{pmatrix}
\End_k(U)&0\\
\Hom_k(U,\Sigma^{n-2}M\otimes_{\aa} U) &\End_k(U)
\end{pmatrix}
\]
with $\xi:(\Sigma\aa)^{\otimes n-1}\r \Sigma \Hom_k(U,\Sigma^{n-2}M\otimes_{\aa} U)$ and ${b_{\aa_\eta}}=b_{\aa}+\eta$.
The required compatibility between cofunctors and codifferentials may be expressed
as
\[
(f+\xi)\circ (b_{\aa}+\eta)-b_\Lambda\circ (f+\xi)=0,
\]
with $b_\Lambda$ being the codifferential on $\Lambda$.
As usual we only have to check this after performing the projection $\BB\Lambda\r\Lambda$. So
the only possible non-trivial evaluation is on $\Sigma \aa^{\otimes n}$ and we get
\[
\xi\circ b_{\aa}-b_\Lambda\circ \xi+f\circ \eta=0,
\]
which may be rewritten as 
\[
d_{\text{Hoch}}(\xi)=f\circ\eta=c_U(\eta).
\]
This proves what we want.
\end{proof}

\section{Obstruction theory}
\label{sec:obstructions}
\subsection{Preliminaries on $A_n$-categories and $A_n$-functors}
\label{ref-7.1-81}
$A_n$ categories and functors are defined by replacing $\BB\bb$ 
with the $n$-truncated $(\BB \bb)_{\le n}$
bar-cocategory. I.e.\ $\Ob(\BB\bb)_{\le n}=\Ob(\bb)$ and
\[
\BB\bb_{\le n}(A,B)=\bigoplus_{i\le n, A_1,\ldots,A_{i-1}\in \Ob(\bb)}\Sigma \bb(A_{i-1},B)\otimes \ldots \otimes \Sigma \bb(A,A_1)
\]
As usual, we describe $A_n$-structures and morphisms via their
Taylor coefficients, which may be evaluated
on sequences of $i\le n$ composable maps. 

Given an $A_n$ category $\bb$, we write $\bb_{\le m}$ for the corresponding
category viewed as an $A_m$-category for $m\le n$. A similar convention applies to functors. 

Like for the $A_\infty$-case, all $A_n$-notions will be assumed to be strictly unital.
\subsection{Obstructions for $A_\infty$-morphisms}
\label{ref-7.2-82}
We will use \cite[\S10]{RizzardoVdB} as a convenient reference. The following lemma is a more precise version of
\cite[Lemma 10.3.1]{RizzardoVdB} (see \cite{BKS} for a related result). 
\begin{lemmas} 
\label{ref-7.2.1-83} 
Let $f_i:\cc\r \tilde{\cc}$ be an $A_i$-functor
between $A_\infty$-categories. Then there
is an ``obstruction'' $o_{i+1}(f_i)\in \HH^{i+1}(H^\ast(\cc),H^\ast(\tilde{\cc}))^{-i+1}$
with the following property: $o_{i+1}(f_i)$ vanishes if and only if there exists
$\delta_i:(\Sigma\cc)^{\otimes i}\r \Sigma \tilde{\cc}$ such that
$b_{\tilde{\cc},1}\circ \delta_i-\delta_i \circ b_{\cc,1}=0$ and such that $f_i+\delta_i$
extends to an $A_{i+1}$-functor. The obstruction $o_{i+1}(f_i)$ is natural in the following sense:
for $A_\infty$-functors $G:\cc'\to\cc$, $g:\tilde{\cc}\to \tilde{\cc}'$ we have
\begin{equation}
\label{ref-7.1-84}
o_{i+1}(g\circ f_i\circ G)=H^\ast(g)\circ o_{i+1}(f_i)\circ H^\ast(G).
\end{equation}
\end{lemmas}
\begin{proof}
First note that in \cite{RizzardoVdB} we worked with non-strictly unital $A_\infty$-functors (in fact: morphisms). 
We may however equally well perform the construction in the strictly unital context by 
working with the normalized Hochschild complex. Here we will follow this approach. 

We view $f_i$ as a cofunctor $f_i:\BB\cc_{\le i+1}\r \BB \tilde{\cc}_{\le i+1}$ by making its $i{+}1$'th Taylor coefficient zero.
Define the $f_i$-coderivation
$D:\BB\cc_{\le i+1}\r \BB \tilde{\cc}_{\le i+1}$ as follows:
\[
D=b_{\tilde{\cc}} \circ f_i- f_i \circ b_{\cc}\,.
\]
It is clear that
\begin{equation}\label{ref-7.2-85}
b_{\tilde{\cc}}\circ D+D\circ b_{\cc}=0.
\end{equation}
Moreover, by construction we have that the Taylor coefficients $D_n$ of $D$ satisfy 
\begin{equation}
\label{ref-7.3-86}
D_n=0\qquad  \text{for $n=1,\ldots,i$ }
\end{equation}
so that the only data in $D$ is $D_{i+1}$, which by \eqref{ref-7.2-85}
descends to a linear map $\bar{D}_{i+1}:H^{\ast}(\Sigma\cc)^{\otimes i+1} \r
H^{\ast}(\Sigma\tilde{\cc})$. As in the proof of \cite[Lemma
10.3.1]{RizzardoVdB}, from \eqref{ref-7.2-85} one computes
\begin{equation}
0=d_{\text{Hoch}}(\bar{D}_{i+1})
\end{equation}
where $d_{\text{Hoch}}$ represent the Hochschild differential. Computing degrees one sees that  $\bar{D}_{i+1}$ represents an element of $\HH^{i+1}(H^\ast(\cc),H^\ast(\tilde{\cc}))^{-i+1}$, which we will call $o_{i+1}(f_i)$.

Furthermore - again as in the proof of \cite[Lemma 10.3.1]{RizzardoVdB} - one sees that, if $o_{i+1}(f_i)$ vanishes, then  $f_i$ extends to $f_{i+1}$ in the way described in the statement
of the lemma. 

To show that the implication goes in both directions, let us repeat the argument. The data $f_{i+1}$, $\delta_i$ as in the statement of the lemma will exist if and only the following equation
has a solution in $(f_{i+1})_{i+1}$, $\delta_i$:
\begin{equation}
\label{ref-7.5-87}
\begin{aligned}
0&=(b_{\tilde{\cc}}\circ f_{i+1}-f_{i+1}\circ b_{\cc})_{i+1}\\
&=D_{i+1}+(b_{\tilde{\cc}}\circ (f_{i+1}-f_i)-(f_{i+1}-f_i)\circ b_{\cc})_{i+1}\\
&=D_{i+1}+b_{\tilde{\cc},1}\circ (f_{i+1})_{i+1}-(f_{i+1})_{i+1}\circ b_{\cc,1}\\
&\qquad+b_{\tilde{\cc},2}\circ (\delta_i\otimes f_1+f_1 \otimes \delta_i)-\sum_{a+b+2=i+1} (\delta_i\circ (\Id^{\otimes a}\otimes b_{\cc,2}\otimes\Id^{\otimes b}) 
\end{aligned}
\end{equation}
It is clear that this has a solution if and only if the corresponding equation in cohomology $0=\bar{D}_{i+1}+d_{\text{Hoch}}\bar{\delta}_i=0$ has a solution, i.e.\ if and only the obstruction $o_{i+1}(f_i)$ vanishes.

Naturality: let us write $D(f_i)$ for $D$ as introduced above. Then it
follows from the definition \eqref{ref-7.2-85} that for
$f'_i=g\circ f_i\circ G$
\[
D(f'_i)=g\circ D(f_i)\circ G.
\]
By \eqref{ref-7.3-86} this yields
\[
D(f'_i)_{i+1}=g_0\circ D(f_i)_{i+1}\circ G_0,
\]
and passing to cohomology
\[
\overline{D(f'_i)}_{i+1}=\overline{g}_0\circ \overline{D(f_i)}_{i+1}\circ \overline{G}_0
\]
which is \eqref{ref-7.1-84}.\qed
\def\qed{}\end{proof}

If $f:H^\ast(\cc)\r H^\ast(\tilde{\cc})$ is a graded functor, then it can always
be completed to an $A_2$-functor $f_2:\cc\r \tilde{\cc}$ (this is essentially choosing homotopies). So the
first non-trivial obstruction is 
$o_{3}(f_2)\in \HH^{3}(H^\ast(\cc),H^\ast(\tilde{\cc}))^{-1}$. 
It is easy to see that it only depends on $f$. Indeed, two choices of $f_2$
only differ by a $\delta_2:(\Sigma\cc)^{\otimes 2}\r \Sigma\tilde{\cc}$ commuting with differentials, and this $\delta_2$ disappears in the obstruction. 
See Corollary \ref{ref-7.2.4-92} for
a variation on this fact.

In general, if we start from $f:H^\ast(\cc)\r H^\ast(\tilde{\cc})$ we may compute
obstructions
\[
o_{3}(f_2),o_4(f_3),o_5(f_4),\ldots
\] 
We will informally write these as
\[
o_{3}(f),o_4(f),o_5(f),\ldots
\]
with the proviso that $o_{i+1}(f)$ only exists when $o_3(f),\ldots,o_{i}(f)$
vanish, and furthermore $o_{i+1}(f)$ depends on prior choices. So it may
be zero for one such choice and non-zero for another. 

We will always apply
Lemma \ref{ref-7.2.1-83} with $\cc$ being a $k$-linear category (i.e.\ concentrated
in degree zero). In that case we have
\begin{equation}
\label{ref-7.6-88}
o_{i}(f)\in  \HH^{i}(\cc,H^{-i+2}(\tilde{\cc})).
\end{equation}
\begin{corollarys} 
\label{ref-7.2.2-89}
Consider a commutative diagram of graded functors 
\begin{equation}
\label{ref-7.7-90}
\xymatrix{
H^\ast(\cc_1)\ar[r]^-{f_1}&H^\ast(\tilde{\cc}_1)\ar[d]^{H^\ast(g)}\\
H^\ast(\cc_2)\ar[u]^{H^\ast(G)}\ar[r]_-{f_2}
&H^\ast(\tilde{\cc}_2)
}
\end{equation}
with $g,G$ being  $A_\infty$-quasi-isomorphisms. Let $\{i_1,i_2\}=\{1,2\}$.
When making choices for computing the obstructions for $f_{i_1}$, we may make
corresponding choices for computing the obstructions for $f_{i_2}$ such that
$o_\ast(f_{2})=H^\ast(g)\circ o_\ast(f_{1})\circ H^\ast(G)$.
\end{corollarys}
\begin{proof} When $i_1=1$, $i_2=2$ this is follows from naturality  (see Lemma \ref{ref-7.2.1-83}). 
For the case $i_1=2$, $i_2=1$ we use the fact that $g,G$ have inverses up to homotopy in the $A_\infty$-category.
Such inverses are true inverses on cohomology. Now use again the naturality for 
$A_\infty$-morphisms.
\end{proof}
\begin{remarks}
\label{ref-7.2.3-91}
It follows from Corollary \ref{ref-7.2.2-89} that in order to calculate
obstructions for $f:\cc\r H^\ast(\tilde{\cc})$ with $\cc$ a $k$-linear
category we may replace $\tilde{\cc}$ with a (strictly unital) minimal
model~$\bar{\cc}$.  By definition the underlying complex of $\bar{\cc}$ is
$H^\ast(\tilde{\cc})$, with zero differential and there is an
$A_\infty$-quasi-isomorphism $g:\bar{\cc}\r \tilde{\cc}$ such that
$g_1$ induces the identity on cohomology. Let $f_1:\cc\r H^\ast(\bar{\cc})$
be such that $H^\ast(g_1)\circ f_1=f$. 
By naturality
 we have $o_\ast^{\tilde{\cc}}(f)=o_\ast^{\tilde{\cc}}(H^\ast(g_1)\circ f_1)
=H^\ast(g_1) \circ o_\ast^{\bar{\cc}}(f_1)$. How we use that $H^\ast(g_1)$
is the identity to obtain $o_\ast^{\tilde{\cc}}(f)=o_\ast^{\bar{\cc}}(f_1)$.
\end{remarks}
\begin{corollarys}
\label{ref-7.2.4-92}
Assume that $\cc$ is a $k$-linear category and let $f:\cc\to H^*(\tilde{\cc})$ be a $k$-linear functor.
If $-{n}<0$ is maximal with the property that
$H^{-{n}}(\tilde{\cc})\neq 0$ then 
\begin{equation}
\label{ref-7.8-93}
o_3(f)=\cdots=o_{{n}+1}(f)=0
\end{equation}
and
$o_{{n}+2}(f)$  does not depend
on any choices.
\end{corollarys}
\begin{proof} 
Since by \eqref{ref-7.6-88} $o_j(f)\in \HH^j(\cc,H^{-j+2}(\tilde{\cc}))$ we have $o_j(f)=0$ for $-{n}+1\le -j+2\le -1$
which yields \eqref{ref-7.8-93}.

To prove the statement about $o_{{n}+2}(f)$ we may as in Remark \ref{ref-7.2.3-91} replace $\tilde{\cc}$ with a minimal model $\bar{\cc}$. But then $\bar{\cc}^{-{n}+1}=\cdots =\bar{\cc}^{-1}=0$.
Following the proof of Lemma \ref{ref-7.2.1-83}, we have
\[
f=f_1=f_2=f_3=\cdots=f_{{n}}
\]
To compute $f_{{n}+1}$ we first have to compute 
\[
D(f_{n})_{{n}+1}=(b_{\bar{\cc}}\circ f-f\circ b_{\cc})_{{n}+1}
\]
If ${n}=1$ then this is zero since $f$ respects the multiplication. If ${n}>1$ then using the fact that
the lowest $b_{\bar{\cc},j}$ for $j>2$ which can be non-zero is $b_{\bar{\cc},{n}+2}$
we also get zero
which is compatible with the fact that we already know $o_{{n}+1}(f_{n})=0$.

To compute
the possible lifts of $f_{n}$ we have to solve (see \eqref{ref-7.5-87})
\[
0=[d,(f_{{n}+1})_{{n}+1}]+d_{\text{Hoch}} (\delta_{n})
\]
For degree reasons we must have $\delta_{n}=0$. Since both $\cc$ and $\bar{\cc}$ have zero differential, it follows that we may choose $(f_{{n}+1})_{{n}+1}:(\Sigma \cc)^{\otimes {n}+1}\r \Sigma(\cc)^{-{n}}$ freely.

By definition
$o_{{n}+2}(f)$ is the class of
\[
(b_{\bar{\cc}}\circ f_{{n}+1}-f_{{n}+1}\circ b_\cc)_{{n}+2}=b_{\bar{\cc},{n}+2}\circ (f\otimes\cdots\otimes f)+d_{\operatorname{Hoch}}((f_{{n}+1})_{{n}+1})
\]
One easily checks that  $b_{\bar{\cc},{n}+2}$ is a Hochschild cocycle, and moreover if we replace $\bar{\cc}$ by another $A_\infty$-isomorphic minimal model (see Remark \ref{ref-7.2.3-91})
then $b_{\bar{\cc},n+2}$ changes only by a Hochschild boundary. From this we deduce that $o_{{n}+2}(f)$ is well defined.
\end{proof}
\begin{remarks}
\label{ref-7.2.5-94}
If we write $p_{{n}+2}(\tilde{\cc})\overset{\text{def}}{=}\bar{b}_{\bar{\cc},{n}+2}\in \HH^{{n}+2}(H^0(\cc),H^{-{n}}(\tilde{\cc}))$ then it was shown in the previous proof that $p_{{n}+2}(\tilde{\cc})$ is well defined and moreover
\[
o_{{n}+2}(f)=p_{{n}+2}(\tilde{\cc})\circ f
\]
Note that we may also think of $p_{{n}+2}(\tilde{\cc})$ as $o_{{n}+2}(j)$ where $j:H^0(\tilde{\cc})\r H^\ast(\tilde{\cc})$ is the inclusion.

If $\cc$ is a $k$-linear category, $\eta\in \HH^{{n}+2}(\cc,M)$ for ${n}\ge 1$ and $j:\cc\r H^\ast(\cc_\eta)=\cc\oplus \Sigma^i M$ is the inclusion then we obtain 
$o_3(j)=\cdots=o_{{n}+1}(j)=0$ and $o_{{n}+2}(j)=\eta$.
\end{remarks}
We will use the following application.
\begin{propositions} 
\label{ref-7.2.6-95}
Let $\aa$, $\cc$ be $k$-linear categories, $M$ an $\aa$-bimodule, $\eta\in \HH^n(\aa,M)$. Assume $f:\cc\r \aa$ is an additive
functor such that $f_\ast(\eta)\overset{\text{def}}{=}\eta\circ f=0$. Then there is a commutative diagram of $A_\infty$-categories
\begin{equation}
\label{ref-7.9-96}
\xymatrix{
\aa_\eta\ar[dr]&&\ar@{.>}[ll]_{\tilde{f}} \cc\ar[dl]^{f}\\
&\aa
}
\end{equation}
\end{propositions}
\begin{proof}
According to Corollary \ref{ref-7.2.4-92} and Remark \ref{ref-7.2.5-94} $f_\ast(\eta)$ is the single obstruction against
the existence of the diagram \eqref{ref-7.9-96}. 
\end{proof}

\subsection{Scalar extensions of derived categories}
\label{ref-7.3-97}
Let $\Gamma$ a $k$-algebra. For a $k$-linear category $\Ascr$ we write $\Ascr_\Gamma$ for the category $\Gamma$-objects
in $\Ascr$, i.e.\ pairs $(M,\rho)$ where $M\in \Ob(\mathfrak{\Ascr})$ and $\rho:\Gamma\r \mathfrak{\Ascr}(M,M)$
is a~$k$-algebra morphism. 

There is a forgetful functor \cite{Rizzardo}
\[
F:D(\bb\otimes_k\Gamma)\r D(\bb)_\Gamma
\]
forgetting the action of $\Gamma$ on the level of complexes but remembering it on the level of the derived category. There is a similar result as Lemma \ref{ref-7.2.1-83}:
\begin{lemmas} 
\label{ref-7.3.1-98}
\begin{enumerate}
\item
Let $T\in D(\bb)_\Gamma$.
 Then there is a sequence
of obstructions
\[
o_{i+2}(T)\in \HH^{i+2}(\Gamma,\Ext^{-i}_{\bb}(T,T))
\]
for $i\ge 1$ such that $T$ lifts to an object in $D(\bb\otimes_k\Gamma)$  if and only if all obstructions vanish. More
precisely $o_{i+1}(T)$ is only defined if $o_3(T),\ldots,o_i(T)$
vanish and it depends on choices. 
\item 
If $f:\cc\r \bb$ is a DG-functor and $f_\ast:D(\bb)\r D(\cc)$ is the corresponding change of rings functor then after having made
choices for $T$ we may make
corresponding choices for $f_\ast(T)$ in such a way that
\[
f_\ast(o_{i+2}(T))=o_{i+2}(f_\ast(T))
\]
\end{enumerate}
\end{lemmas}
\begin{proof}
If $T$ has a lift,
then it is represented by a cofibrant object $\tilde{T}$ in $\underline{\Mod}(\bb\otimes_k\Gamma)$ for the standard projective model structure, which is in particular a cofibrant object in $\underline{\Mod}(\bb)$ equipped with an $A_\infty$-$\Gamma$-action. Conversely, 
an object in $\underline{\Mod}(\bb)$ with an  $A_\infty$-$\Gamma$-action may be regarded as an object in $D(\bb\otimes_k\Gamma)$.
See \cite[Lemma 10.2.1]{RizzardoVdB} for an analogous statement which is proved in the same way. 

Let $\tilde{T}$ be a cofibrant object in $\underline{\Mod}(\bb)$ representing $T$. Put $\tilde{\Gamma}=\Hom_\bb(\tilde{T},\tilde{T})$.
We have a graded $k$-algebra morphism $f:\Gamma\r H^\ast(\tilde{\Gamma})$, and the question is when can we lift it to an $A_\infty$-morphism
$\Gamma\r \tilde{\Gamma}$. This is controlled by the obstructions $o_i(T)\overset{\text{def}}{=}o_i(f)$. We still have to show, however, that 
$o_i(T)$ is independent of the choice of $\tilde{T}$. Suppose that $\tilde{T}_1$, $\tilde{T}_2$ are cofibrant objects both representing
$T$, i.e.\ that there is a quasi-isomorphism $u:\tilde{T}_1\r \tilde{T}_2$ inducing the identity on $T$. Put $\tilde{\Gamma}_i=\End_{\bb}(\tilde{T}_i)$
and let $f_i:\Gamma\r H^i(\tilde{\Gamma}_i)$ be the corresponding actions. Let $\tilde{X}=\cone u$. Note that $\tilde{X}$ is cofibrant and acyclic.  Let
 $\tilde{\Gamma}$ be the sub-DG-algebra of $\End_\bb(\tilde{X})$ given by
\[
\tilde{\Gamma}=
\begin{pmatrix}
\tilde{\Gamma}_1 & 0\\
\Hom_{\bb}(\Sigma \tilde{T}_1,\tilde{T}_2) & \tilde{\Gamma}_2
\end{pmatrix}.
\]
The kernels of the projections $\pr_1:\tilde{\Gamma}\r \tilde{\Gamma}_1$, $\pr_2:\tilde{\Gamma}\r \tilde{\Gamma}_2$ are respectively given by $\Hom_{\bb}(\tilde{X},\tilde{T}_2)$
and
$\Hom_{\bb}(\Sigma\tilde{T}_1,\tilde{X})$ and hence they are acyclic. It follows that $\pr_1$, $\pr_2$ are quasi-isomorphisms.

We obtain a morphism of graded rings $f:\Gamma\r H^\ast(\tilde{\Gamma})$ given by
\[
f=
\begin{pmatrix}
f_1&0\\
0&f_2
\end{pmatrix}.
\]
Since $\pr_i:\tilde{\Gamma}\r \tilde{\Gamma}_i$  are quasi-isomorphisms of DG-algebras, and the resulting
 identification of $H^\ast(\tilde{\Gamma}_1)$ and $H^\ast(\tilde{\Gamma}_2)$ is conjugation by $H^\ast(u)$,
we conclude 
by  Corollary \ref{ref-7.2.2-89} that the obstructions against lifting $f$, $f_1$ and $f_2$ all coincide, and those of $f_1$, $f_2$
are naturally identified. This finishes the proof that $o_i(T)$ is well defined. 

\medskip

The verification that the obstructions are natural is similar. Let $\tilde{T}_2\r T$ be a cofibrant replacement of $T$ in $\underline{\Mod}(\bb)$  and
let $u:\tilde{T}_1\r f_\ast\tilde{T}_2$ be a cofibrant replacement of $f_\ast\tilde{T}_2$ in $\underline{\Mod}(\cc)$. Put 
$\tilde{\Gamma}_1=\End_\cc(\tilde{T}_1)$, $\tilde{\Gamma}_2=\End_{\bb}(\tilde{T}_2)$ and consider $\Hom_{\cc}(\Sigma \tilde{T}_1,f_\ast\tilde{T}_2)$
as $\tilde{\Gamma}_2-\tilde{\Gamma}_1$-module. Put
\[
\tilde{\Gamma}=
\begin{pmatrix}
\tilde{\Gamma}_1 & 0\\
\Hom_{\cc}(\Sigma \tilde{T}_1,f_\ast\tilde{T}_2) & \tilde{\Gamma}_2
\end{pmatrix}.
\]
with the differential being the sum of the natural one given by the differentials on $ \tilde{\Gamma}_1$, $ \tilde{\Gamma}_2$ and the commutator with $\left(\begin{smallmatrix} 0&0\\ u&0\end{smallmatrix}\right)$. We now have 
projection maps $p_i:\tilde{\Gamma}\r \tilde{\Gamma}_i$ with $p_2$ being a quasi-isomorphism (since its kernel is given by $\Hom_{\cc}(\Sigma \tilde{T}_1,\cone u$), $\Sigma \tilde{T}_1$
is cofibrant and $\cone u$ is acyclic). Moreover we find that $H^\ast(p_1)\circ H^\ast(p_2)^{-1}:
H^\ast(\tilde{\Gamma}_2)\r H^\ast(\tilde{\Gamma}_1)$ is the map $f_\ast:\Ext_{\bb}^\ast(T,T)\r \Ext_{\cc}^\ast(f_\ast T,f_\ast T)$ given by functoriality. 
The naturally of the obstructions now follows by applying Corollary \ref{ref-7.2.2-89}.
\end{proof}

From Corollary \ref{ref-7.2.4-92} and the proof of Lemma \ref{ref-7.3.1-98} we also deduce:
\begin{corollarys}
\label{ref-7.3.2-99}
Let $T\in D(\bb)_\Gamma$.
  If $-n<0$ is maximal with the property that
  $\Ext^{-n}_\bb(T,T)\neq 0$ then $o_3(T)=\cdots=o_{n+1}(T)=0$,
  and $o_{n+2}(T)$ does not depend on any choices.
\end{corollarys}
\section{Sheaves and presheaves}
\label{sec:sheaves}
\subsection{Introduction}
This technical section is mainly concerned with relating properties of quasi-compact separated schemes $X$ to
similar properties
of the corresponding categories $\Xscr$ defined in the introduction. 
This material is necessary as we will use $\Xscr$ to
deform the derived category of quasi-coherent sheaves on $X$. We discuss:
\begin{itemize}
\item the relation on the level of modules and bimodules (the functors $w$, $W$ in \S\ref{ref-8.2-103});
\item compatibility with Hochschild cohomology (see \eqref{eq:hochschild});
\item  compatibility with certain Fourier-Mukai functors  (Lemma \ref{ref-8.3.1-107}, \eqref{eq:equivariant});
\item compatibility with the characteristic morphism (see \eqref{ref-8.11-117});
\item functoriality of $X\mapsto \Xscr$ for closed immersions (see \S\ref{sec:functoriality});
\item vector bundles (see \S\ref{sec:vectorbundles}).
\end{itemize}
The last section \S\ref{sec:compactgenerators} discusses an auxilliary result which will be needed later. The reader may be willing to skip this section on
first reading.
\subsection{Presheaves}
We discuss some notions introduced in \cite{GS1,lowenvdb3}. We follow
more or less \cite{lowenvdb3}, but as we use left instead of right modules
the conventions will be slightly different.

Let $(I,\le)$ be a poset and let $\Oscr$ be a presheaf of $k$-algebras on $I$.
For $j\le i$, we denote the corresponding 
restriction morphism $\Oscr(i)\r \Oscr(j)$ by $\rho_{ij}$.
In \cite[\S2.2]{lowenvdb3} (following a similar construction in \cite{GS1})
a $k$-linear category, which we will denote by $\tilde{\Oscr}$, 
is associated to $\Oscr$ as follows: $\Ob(\tilde{\Oscr})=I$ and 
\[
\tilde{\Oscr}(i,j)=
\begin{cases} 
\Oscr(j)&\text{if $j\le i$}\\
0&\text{otherwise.}
\end{cases}
\]
The non-trivial compositions for $k\le j\le i$
\[
\tilde{\Oscr}(j,k)\otimes_k \tilde{\Oscr}(i,j)\r \tilde{\Oscr}(i,k)
\]
are given by
\[
\Oscr(k)\otimes_k \Oscr(j)\xrightarrow{\text{restriction}} \Oscr(k)
\otimes_k \Oscr(k)\xrightarrow{\text{multiplication}} \Oscr(k).
\]

Let $\Mod(\Oscr)$ be the category of presheaves of $\Oscr$-modules. 
There is an equivalence of categories
\[
\pi^\ast:\Mod(\Oscr)\r \Mod(\tilde{\Oscr})
\]
such that
\[
\pi^\ast(M)(i)=M(i),
\]
and the multiplication map for $j\le i$ 
\[
\tilde{\Oscr}(i,j)\otimes_k (\pi^\ast M)(i)\r (\pi^\ast M)(j)
\]
is given by
\begin{equation}
\label{ref-8.1-100}
\Oscr(j)\otimes_k M(i)\xrightarrow{\text{restriction}} \Oscr(j)\otimes_k M(j)\xrightarrow{\text{action}} M(j).
\end{equation}
Conversely, we may recover the restriction map $\rho_{ij}:M(i)\r M(j)$ as the action of the element
$1_{\Oscr(j)}\in \tilde{\Oscr}(i,j)$.

Assume that $\Oscr'$ is a second presheaf of rings on $I$. Write
$\Bimod_k(\Oscr,\Oscr')=\Mod(\Oscr\otimes_k \Oscr^{\prime \circ})$. 
Then
there is a functor (see \cite{GS1}\cite[\S3.4, Lemma 5.2]{lowenvdb3})
\[
\Pi^\ast:\Bimod_k(\Oscr,\Oscr')\r \Bimod(\tilde{\Oscr},\tilde{\Oscr}')
\]
defined as follows
\[
\Pi^\ast(M)(i,j)=
\begin{cases}
M(j)&\text{if $j\le i$}\\
0&\text{otherwise.}
\end{cases}
\]
The bimodule structure on $\Pi^\ast(M)(i,j)$ is defined using a similar formula as \eqref{ref-8.1-100}.

The functor $\Pi^\ast$ is obviously exact and by  \cite{GS1}\cite[Thm 4.1, Lemma 5.2]{lowenvdb3} the corresponding
derived functor
\[
\Pi^\ast:D(\Bimod_k(\Oscr,\Oscr'))\r D(\Bimod(\tilde{\Oscr},\tilde{\Oscr}'))
\] 
is fully
faithful.
\begin{lemmas}
\label{ref-8.1.1-101}
Let $U\in D(\Oscr')$ and $M\in D(\Bimod_k(\Oscr,\Oscr'))$. Then there is a natural isomorphism
\begin{equation}
\label{ref-8.2-102}
\pi^\ast(M\Lotimes_{\Oscr'} U)=\Pi^\ast(M)\Lotimes_{\tilde{\Oscr}'} \pi^\ast U
\end{equation}
as objects in $D(\tilde{\Oscr})$.
\end{lemmas}
\begin{proof} It suffices to prove the non-derived statement for $U$ a projective object in $\Mod(\Oscr')$,
as $\pi^\ast U$ is then also projective in $\Mod(\tilde{\Oscr}')$ since
$\pi^\ast$ is an equivalence of categories. Equivalently, we may assume
 $U=\pi_\ast V$, with $V$ a projective object in $\Mod(\tilde{\Oscr}')$ where $\pi_\ast=(\pi^\ast)^{-1}$.
Furthermore, we may assume that $V$ is of the form $\tilde{\Oscr}'(i,-)$.

Then we find 
\begin{align*}
\pi^\ast(M\otimes_{\Oscr'} \pi_\ast \tilde{\Oscr}'(i,-))(j)&=M(j)\otimes_{\Oscr'(j)}  \tilde{\Oscr}'(i,j)\\
&=
\begin{cases}
M(j)\otimes_{\Oscr'(j)}  \Oscr'(j)=M(j)&\text{if $j\le i$}\\
0&\text{otherwise.}
\end{cases}
\end{align*}
Similarly, by the general property of tensor products
\[
\Pi^\ast(M)\otimes_{\tilde{\Oscr}'}\tilde{\Oscr}'(i,-)=\Pi^\ast(M)(i,-),
\]
and hence, by the definition of $\Pi^\ast(-)$,
\[
(\Pi^\ast(M)\otimes_{\tilde{\Oscr}'}\tilde{\Oscr}'(i,-))(j)
=
\begin{cases}
M(j)&\text{if $j\le i$}\\
0&\text{otherwise.}
\end{cases}
\]
In other words, for all $i,j$
\[
\pi^\ast(M\otimes_{\Oscr'} \pi_\ast \tilde{\Oscr}'(i,-))(j)=(\Pi^\ast(M)\otimes_{\tilde{\Oscr}'}\tilde{\Oscr}'(i,-))(j)
\]
It remains to show that this identification is natural in $i$, $j$. This is a routine verification, which we omit. 
\end{proof}
\subsection{Sheaves}
\label{ref-8.2-103}
Here we recall some results from \cite[\S7.5ff]{lowenvdb2}.  Unless
otherwise specified, in the rest of this section $X$ will be a
quasi-compact separated $k$-scheme. The separatedness hypothesis
ensures that $D(\Qch(X))\cong D_{\Qch}(\Mod(X))$ \cite{Neeman}.  Hence
we will make no distinction between those two categories. Note, in
particular, that $D(\Qch(X))$ has enough homotopically flat objects
(see \cite{AJL1}), so the derived tensor product may be computed
entirely on the quasi-coherent level.

\medskip

Let $X=\bigcup_{i=1}^n
U_i$ be an affine covering. For $I\subset \{1,\ldots,n\}$ put
$U_I=\bigcap_{i\in I} U_i$. 
Let $\Iscr$ be the poset $\{I\subset\{1,\ldots,n\}\mid I\neq \emptyset\}$,
ordered in such a way that $I\le J$ if $J\subset I$ (the strange ordering
is motivated by the fact that $J\subset I$ implies $U_I\subset U_J$). 

Let $\widehat{\Oscr}_X$ be the presheaf of rings on $\Iscr$ associated to $\Oscr_X$,
and, for a quasi-coherent sheaf $M$ on $X$, let $\epsilon^\ast M$ be the
corresponding presheaf of $\widehat{\Oscr}_X$-modules.
The corresponding derived functor 
\[
\epsilon^\ast:D(\Qch(X))\r D(\Mod(\widehat{\Oscr}_X))
\]
has a right adjoint \cite{lowenvdb2}, which we will denote by $R\epsilon_\ast$. It may be computed using
a version of the \v{C}ech complex. More precisely
\begin{equation}
\label{ref-8.3-104}
R\epsilon_\ast(M)=\left(\bigoplus_{I\in \Iscr} j_{U_I,\ast} \Sigma^{-|I|+1}\widetilde{M(I)}, d \right),
\end{equation}
where $\tilde{?}$ is the quasi-coherent sheaf associated to a module over a commutative ring, and $j_{U_I}:U_I\r X$ is the inclusion map.
The differential $d$ is the usual alternating sum of restriction morphisms. Recall the following\footnote{This is stated
in somewhat greater generality than in loc.\ cit. However, it can be proved in the same way.}:
\begin{lemmas} \cite[Theorem 7.6.6]{lowenvdb2} \label{ref-8.2.1-105} The functor $\epsilon^\ast:D(\Qch(X))\r D(\widehat{\Oscr}_X)$ is fully faithful
and a left inverse is given by $R\epsilon_\ast$. Furthermore, the essential image of $\epsilon^\ast$ is $D_{\epsilon^*\Qch(X)}(\widehat{\Oscr}_X)$.
\end{lemmas}
Below, if $X$ is a quasi-compact separated scheme, we will denote by the corresponding
curly letter $\Xscr$ the category $\widetilde{\widehat{\Oscr}}_X$, as introduced in this section and the previous one. 
There is now a fully faithful embedding 
\begin{equation}
\label{eq:fullyfaithful}
w:D(\Qch(X))\r D(\Xscr),
\end{equation}
given by the composition
\[
D(\Qch(X))\xrightarrow{\epsilon^\ast} D(\Mod(\widehat{\Oscr}_X))\xrightarrow[\cong]{\pi^\ast} D(\Xscr).
\]
We have a similar statement for bimodules. 
Let $D^{\delta}(\Qch(X))$ be the category
whose objects are the same as those of $D(\Qch(X))$, but whose
$\Hom$-sets are given by
\[
\Hom_{D^{\delta}(\Qch(X))}(M,N)=\Hom_{D(\Qch(X\times_k X))}(i_{\Delta,\ast} M,i_{\Delta,\ast} N),
\]
where $i_{\Delta}:X\r X\times_k X$ is the diagonal. Let $Z=\bigcup_{i=1}^n U_i\times_k U_i$. Note that $\widehat{\Oscr}_Z=\widehat{\Oscr}_X\otimes_k\widehat{\Oscr}_X$.
Then the map 
\[
D^\delta(\Qch(X))\xrightarrow{M\mapsto i_{\Delta,\ast} M{|} Z} D(\Qch(Z))
\]
is fully faithful, since the support of $i_\ast M$ is closed in $X\times_k X$ and contained in the open set $Z$. By the
above discussion, we obtain a fully faithful embedding 
\[
W:D^\delta(\Qch(X))\r D(\Xscr\otimes_k \Xscr^\circ)
\]
given as the composition
\[
D^\delta(\Qch(X))\xrightarrow{M\mapsto i_\ast M{|} Z} D(\Qch(Z))
\xrightarrow{\epsilon^\ast} D(\widehat{\Oscr}_X\otimes_k \widehat{\Oscr}_X)
\xrightarrow{\Pi^\ast} D(\Xscr\otimes_k\Xscr^\circ).
\]
If $M$ is a quasi-coherent $\Oscr_X$-module, then following \cite{Swan}
its Hochschild cohomology is defined as
\[
\HH^\ast(X,M)\overset{\text{def}}{=} \Ext^\ast_{X\times_k X}(i_{\Delta,\ast}\Oscr_X,i_{\Delta,\ast}M).
\]
Hence by the full faithfulness of $W$ we have a canonical isomorphism \cite{lowenvdb2}
\begin{equation}
\label{eq:hochschild}
\HH^\ast(X,M)\cong \HH^\ast(\Xscr,W(M)).
\end{equation}
\subsection{Actions of bimodules on modules}
Consider the following  bifunctor, 
\begin{equation}
\label{ref-8.4-106}
\Fscr:D^\delta(\Qch(X))\times D(\Qch(X))\r D(\Qch(X)):(M,U)\mapsto R\pr_{1\ast}(i_{\Delta,\ast}M\Lotimes_{\Oscr_{X\times X}} \pr_2^\ast U)
\end{equation}
\begin{lemmas}
\label{ref-8.3.1-107}
The following diagram is commutative:
\[
\xymatrix@C=1pt{
D^\delta(\Qch(X))\ar[d]_W&\times & D(\Qch(X))\ar[d]^w\ar[rrrrr]^{\Fscr} &&&&& D(\Qch(X))\ar[d]^w\\
D(\Xscr\otimes_k \Xscr^\circ)&\times & D(\Xscr)\ar[rrrrr]_{-\Lotimes_{\Xscr}-} &&&&& D(\Xscr)
}
\]
\end{lemmas}
\begin{proof}  
Let $\alpha,\beta:Z\r X$
be the first and the second projection respectively. Then we have
\[
R\pr_{1\ast}(i_{\Delta,\ast}M\Lotimes_{\Oscr_{X\times X}} \pr_2^\ast U)=R\alpha_\ast(i_{\Delta,\ast}M|_{Z}\Lotimes_{\Oscr_Z} \beta^\ast U).
\]
So, by the definition of $w$ and $W$, we have to prove that for $N=i_{\Delta,\ast}M\in D(\Qch(Z))$ there is an isomorphism
\[
\pi^\ast \epsilon^\ast (R\alpha_\ast(N\Lotimes_{\Oscr_Z} \beta^\ast U))=\Pi^\ast \epsilon^\ast N\Lotimes_{\Xscr} \pi^\ast \epsilon^\ast U,
\]
which is natural in $N$ considered
as an object in $D(\Qch(Z))$.

It follows from \eqref{ref-8.9-114} and Lemma \ref{ref-8.3.5-115} below that for any
$N\in D(\Qch(Z))$ there is a canonical
morphism
\[
\epsilon^\ast R\alpha_\ast(N\Lotimes_{\Oscr_Z}\beta^\ast U)
\r
\epsilon^\ast N\Lotimes_{\widehat{\Oscr}_X} \epsilon^\ast U,
\]
which is moreover an isomorphism if $N=i_{\Delta,\ast}M$.

Applying $\pi^\ast$ and using Lemma \ref{ref-8.1.1-101} we get a canonical
morphism
\[
\pi^\ast\epsilon^\ast R\alpha_\ast(N\Lotimes_{\Oscr_Z}\beta^\ast U)
\r
\Pi^\ast\epsilon^\ast N\Lotimes_{\Xscr} \pi^\ast\epsilon^\ast U
\]
having the same property.
This finishes the proof.
\end{proof}
We now give the lemmas on which the previous proof was based.
\begin{lemmas} Let $N,U\in D(\Qch(X))$. Then 
\begin{equation}
\label{ref-8.5-108}
\epsilon^\ast(N\Lotimes_{\Oscr_X}U)\cong \epsilon^\ast N
\Lotimes_{\widehat{\Oscr}_X} \epsilon^\ast U
\end{equation}
\end{lemmas}
\begin{proof} We may assume that $U$ is homotopically flat \cite{AJL1}
  and it is easy to see that this implies that $\epsilon^\ast U$ is also
  homotopically flat.  Hence we have to prove \eqref{ref-8.5-108} for
  quasi-coherent sheaves, which is obvious.
\end{proof}
\begin{lemmas} \label{ref-8.3.3-109}
Let $P$ be in $D(\Qch(Z))$. Then 
\begin{equation}
\label{ref-8.6-110}
R\alpha_\ast P=R\epsilon^{\mathrm{left}}_\ast(\epsilon^\ast P),
\end{equation}
where $\epsilon^{\mathrm{left}}$ refers to the fact that we only consider 
the left $\widehat{\Oscr}_X$-structure on $\epsilon^\ast P$.
\end{lemmas}
\begin{proof}  Since $P$ is a complex of  quasi-coherent sheaves
$P$ is quasi-isomorphic to its \v{C}ech complex. In other words it is isomorphic to
\begin{equation}
\label{ref-8.7-111}
\left(\bigoplus_{I\in \Iscr} j_{U_I\times_k U_I,\ast} \Sigma^{-|I|+1}P(U_I\times_k U_I)\,\tilde{}, d \right)
\end{equation}
\eqref{ref-8.7-111}  consists of modules which
are acyclic for $\alpha_\ast$. In fact, $j_{U_I\times_k U_I}$ and $\alpha\circ j_{U_I\times_k U_I}$ are affine and hence have no higher direct images for quasi-coherent sheaves. It follows by the Leray spectral sequence that the same is true for $\alpha$. Moreover, $\alpha_\ast$ has finite cohomological dimension. Hence we have 
\begin{align*}
  R\alpha_\ast P&=\left(\bigoplus_{I\in \Iscr} \alpha_\ast j_{U_I\times_k U_I,\ast} 
\Sigma^{-|I|+1}P(U_I\times_k U_I)\,\tilde{}, d \right)\\
  &=\left(\bigoplus_{I\in \Iscr} j_{U_I\ast} \pr_{U_I\times_k
      U_I,U_I\ast} \Sigma^{-|I|+1}P(U_I\times_k U_I)\,\tilde{}, d \right)
\end{align*}
where $\pr_{U_I\times_k U_I,U_I\ast}$ is the projection map. This is precisely $R\epsilon_\ast^{\text{left}}$ applied to the 
presheaf on $\Iscr$ given by $I\mapsto P(U_I\times_k U_I)$, and the latter
is of course $\epsilon^\ast P$.
\end{proof}\begin{lemmas}
\label{ref-8.3.4-112}
Let $N$ be in $D(\Qch(Z))$ and $U\in D(\Qch(X))$. Then we have
\begin{equation}
\label{ref-8.8-113}
R\alpha_\ast(N\Lotimes_{\Oscr_Z}\beta^\ast U)=
R\epsilon_\ast(\epsilon^\ast N\Lotimes_{\widehat{\Oscr}_X} \epsilon^\ast U)
\end{equation}
\end{lemmas}
\begin{proof}
We have
\begin{align*}
  R\alpha_\ast(N\Lotimes_{\Oscr_Z}\beta^\ast U)&
=R\epsilon^{\text{left}}_\ast(\epsilon^\ast(N\Lotimes_{\Oscr_Z}\beta^\ast U))\\
  &=R\epsilon_\ast(\epsilon^\ast N\Lotimes_{\widehat{\Oscr}_Z}\epsilon^\ast\beta^\ast U)\\
  &=R\epsilon_\ast(\epsilon^\ast N\Lotimes_{\widehat{\Oscr}_Z}
(\widehat{\Oscr}_Z\otimes_{\widehat{\Oscr}_X} \epsilon^\ast U))\\
  &=R\epsilon_\ast(\epsilon^\ast N\Lotimes_{\widehat{\Oscr}_X}
  \epsilon^\ast U)
\end{align*}
where in the first equality we use \eqref{ref-8.6-110}, and in the second equality we use \eqref{ref-8.5-108}.
\end{proof}
By adjointness, we obtain from \eqref{ref-8.8-113} a canonical morphism
\begin{equation}
\label{ref-8.9-114}
\epsilon^\ast R\alpha_\ast(N\Lotimes_{\Oscr_Z}\beta^\ast U)
\xrightarrow{\cong} \epsilon^\ast R\epsilon_\ast (\epsilon^\ast N\Lotimes_{\widehat{\Oscr}_X} \epsilon^\ast U)
\xrightarrow{\text{counit}} \epsilon^\ast N\Lotimes_{\widehat{\Oscr}_X} \epsilon^\ast U.
\end{equation}
\begin{lemmas}
\label{ref-8.3.5-115} \eqref{ref-8.9-114} is an isomorphism if $N$ is of the form
$i_{\Delta,\ast} M$.
\end{lemmas}
\begin{proof}
In that case, one easily checks that
\[
\epsilon^\ast i_{\Delta,\ast} M\Lotimes_{\widehat{\Oscr}_X} \epsilon^\ast U
=\epsilon^\ast(M\Lotimes_{\Oscr_X} U).
\]
Since $M\Lotimes_{\Oscr_X} U$ is quasi-coherent, we obtain that the counit morphism
in \eqref{ref-8.9-114} is an isomorphism.
\end{proof}
\subsection{Equivariant version}
Assume now that $\Gamma$ is a $k$-algebra (non necessarily commutative). Let $\Qch(X)_\Gamma$ be the category
of quasi-coherent sheaves on $X$ equipped with a left $\Gamma$-action. 
Let $\Oscr_{X,\Gamma}=\Oscr_X\otimes_k \Gamma$ so that $\Qch(X)_\Gamma\cong
\Qch(\Oscr_{X,\Gamma})$. Furthermore put $\Xscr_\Gamma=\Xscr\otimes_k\Gamma$.
Then $\Xscr_\Gamma$ is obtained from $\Oscr_{X,\Gamma}$ in the same way
as $\Xscr$ is obtained from $\Oscr_X$.
Then using \cite[Theorem 7.6.6]{lowenvdb2} 
we obtain as in \eqref{eq:fullyfaithful} a full faithful embedding
\begin{equation}
\label{ref-8.10-116}
w:D(\Qch(X)_\Gamma)\r D(\Xscr_\Gamma)
\end{equation}
and furthermore we have a  commutative diagram with the same proof as Lemma \ref{ref-8.3.1-107}
\begin{equation}
\label{eq:equivariant}
\xymatrix@C=1pt{
D^\delta(\Qch(X))\ar[d]_W&\times & D(\Qch(X)_\Gamma)\ar[d]^w\ar[rrrrr]^{\Fscr} &&&&& D(\Qch(X)_\Gamma)\ar[d]^w\\
D(\Xscr\otimes_k \Xscr^\circ)&\times & D(\Xscr_\Gamma)\ar[rrrrr]_{-\Lotimes_{\Xscr}-} &&&&& D(\Xscr_\Gamma)
}
\end{equation}
\subsection{The characteristic morphism}
In \S\ref{ref-6.3-76} we introduced the characteristic morphism for
DG-categories (following \cite{lowen6}). A similar definition works
for schemes. We present a restricted version which is sufficient
for our applications. Let $X$ be as above and let 
$M,U\in D(\Qch(X))$. Then the characteristic morphism
\[
c_U:\HH^\ast(X,M)\r \Ext^\ast_X(U,M\Lotimes_X U)
\]
is defined as follows. Let $\eta\in \HH^n(X,M)$, and view it as a map
$\Oscr_X\r \Sigma^n M$ in the category $D^\delta(\Qch(X))$. Then
\[
c_U(\eta)=\Fscr(\eta,\Id),
\]
with $\Fscr$ as in \eqref{ref-8.4-106}. From Lemma \ref{ref-8.3.1-107} we immediately
obtain the following commutative diagram:
\begin{equation}
\label{ref-8.11-117}
\xymatrix{
\HH^\ast(X,M)\ar[d]_W^{\cong}\ar[r]^{c_U} & \Ext^\ast_X(U,M\Lotimes_X U)\ar[d]_{\cong}^{w}\\
\HH^\ast(\Xscr,W(M))\ar[r]_-{c_{w(U)}} & \Ext^\ast_\Xscr(w(U),W(M)\Lotimes_\Xscr w(U))
}
\end{equation} 
(the vertical maps are isomorphism because of \eqref{eq:hochschild}, 
Lemma \ref{ref-8.3.1-107}, and the fact that $w$ is fully faithful (see \S\ref{ref-8.2-103}).
If $U$ is an object in $D(\Qch(X)_\Gamma)$, then there is a characteristic map
\begin{equation}
\label{ref-8.12-118}
c_{U,\Gamma}:\HH^\ast(X,M)\r \Ext^\ast_{\Qch(X)_\Gamma}(U,M\Lotimes_X U)
\end{equation}
which fits in a similar commutative diagram as \eqref{ref-8.11-117}.

\subsection{Functoriality}
\label{sec:functoriality}
Now assume that $X,Y$ are quasi-compact separated $k$-schemes, and let $f:X\r Y$ be a closed immersion.
Let $Y=\bigcup_i^n V_i$ be an affine covering, and let $U_i=f^{-1}(V_i)$ be the induced covering on $X$.

The map $f$ induces a dual functor
\[
f:\Yscr\r \Xscr
\]
and hence a ``change of rings'' functor
\[
f_\ast:D(\Xscr)\r D(\Yscr).
\]
\begin{lemmas} \label{ref-8.6.1-119}
The following diagram is commutative:
\[
\xymatrix{
D(\Qch(X))\ar[r]^{f_\ast}\ar[d]_w& D(\Qch(Y))\ar[d]^w\\
D(\Xscr)\ar[r]_{f_\ast}& D(\Yscr)
}
\]
\end{lemmas}
\begin{proof}
All functors are induced from exact functors on the level of abelian categories.
Hence it suffices to check the commutativity on the level of sheaves, which is obvious. 
\end{proof}
The functor $(f,f)_\ast:D(\Qch(X\times_k X))\r D(\Qch(Y\times_k Y))$ descends to a functor
\[
f_\ast :D^\delta(\Qch(X))\r D^\delta(\Qch(Y)).
\]
\begin{lemmas} The following diagram is commutative:
\begin{equation}
\label{ref-8.13-120}
\xymatrix{
D^\delta(\Qch(X))\ar[d]_W \ar[r]^{f_\ast} &D^\delta(\Qch(Y))\ar[d]^W\\
D(\Xscr\otimes_k\Xscr^\circ)\ar[r]_{(f,f)_\ast}& D(\Yscr\otimes_k \Yscr^\circ)
}
\end{equation}
\end{lemmas}
\begin{proof}
All functors are induced from exact functors on the level of abelian categories.
Hence it suffices to check the commutativity on the level of sheaves, which is obvious. 
\end{proof} Applying \eqref{ref-8.13-120} to morphisms $\Oscr_X\r \Sigma^n M$ in $D^\delta(\Qch(X))$, we get a commutative diagram
\[
\xymatrix{
D^\delta(\Qch(X))(\Oscr_X,\Sigma^n M)\ar[d]_{W}^\cong\ar[r]^-{f_\ast}
&
D^\delta(\Qch(Y))(f_\ast \Oscr_X,\Sigma^n f_\ast M)\ar[d]^W_{\cong}\ar[r]&D^\delta(\Qch(Y))(\Oscr_Y,\Sigma^n f_\ast M)\ar[d]^W_{\cong}
\\
\Ext^n_{\Xscr\otimes_k \Xscr^\circ}(\Xscr,W(M))\ar[r]_-{(f,f)_\ast}
&
\Ext^n_{\Yscr\otimes_k \Yscr^\circ}(\Xscr,(f,f)_\ast(W( M)))\ar[r]&\Ext^n_{\Yscr\otimes_k \Yscr^\circ}(\Yscr,(f,f)_\ast(W( M)))
}
\]
where the rightmost square is obtained by precomposing with $\Oscr_Y\r f_\ast\Oscr_X$ and $\Yscr\r \Xscr$
respectively.

We obtain a commutative diagram on the level of Hochschild cohomology
\begin{equation}
\label{ref-8.14-121}
\xymatrix{
\HH^\ast(X,M)\ar[r]^{f_\ast}\ar[d]_W^{\cong} & \HH^\ast(Y,f_\ast M)\ar[d]^W_{\cong}\\
\HH^\ast(\Xscr,W(M))\ar[r]_-{(f,f)_\ast} & \HH^\ast(\Yscr, (f,f)_\ast(W(M)))
}
\end{equation}
\subsection{Vector bundles and projectives}
\label{sec:vectorbundles}
\begin{lemmas}
\label{ref-8.7.1-122}
\begin{enumerate}
\item Let $M,N$ we quasi-coherent sheaves on $X$ and put $\Mscr=W(M)$, $\Nscr=W(N)$. Then
$
W(M\otimes_{\Oscr_X} N)=\Mscr\otimes_{\Xscr} \Nscr
$.
\item Assume that $M$ is a vector bundle on $X$. Then $\Mscr$ is projective on the left and on the right. That is, for every $I\in \Iscr$ we have that
$\Mscr(I,-)$ and $\Mscr(-,I)$ are respectively projective left and right $\Xscr$-modules. 
\end{enumerate}
\end{lemmas}
\begin{proof}
\begin{enumerate}
\item
It is an immediate verification that 
\begin{equation}
\label{ref-8.15-123}
\begin{aligned}
\Mscr(I,-)&=\Xscr(I,-)\otimes_{\Oscr(U_I)} M(U_I)\\
\Mscr(-,I)&=M(U_I)\otimes_{\Oscr(U_I)} \Xscr(-,I)
\end{aligned}
\end{equation}
We compute
\begin{align*}
(\Mscr\otimes_{\Xscr} \Nscr)(I_1,I_2)&=\Mscr(-,I_2)\otimes_{\Xscr} \Nscr(I_1,-)\\
&=M(U_{I_2})\otimes_{\Oscr(U_{I_2})}\Xscr(-,I_2)\otimes_{\Xscr} \Xscr(I_1,-)\otimes_{\Oscr(U_{I_1})} N(U_{I_1})\\
&=M(U_{I_2})\otimes_{\Oscr(U_{I_2})}\Xscr(I_1,I_2)\otimes_{\Oscr(U_{I_1})} N(U_{I_1}).
\end{align*}
Assume now $I_2\subset I_1$ (for otherwise there is nothing to prove). Then we have
\begin{align*}
M(U_{I_2})\otimes_{\Oscr(U_{I_2})}\Xscr(I_1,I_2)\otimes_{\Oscr(U_{I_1})} N(U_{I_1})
&=M(U_{I_2})\otimes_{\Oscr(U_{I_2})}\Oscr_X(U_{I_2})\otimes_{\Oscr(U_{I_1})} N(U_{I_1})\\
&=M(U_{I_2})\otimes_{\Oscr(U_{I_2})} N(U_{I_2})\\
&=(M\otimes_{\Oscr_X} N)(U_{I_2})\\
&=W(M\otimes_{\Oscr_X} N)(I_1,I_2)
\end{align*}
\item
Now $M(U_I)$ is a finitely generated "projective $\Oscr(U_I)$-module, and hence a
summand of a free module. By \eqref{ref-8.15-123} this implies that $\Mscr(I,-)$ is a summand of
 $\Xscr(I,-)^{\oplus n}$ for some $n$, and similarly for
$\Mscr(-,I)$. This means both are projective.
\end{enumerate}
\def\qed{}\end{proof}
\subsection{Compact generators}
\label{sec:compactgenerators}
For a perfect complex $P$ in $D(\Qch(X))$, put $P^D=\uRHom_X(P,\Oscr_X)$. 
Recall the following
\begin{lemmas} \label{ref-8.8.1-124} Let $P$ be perfect object in  
$D(\Qch(X))$. Then $P$ generates $D(\Qch(X))$ if and only if $P^D$ generates
$D(\Qch(X))$.
\end{lemmas}
\begin{proof} By \cite{Neeman1,Neeman3}  $P$ generates $D(\Qch(X))$
if and only if it classically generates the category $\Perf(X)$ of perfect complexes
in $D(\Qch(X))$. The fact that $(-)^D$ is a duality on $\Perf(X)$ proves what we want. 
\end{proof}
\begin{propositions} 
\label{ref-8.8.2-125} Let $T\in \Qch(X)$ be a tilting bundle, i.e.\ a
  vector bundle generating $D(\Qch(X))$ such that $\Ext^i_X(T,T)=0$ for $i>0$. Set $\Gamma=\End_X(T)$. Then
\[
c_{T,\Gamma}:\HH^\ast(X,M)\r \Ext^\ast_{\Qch(X)_\Gamma}(T,M\Lotimes_X T)
\]
is an isomorphism.
\end{propositions}
\begin{proof} We claim that the functor
\[
H:D(\Qch(X\times_k X))\r D(\Qch(X)_\Gamma): N\mapsto R\pr_{1\ast}(N\Lotimes_{X\times X}\pr_2^\ast T)
\]
is an equivalence of categories. This implies what we want. 

By Lemma \ref{ref-8.8.1-124} and \cite[\S3.4]{BondalVdb},
$T\boxtimes T^D$ is a compact generator for $D(\Qch(X\times_k X))$ and it is also clear that $T\otimes_k \Gamma$ is a compact generator for $D(\Qch(X)_\Gamma)$. We compute
\begin{align*}
H(T\boxtimes T^D)&=R\pr_{1\ast}((T\boxtimes T^D)\Lotimes_{X\times X}\pr_2^\ast T)\\
&=R\pr_{1\ast}(T\boxtimes \uRHom_X(T,T))\\
&=T\otimes_k\Gamma
\end{align*}
and is clear that in this way $H$ yields an isomorphism between
\[
\REnd_{X\times_k X}(T\boxtimes T^D)=\Gamma\otimes_k \Gamma^\circ
\]
and
\[
\REnd_{\Qch(X)_\Gamma}(T\otimes_k\Gamma)=\Gamma\otimes_k \Gamma^\circ
\]
This implies that $H$ is an equivalence in the usual way.
\end{proof}
\section{Some properties of divisors}
\label{sec:divisors}
\subsection{Preliminaries}
Let $X$ be a quasi-compact separated scheme. 
If $\Ascr$, $\Bscr$ are sheaves of $k$-algebras on $X$ then an
$\Ascr-\Bscr$-bimodule $\Fscr$ is defined to be a sheaf of
$\Ascr\otimes_k \Bscr^\circ$-modules. Note that even if $\Ascr$,
$\Bscr$ are quasi-coherent this will usually not be the case for
$\Ascr\otimes_k \Bscr^\circ$.  To compute things like
$\Fscr\Lotimes_\Bscr-$ we may take a flat resolution of $\Fscr$ as sheaf of
$\Ascr\otimes_k \Bscr^\circ$-modules.  This is then automatically also
a flat resolution as right $\Bscr$-modules which can be used to
compute the derived tensor product.

\medskip

Define $D^\delta(\Oscr_X)$ as the full subcategory of $D_{\Qch}(\Oscr_X\otimes_k\Oscr_X)$
whose objects are obtained from complexes of quasi-coherent $\Oscr_X$-modules with $\Oscr_X\otimes \Oscr_X$ acting via the multiplication map $\Oscr_X \otimes \Oscr_X \to \Oscr_X$. We will need the following lemma
\begin{lemmas} There is an equivalence of categories
\[
D^\delta(\Oscr_X)\cong D^{\delta}(\Qch(X))
\]
which is the identity on objects where $D^{\delta}(\Qch(X))$ was introduced in \S\ref{ref-8.2-103}.
\end{lemmas}
\begin{proof} As usual let $i_\Delta:X\r X\times X$ be the diagonal. Put $\Ascr=i^{-1}_\Delta(\Oscr_{X\times X})$. Then
there is an obvious morphism of sheaves of algebras $\Oscr_X\otimes_k \Oscr_X\r \Ascr$ and we claim
it is flat. Indeed the stalk of $\Ascr$ at $x\in X$ is equal to the stalk
at $\Oscr_{X\times X}$ at $\Delta(x)$. So this stalk is equal
to the localization
of $\Oscr_{X,x}\otimes \Oscr_{X,x}$ at the kernel of the map $\Oscr_{X,x}\otimes \Oscr_{X,x}\r k(x)\otimes k(x)\r k(x)$.
Note that $\Ascr$ also maps to $\Oscr_X$, so we may define $D^\delta(\Ascr)$ as the full subcategory of
$D(\Ascr)$ spanned by objects which are obtained from complexes of quasi-coherent $\Oscr_X$-modules.

We have pairs of adjoint functors
\begin{equation}
\label{ref-9.1-126}
\xymatrix{
D(\Oscr_X\otimes \Oscr_X)\ar@/^1em/[rrr]^-{\Ascr\otimes_{\Oscr_X\otimes \Oscr_X}-}&&&
D(\Ascr)\ar@/^1em/[lll]^{-\mid\Oscr_X\otimes \Oscr_X} \ar@/_1em/[r]_{i_{\Delta,\ast}}&
D(\Oscr_{X\times X})\ar@/_1em/[l]_{i_\Delta^{-1}}
}
\end{equation}
whose unit/counit maps are isomorphisms in the categories $D^\delta(-)$. From this it follows
immediately that the functors in \eqref{ref-9.1-126} define inverse equivalences
between the $D^\delta(-)$.
\end{proof}
\begin{corollarys} 
\label{ref-9.1.2-127}  If $M\in D(\Qch(X))$ then
\[
\HH^\ast(X,M)=\Ext^\ast_{\Oscr_X\otimes\Oscr_X}(\Oscr_X,M).
\]
\end{corollarys}
\subsection{The characteristic class of a divisor}
Unless otherwise specified, in the rest of this section $X$ will be a
closed subscheme of a quasi-compact separated $k$-scheme $Y$ defined
by an invertible ideal $I$.  With a slight abuse of notation, we will
write $\Oscr_X=\Oscr_Y/I$ and we consider $\Oscr_X$ as a sheaf of
$k$-algebras on $Y$.

We prove a technical result (Lemma \ref{ref-9.2.1-132} below)
which will be used to show that certain Hochschild cohomology classes
are non-trivial (see Proposition \ref{ref-9.3.1-140} below). The result is probably known in some form to experts.  For
example Andrei C\u{a}ld\u{a}raru tells us that it would also follow from his
work with Arinkin \cite{CA} on derived self-intersections, modulo some technical verifications. Nonetheless, since
we were unable to find a written proof in the literature, we provide one here.

We consider the complex of $\Oscr_X$-bimodules
\[
C(X/Y)\overset{\text{def}}{=}\Oscr_X\Lotimes_{\Oscr_Y} \Oscr_X
\]
To compute the cohomology of $C(X/Y)$ we may view $C(X/Y)$ as a complex of $\Oscr_Y-\Oscr_Y$-bimodules. Using
the obvious $\Oscr_Y$-flat resolution of $\Oscr_X$ 
\[
0\r I\r \Oscr_Y \r \Oscr_X\r 0,
\]
we find
\begin{equation}
\label{ref-9.2-128}
H^\ast(C(X/Y))=
\begin{cases}
\Oscr_X&\text{if $i=0$}\\
I/I^2&\text{if $i=-1$}\\
0&\text{otherwise.}
\end{cases}
\end{equation}
In particular, we have a distinguished triangle of complexes of $\Oscr_X$-bimodules
\begin{equation}
\label{ref-9.3-129}
C(X/Y)\r \Oscr_X\xrightarrow{\xi_{X/Y}} \Sigma^2 I/I^2\r 
\end{equation}
with $\xi_{X/Y}\in \Ext^2_{\Oscr_X\otimes \Oscr_X}(\Oscr_X,I/I^2)=\HH^2(X,I/I^2)$, where we have used Corollary \ref{ref-9.1.2-127}. 

We will now give a classical avatar of $\xi_{X/Y}$ (see  Lemma \ref{ref-9.2.1-132}). By change of rings we have
\begin{equation}
\label{ref-9.4-130}
\HH^2(X,I/I^2)=\Ext^2_{\Oscr_X\otimes \Oscr_X}(\Oscr_X,I/I^2)=\Ext^2_{\Oscr_X}(\Oscr_X\Lotimes_{\Oscr_{X}\otimes \Oscr_X} \Oscr_X,I/I^2)
\end{equation}
and if $X,Y$ are smooth there is the Hochschild-Kostant-Rosenberg isomorphism
in the derived category of $\Oscr_X$-modules
\begin{equation}
\label{ref-9.5-131}
\HKR_\ast:\Oscr_X \Lotimes_{\Oscr_{X\times X}} \Oscr_X\r \wedge^\bullet\Omega_X
\end{equation}
If we represent $\Oscr_X \Lotimes_{\Oscr_{X}\otimes \Oscr_X} \Oscr_X$ by the usual Hochschild complex 
\[
\CC_\bullet(X):=\cdots \r\Oscr_X\otimes_k\Oscr_X\otimes_k\Oscr_X\r \Oscr_X\otimes_k\Oscr_X\r \Oscr_X,
\]
then $\HKR_\ast $ is given by sending a local section $f_0\otimes f_1\otimes\cdots \otimes f_n$ of $\Oscr_X^{\otimes n+1}$ to $f_0df_1\cdots df_n$.

In particular combining \eqref{ref-9.5-131} with \eqref{ref-9.4-130} we get a split injective map
\[
\HKR: \Ext^1_X(\Omega_X,I/I^2)\r \HH^2(X,I/I^2)
\]
\begin{lemmas}
\label{ref-9.2.1-132}
Assume that $X$, $Y$ are smooth, and
let $\xi'_{X/Y}\in \Ext^1_X(\Omega_X,I/I^2)$ correspond to the conormal sequence \cite[Prop.\ 8.4A]{H}
\[
0\r I/I^2\r \Oscr_X\otimes_{\Oscr_Y}\Omega_Y \r \Omega_X\r 0.
\]
Then
\begin{equation}
\label{ref-9.6-133}
\xi_{X/Y}=\HKR(\xi'_{X/Y}).
\end{equation}
In particular, if $\xi'_{X/Y}$ is non-zero then so is $\xi_{X/Y}$.
\end{lemmas}
\begin{proof}
To prove this lemma we have to understand better the distinguished triangle
\begin{equation}
\label{ref-9.7-134}
\Sigma I/I^2\r \Oscr_X\Lotimes_{\Oscr_Y}\Oscr_X\r \Oscr_X\r
\end{equation}
corresponding to $\xi_{X/Y}$. We represent $\Oscr_X\Lotimes_{\Oscr_Y}\Oscr_X$ by the bar complex $B_\bullet(X/Y)$
\[
\cdots \r\Oscr_X\otimes_k \Oscr_Y\otimes_k \Oscr_Y\otimes_k \Oscr_X\r  \Oscr_X\otimes_k \Oscr_Y\otimes_k\Oscr_X \r \Oscr_X\otimes_k \Oscr_X
\]
with the usual bar-differential. The analogous bar-complex $B_\bullet(X/X)$ is quasi-isomorphic to $\Oscr_X$ and the map
$\Oscr_X\Lotimes_{\Oscr_Y}\Oscr_X\r\Oscr_X$ in \eqref{ref-9.7-134} is represented by the map of complexes $B_\bullet(X/Y)\r B_\bullet(X/X)$. So
we obtain an exact sequence of complexes of $\Oscr_X$-bimodules
\begin{equation}
\label{ref-9.8-135}
0\r J_\bullet\r B_\bullet(X/Y)\r B_\bullet(X/X)\r 0
\end{equation}
where $J_\bullet$ is a complex concentrated in degrees $\le -1$ of the form
\[
\cdots \r\underbrace{\Oscr_X\otimes (I\otimes \Oscr_Y+\Oscr_Y\otimes I)\otimes \Oscr_X}_{J_2} \r 
\underbrace{\Oscr_X\otimes I\otimes \Oscr_X}_{J_1}\r \underbrace{0}_{J_0} \r 0
\]
Now one computes locally that $H^{-1}(J^\bullet)=I/I^2$ and since $J^\bullet$ is acyclic in other degrees by \eqref{ref-9.2-128} we obtain a quasi-isomorphism
\begin{equation}
\label{ref-9.9-136}
J_\bullet\r \Sigma(I/I^2)
\end{equation}
which sends local sections $f\otimes g\otimes h$ of $J_1$ to $f\bar{g}h$. Hence \eqref{ref-9.8-135}\eqref{ref-9.9-136} define a distinguished
triangle \emph{isomorphic} to \eqref{ref-9.7-134}. It takes some more straightforward verifications 
to show
that the two distinguished triangles are actually the same, which we leave to the reader. If one does not want to do this then  
one may take \eqref{ref-9.8-135}\eqref{ref-9.9-136} as defining $\xi_{X/Y}$. It then differs from the prior definition by
at most a global unit (as $I/I^2$ is a line bundle on $X$, this is the ambiguity in the choise of \eqref{ref-9.9-136}). Since we will be only interested in when $\xi_{X/Y}\neq 0$, this makes no difference.

Now we tensor \eqref{ref-9.8-135} on the left by $\Oscr_X\otimes_{\Oscr_X\otimes\Oscr_X}-$. Since all complexes in \eqref{ref-9.8-135} are flat over $\Oscr_X\otimes\Oscr_X$
this is in fact the derived tensor product. Furthermore one has obvious identifications of complexes of $\Oscr_X$-modules 
\begin{align*}
\Oscr_X\otimes_{\Oscr_X\otimes\Oscr_X} B_\bullet(X/Y)&=\Oscr_X\otimes_{\Oscr_Y} C_\bullet(Y)\\
\Oscr_X\otimes_{\Oscr_X\otimes\Oscr_X} B_\bullet(X/X)&=C_\bullet(X)
\end{align*}
Moreover $\Oscr_X\otimes_{\Oscr_{X}\otimes\Oscr_{X}} J_\bullet$ is a complex which ends with the term $\Oscr_X\otimes I$ in degree $-1$. We now immediately verify that we have a commutative
diagram of complexes
\begin{equation}
\label{ref-9.10-137}
\xymatrix{
0\ar[r]& \Oscr_X\otimes_{\Oscr_X\otimes\Oscr_X}J_\bullet\ar[r]\ar[d]_{\text{can}}& \Oscr_X\otimes_{\Oscr_Y} C_\bullet(Y)\ar[r]\ar[d] & C_\bullet(X)\ar[d]^{\operatorname{\HKR}}\ar[r]&0\\
0\ar[r]& \Sigma(I/I^2)\ar[r] &\Sigma(\Oscr_X\otimes_{\Oscr_Y}\Omega_Y)\ar[r]& \Sigma\Omega_X\ar[r]&0
}
\end{equation}
where the middle vertical map sends $f\otimes g\otimes h$ (in degree $-1$) to $f\bar{g}dh$. Then \eqref{ref-9.10-137} give a map between distinguished triangles in the derived category of $\Oscr_X$-modules. 
\begin{equation}
\label{ref-9.11-138}
\xymatrix{
\Oscr_X\Lotimes_{\Oscr_{X}\otimes\Oscr_X}\Sigma(I/I^2)\ar[d]_{\text{can}}
\ar[r]&\Oscr_X\Lotimes_{\Oscr_{X}\otimes\Oscr_X}C(X/Y)\ar[r]\ar[d]&
\Oscr_X\Lotimes_{\Oscr_X\otimes_k \Oscr_X} \Oscr_X\ar[r]^-{1\otimes \xi_{X/Y}}\ar[d]^{\HKR}&\\
\Sigma(I/I^2)\ar[r]& \Sigma(\Oscr_X\otimes_{\Oscr_Y}\Omega_Y)\ar[r]& \Sigma \Omega_X\ar[r]_{\Sigma \xi'_{X/Y}}&
}
\end{equation}
Completing \eqref{ref-9.11-138} with a third
commutative square we get
\begin{equation}
\label{ref-9.12-139}
\xymatrix{
\Oscr_X\Lotimes_{\Oscr_{X}\otimes_k \Oscr_X} \Oscr_X \ar[rr]^{\Id \otimes \xi_{X/Y}}
\ar[d]_{\operatorname{HKR}} 
& &\Sigma^2(\Oscr_X \Lotimes_{\Oscr_{X}\otimes \Oscr_X} I/I^2)
\ar[d]^{\text{can}}
\\
\Sigma\Omega_X\ar[rr]_{\Sigma \xi'_{X/Y}}
&&
\Sigma^2(I/I^2)
}
\end{equation}
Under the isomorphism \eqref{ref-9.4-130},
$\xi_{X/Y}$ corresponds to the diagonal composition 
$
\text{can}\circ (\xi_{X/Y}\otimes \id)
$,
 whereas $\operatorname{HKR}(\xi'_{X/Y})$ 
corresponds (by construction) to the other diagonal composition. Hence we are done.
\end{proof}
\subsection{A concrete example}
Here is a concrete example of a situation where $\xi_{X/Y}\neq 0$.
\begin{propositions} 
\label{ref-9.3.1-140} Let $X$ be a smooth hypersurface of degree $d>1$ in $Y=\PP^n$, $n\ge 2$.
Then $\xi_{X/Y}\neq 0$.
\end{propositions}
\begin{proof}
According to \cite{vdv}, the conormal sequence on $X$ is not split. The conclusion then follows
from Lemma \ref{ref-9.2.1-132}.
\end{proof}
\subsection{The long exact sequence associated with a divisor}
Recall the following:
\begin{lemmas}
Assume that $M,N$ be complexes of $\Oscr_X$-modules. Then there is a long exact sequence
\begin{equation}
\label{ref-9.13-141}
\cdots \r\Ext^{n-2}_X(I/I^2\otimes_{\Oscr_X} M,N)\xrightarrow{-\circ(\xi_{X/Y}\Lotimes \Id_M)}
\Ext^n_X(M,N)\xrightarrow{f_\ast}
\Ext^n_Y(M,N)\r\cdots
\end{equation}
\end{lemmas}
\begin{proof} Using change of rings we have
\begin{align*}
\Ext^n_Y(M,N)&=\Ext^n_X((\Oscr_X\Lotimes_{\Oscr_Y}\Oscr_X)\Lotimes_{\Oscr_X} M,N)\\
&=\Ext^n_X( C(X/Y)\Lotimes_{\Oscr_X} M,N)
\end{align*}
From \eqref{ref-9.3-129} we obtain a distinguished triangle
\[
C(X/Y)\Lotimes_{\Oscr_X} M\r M\xrightarrow{\xi_{X/Y}\Lotimes\Id_M} \Sigma^2(I/I^2\Lotimes_{\Oscr_X} M)\r 
\]
which yields the required long exact sequence by applying $\Hom_X(-,N)$.
\end{proof}
\subsection{Application to Hochschild cohomology}
\begin{propositions}
\label{ref-9.5.1-142}
Let $M$ be a complex of $\Oscr_X$-modules. There is a  
long exact sequence
\begin{multline*}
\cdots \r
\HH^{n-2}(X,(I/I^2)^{-1}\otimes_X M)
\xrightarrow{\xi_{X/Y}\otimes-}
\HH^n(X,M)\xrightarrow{f_\ast}
\HH^n(Y,M)\r\cdots
\end{multline*}
\end{propositions}
\begin{proof}
Let $\Gamma_f\subset Y\times_k X$ be the graph of $f$.
 The long exact sequence \eqref{ref-9.13-141} applied to $X\times_k X\r Y\times_k X$
becomes (using the dictionary $X\mapsto X\times_k X$, $Y\mapsto Y\times_k X$, $I/I^2\mapsto I/I^2\boxtimes_k \Oscr_X$, 
$M\mapsto \Oscr_{\Delta_X}$, $f\mapsto (f,\Id)$, $N\mapsto i_{\Delta_X,\ast}M$)
\begin{multline}
\label{ref-9.14-143}
\cdots \r\Ext^{n-2}_{X\times_k X}((I/I^2\boxtimes \Oscr_X)\otimes_{\Oscr_{X\times_k X}} \Oscr_{\Delta_X},i_{\Delta_X,\ast}M)
\xrightarrow{-\circ ((\xi_{X/Y}\boxtimes 1)\otimes\Id_{\Oscr_{\Delta_X}})}\\
\Ext^n_{X\times_k X}(\Oscr_{\Delta_X},i_{\Delta_X,\ast}M)\xrightarrow{(f,\Id)_\ast}
\Ext^n_{Y\times_k X}(\Oscr_{\Gamma_f},(f,\Id_X)_\ast M)\r\cdots
\end{multline}
Now $(I/I^2\boxtimes \Oscr_X)\otimes_{\Oscr_{X\times_k X}} \Oscr_{\Delta_X}=i_{\Delta_X,\ast}(I/I^2)$, and it is easy to see that
with this identification we have $(\xi_{X/Y}\boxtimes1)\otimes \Id_{\Oscr_{\Delta_X}}=\xi_{X/Y}$.

We have 
\[
\HH^n(X,M)=\Ext^n_{X\times_k X}(\Oscr_{\Delta_X},i_{\Delta_X,\ast}M)
\]
If we consider $M$ as $\Oscr_{Y\times_k Y}$ module, then it is in fact supported on $X\times_k X$. Hence
by adjointness
\begin{align*}
\HH^n(Y,M)&=\Ext^n_{Y\times_k Y}(\Oscr_{\Delta_Y},i_{\Delta_Y,\ast}M)\\
&=\Ext^n_{Y\times_k X}((\Id_Y,f)^\ast\Oscr_{\Delta_Y},(f,\Id_X)_\ast M)\\
&=\Ext^n_{Y\times_k X}(\Oscr_{\Gamma_f},(f,\Id_X)_\ast M)
\end{align*}
Finally, since $I/I^2$ is an invertible $\Oscr_X$-module, we have
\[
\Ext^{n-2}_{X\times_k X}(i_{\Delta,\ast}(I/I^2),i_{\Delta_X,\ast}M)=\HH^{n-2}(X,(I/I^2)^{-1}\otimes_X M)
\]
Substituting all this in \eqref{ref-9.14-143} yields what we want. 
\end{proof}
\subsection{The smooth proper case}
\label{ref-9.6-144}
Here we assume that $X$ and $Y$ are smooth, proper and connected, and are of dimension $m$, $m+1$ respectively.
\begin{lemmas}
\label{ref-9.6.1-145}
One has
\[
\HH^{2m}(X,\omega_X^{\otimes 2})=k.
\]
\end{lemmas}
\begin{proof} We have
\begin{align*}
\HH^{2m}(X,\omega_X^{\otimes 2})&=\Ext^{2m}_{\Oscr_{X\times X}}(\Oscr_\Delta,\omega_{\Delta}^{\otimes 2})\\
&=\Ext^{2m}_{\Oscr_{X\times X}}(\Oscr_\Delta,\Oscr_{\Delta}\otimes_{\Oscr_{X\times X}}(\omega_{X}\boxtimes \omega_{X}))\\
&=\Hom_{\Oscr_{X\times X}}(\Oscr_\Delta,\Oscr_{\Delta})^\ast\\
&=k.
\end{align*}
In the third line we have used that $\omega_X\boxtimes \omega_X$ is the canonical sheaf on $X\times X$, together
with  Serre duality \cite{Bondal4}.
\end{proof}
\begin{lemmas}
\label{ref-9.6.2-146}
Let $\Lscr$ be an invertible $\Oscr_X$-module. The multiplication pairing
\begin{equation}
\label{ref-9.15-147}
\HH^i(X,\Lscr)\otimes \HH^{2m-i}(X,\Lscr^{-1}\otimes_X \omega_X^{\otimes 2})\r \HH^{2m}(X,\omega_X^{\otimes 2})=k
\end{equation}
is non-degenerate.
\end{lemmas}
\begin{proof} This again a straightforward application of Serre
  duality, which says that the following pairing by composition is
  non-degenerate
\[
\Ext^i_{X\times X}(\Oscr_{\Delta},i_{\Delta,\ast} \Lscr)\otimes \Ext^{2m-i}_{X\times X}(i_{\Delta,\ast} \Lscr,\Oscr_{\Delta}\otimes_{X\times X} (\omega_X\boxtimes\omega_X))
\r  \Ext^{2m}_{X\times X}(\Oscr_{\Delta},\Oscr_{\Delta}\otimes_{X\times X} (\omega_X\boxtimes\omega_X))
\]
It is easy to see that this pairing  coincides with \eqref{ref-9.15-147}.
\end{proof}
\begin{propositions} \label{ref-9.6.3-148} Assume that $\xi_{X/Y}\neq 0$. Then 
\[
f_\ast:\HH^{2m}(X,\omega_X^{\otimes 2})\r \HH^{2m}(Y,\omega_X^{\otimes 2})
\]
is the zero map.
\end{propositions}
\begin{proof}
By Lemma \ref{ref-9.6.1-145} and Proposition \ref{ref-9.5.1-142}, it is sufficient to show that the map
\[
\HH^{2m-2}(X,(I/I^2)^{-1}\otimes_X \omega_X^{\otimes 2})
\xrightarrow{\xi_{X/Y}\otimes-}
\HH^{2m}(X,\omega_X^{\otimes 2})
\]
is not zero. This follows from Lemma \ref{ref-9.6.2-146}.
\end{proof}

\section{Construction of potential non-Fourier-Mukai functors}
\label{sec:construction}
We now assume that $X/k$, $Y/k$ are smooth of dimension $m$, $m+1$, and that $X$ is embedded as a 
divisor in $Y$. Let $f:X\r Y$ be the inclusion.
Define $\Xscr$, $\Yscr$ as 
in \S\ref{ref-8.2-103}. We have a fully faithful embedding
$w:\Qch(X)\r\Mod(\Xscr)$ (see \S\ref{ref-8.2-103}). Recall the following:
\begin{lemma}
\label{ref-10.1-149}
\begin{enumerate}
\item If $E\in \Qch(X)$ is injective then so is $wE$.
\item Every object in $\Qch(X)$ has injective dimension $\le m$.
\end{enumerate}
\end{lemma}
\begin{proof}
\begin{enumerate}
\item In this case, $\Mod(\Xscr)$ is a locally noetherian category, so a
  direct sum of injectives in $\Mod(\Xscr)$ is injective.  Let $E\in \Qch(X)$ be
  injective. Since $w$ commutes with direct sums, and since $X$ is
  locally noetherian, we may without loss of generality assume that 
  $E$ is indecomposable, and hence as in $J(x)$ for
  $x$ a not necessarily closed point in $X$, see \cite[Thm II.7.18, proof]{RD}.
Since $J(x)$ only depends on the local ring $\Oscr_{X,x}$ we obtain
that $wE$ satisfies
\[
(wE)(I)=
\begin{cases}
\Gamma(X,E)&x\in U_I\\
0&\text{otherwise}
\end{cases}
\]
where $\Gamma(X,E)$ is an injective $\Gamma(X,\Oscr_{U_I})$-module for all $I$ such that $x\in U_I$.
Let $I$ be the largest subset of $\{1,\ldots,n\}$ such that $x\in U_I$.
We find for $M\in \Mod(\Xscr)$:
\[
\Hom_{\Xscr}(M,wI)=\Hom_{\Oscr_X(U_I)}(M(I),\Gamma(X,E))
\]
which is an exact functor.
\item For each $U_i\subset X$ we have $\gldim \Oscr_X(U_i)=m$. It now suffices to note that on a noetherian scheme the propery of being quasi-coherent injective is local
\cite[Prop.\ II.7.17]{RD}.\qedhere
\end{enumerate}
\end{proof}
Let $M$ be a line bundle on $X$, and let $\Mscr=W(M)$ be the corresponding
$\Xscr$-bimodule. By Lemma \ref{ref-8.7.1-122}  we know that $\Mscr$ is invertible. We will denote its two sided inverse by $\Mscr^{-1}$.

Choose $n\ge m+3$ and
let $\eta\in\ker(\HH^{n}(X, M)\r \HH^{n}(Y,f_\ast M))$. See Lemma \ref{ref-9.6.1-145} and in particular
Proposition \ref{ref-9.6.3-148} for how one may choose such $\eta$ in the proper case.

For simplicity, denote the corresponding element $W(\eta)\in \HH^{n}(\Xscr,\Mscr)$ by $\eta$ as well.  
Define $\Xscr_\eta$ as in \S\ref{ref-6.1-74}. By \eqref{ref-8.14-121}, we have that $W(\eta)
\in \ker(\HH^{n}(\Xscr,\Mscr)\r \HH^{n}(\Yscr,(f,f)_\ast \Mscr))$. Hence, by
Proposition \ref{ref-7.2.6-95}, we have a commutative diagram
\begin{equation}
\label{diag:lift}
\xymatrix{
\Xscr_\eta\ar[dr]&&\ar[ll]_{\tilde{f}} \Yscr\ar[dl]^{f}\\
&\Xscr
}
\end{equation}
Put $\Xscr^{\dg}_\eta=U^u(\Xscr_\eta)$ (see Appendix \ref{ref-C.1-183}). Then we
have 
\begin{equation}
\label{ref-10.1-150}
H^\ast(\Xscr^{\dg}_\eta)=
\begin{cases}
\Xscr&\text{if $i=0$}\\
\Mscr&\text{if $i=-n+2$}\\
0&\text{otherwise.}
\end{cases}
\end{equation}
We define the functor
\[
{L}:\Inj \Qch(X)\r D(\Xscr):E\mapsto {L}(wE),
\]
where ${L}(wE)$ is the derived injective (see \S\ref{ref-5.1-60}) in $D(\Xscr)$ associated to the injective $wE$.

Since $\Qch(X)$ has global dimension $m$, by Lemma \ref{ref-10.1-149}(2) we are now in
a position to apply  Proposition \ref{ref-5.3.1-69} with $\Ascr=\Qch(X)$, $\cc=\Xscr^{\dg}_\eta$.
This yields an exact functor
\begin{equation}
\label{ref-10.2-151}
{L}:D^b(\Qch(X))\r D(\Xscr^{\text{dg}}_\eta).
\end{equation}
\begin{remark} \label{ref-10.2-152} We cannot apply Proposition \ref{ref-5.3.1-69} with $\Ascr=\Mod(\Xscr)$, since if
$X$ is proper then
it is easy to see that $\gldim \Xscr\ge 2m$. 
\end{remark}

\begin{lemma}
\label{ref-10.3-153}
Let $B\in D^b(\Qch(X))$ and $\Bscr=wB$. Then there is a distinguished
  triangle in $D(\Xscr_{\eta}^{\dg})$
\begin{equation}
\label{ref-10.3-154}
\Bscr\r {L}(B)\r \Sigma^{-n+2} \Mscr^{-1}\otimes_{\Xscr} \Bscr\r 
\end{equation}
where $\Bscr$, $\Mscr^{-1}\otimes_{\Xscr} \Bscr$ are viewed as $\Xscr_{\eta}^{\dg}$-modules via the map
\[
\Xscr_\eta^{\dg}\r \Xscr.
\]
\end{lemma}
\begin{proof}
We have a distinguished triangle in $D(\Xscr^{\dg}_\eta\otimes_k \Xscr^{\dg,\circ}_\eta)$
\[
\Sigma^{n-2}\Mscr\r\Xscr^{\dg}_\eta\r \Xscr\r.
\]
Applying $\RHom_{\Xscr^{\dg}_\eta}(-,L(B))$
gives a distinguished triangle in $D(\Xscr_{\eta}^{\dg})$
\begin{equation}
\label{eq:dist2}
\Hom_{\Xscr^{\dg}_\eta}(\Xscr,L(B))
\r {L}(B)\r 
\Hom_{\Xscr^{\dg}_\eta}(\Sigma^{n-2} \Mscr,L(B))\r.
\end{equation}
 and using \eqref{ref-5.11-72} we get in $D(\Xscr)$
\begin{align*}
\Hom_{\Xscr^{\dg}_\eta}(\Xscr,L(B))&\cong\Bscr\\
\Hom_{\Xscr^{\dg}_\eta}(\Sigma^{n-2} \Mscr,L(B))&\cong\Hom_{\Xscr}(\Sigma^{n-2} \Mscr,\Bscr).
\end{align*}
By applying $D(\Xscr)\r D(\Xscr^{\dg}_\eta)$ we see that these identities also hold in $D(\Xscr^{\dg}_\eta)$.
So \eqref{eq:dist2} becomes a distinguished triangle in $D(\Xscr^{\dg}_\eta)$:
\[
\Bscr\r {L}(B)\r \Hom_{\Xscr}(\Sigma^{n-2} \Mscr,\Bscr)\r.
\]
It now suffices to observe that $\Mscr$ is invertible.
\end{proof}
\begin{corollary}
\label{ref-10.4-155}
If $B\in D^b(\coh(X))$ then $H^\ast({L}(B))\in D^b_{w\coh(X)}(\Xscr^{\dg}_\eta)$.
\end{corollary}
\begin{proof} This follows from Lemma \ref{ref-10.3-153} and Lemma \ref{ref-8.3.1-107}.
\end{proof}
The functor we would now want to consider is the composition 
\begin{equation}
\label{ref-10.4-156}
\Psi:D^b(\coh(X))\xrightarrow{{L}} D^b_{w\coh(X)}(\Xscr^{\text{dg}}_\eta)\xrightarrow{\psi_{\Xscr_\eta,\ast}} D^b_{w\coh(X)}(\Xscr_\eta)\xrightarrow{\tilde{f}_\ast} D^b_{w\coh(Y)}(\Yscr)
\cong D^b(\coh(Y))
\end{equation}
($\psi_{\Xscr_\eta}$ is defined in \S\ref{ref-C.1-183}, the last isomorphism is from Lemma \ref{ref-8.2.1-105}).

\section{Proof of Theorem \ref{ref-11.1-160}}
\label{ref-11-159}
We remind the readers of the statement of the theorem:
\begin{theorem} Let $X$
  be a smooth quadric in $Y=\PP^{4}$ whose defining equation has maximal isotropy index\footnote{This condition ensures that the matrix factorization in Theorem \ref{thm:kapranov} is defined over any field. One may take  $x_0^2+x_1x_2+x_3x_4=0$.}
   and let $f:X\r
  Y$ be the inclusion. Let $M=\omega_X^{\otimes 2}$ and let $0\neq \eta\in\HH^{6}(X,\omega_X^{\otimes 2})\cong k$. Then
  $f_\ast\eta\in \HH^{6}(Y,f_\ast (\omega_X^{\otimes 2}))$ is zero. The functor
  $\Psi$ in \eqref{ref-1.7-9} restricts to an exact functor
\[
\Psi:D^b(\coh(X))\r D^b(\coh(Y))
\]
which is not a Fourier-Mukai functor. 
\end{theorem}
The fact that $\HH^6(X,M)=k$ is Lemma \ref{ref-9.6.1-145}.
The fact that $f_\ast(\eta)=0$ follows 
from Proposition \ref{ref-9.6.3-148} using Proposition \ref{ref-9.3.1-140}.

\medskip

Let $\Oscr_Y(1)$ be the tautological line bundle on $Y=\PP^4$ and let $\Oscr_X(1)$ be its restriction to $X$. Then we have
$\omega_X=\Oscr_X(-3)$. If, as usual, $I$ is the defining ideal of $X$ in $Y$, then $I\cong \Oscr_Y(-2)$ and hence
$I/I^2=\Oscr_X(-2)$. Recall the following:
\begin{theorem} \label{thm:kapranov} \cite{Kapranov3} There is a full strong exceptional sequence on $X$ given by
\begin{equation}
\label{ref-11.1-161}
\Oscr_X(-2), \Oscr_X(-1), \Oscr_X, C
\end{equation}
where $C$ is associated to a matrix factorization of the defining equation of $X$ in $4\times 4$-matrices.
In particular, it has a resolution
on $Y$ given by
\[
0\r \Oscr_Y(-1)^4\r \Oscr_Y^4\r C\r 0
\]
\end{theorem}
If follows that all the objects occurring in the exceptional sequence \eqref{ref-11.1-161} are arithmetically Cohen-Macaulay. Furthermore
since odd dimensional quadrics have up to shift only a single non-trivial indecomposable Cohen-Macaulay module $\Cscr$ (e.g.\ by Kn{\"o}rrer periodicity \cite{Knorrer}),
$\Cscr$ must be its own syzygy up to shift.
One finds that there is a short exact sequence on $X$
\begin{equation}
\label{ref-11.2-162}
0\r C(-1)\r \Oscr_X^{4}\r C\r 0.
\end{equation}
Let $T$ be the sum of the exceptional collection \eqref{ref-11.1-161} and put $\Gamma=\End_X(T)$. Since $\Gamma$ is given by a directed algebra
with maximal compositions of length 3, we find
\begin{equation}
\label{ref-11.3-163}
\gldim \Gamma\le 3.
\end{equation}
For use below we record the following technical vanishing result. This lemma, \eqref{ref-11.3-163} and the fact that $T$ is a coherent sheaf are the only special properties of $T$ that we will use.
\begin{lemma} 
\label{ref-11.3-164} One has
\[
\Ext^i_X(T,(I/I^2)^{-1}\otimes_X M\otimes_X T)=0
\]
for $i=1,2$.
\end{lemma}
\begin{proof} We have $(I/I^2)^{-1}\otimes_X M=\Oscr_X(-4)$.
By Serre duality
\[
\Ext^i_X(T,T(-4))=\Ext^{3-i}_X(T,T(1))^\ast
\]
Since $T$ is arithmetically Cohen-Macaulay, we have
\[
H^i(X,T(j))=0
\]
for all $j$ and $i=1,2$. Hence by Serre duality
\[
\Ext^i_X(T,\Oscr_X(j))=0
\]
for all $j$ and $i=1,2$.

 So the only possible issue is the value of
\[
\Ext^i_X(C,C(1))
\]
for $i=1,2$. From (a shift by 1 of) the exact sequence \eqref{ref-11.2-162} we get 
\[
\Ext^1_X(C,C(1))=\Ext_X^2(C,C)=0
\]
since $C$ is exceptional. Similarly
\[
\Ext^2_X(C,C(1))\hookrightarrow \Ext_X^3(C,C)=0.\qed
\]
\def\qed{}\end{proof}
Consider first the functor (see \eqref{ref-10.2-151} and Corollary \ref{ref-10.4-155})
\[
{L}:D^b(\coh(X))\r D^b_{w\coh(X)} (\Xscr^{\dg}_\eta)
\]
Put $\Tscr=wT\in \Mod(\Xscr)$ and 
 $\tilde{\Tscr}={L}(T)\in D^b_{w\coh(X)} (\Xscr^{\dg}_\eta) $. Since $\Gamma$ acts on $T$ it also acts on $\tilde{\Tscr}$ and hence $\tilde{\Tscr}\in D(\Xscr^{\dg}_\eta)_\Gamma$.
According to Lemma \ref{ref-7.3.1-98} and Corollary \ref{ref-7.3.2-99}, there is a well defined obstruction 
\[
o_3(\tilde{\Tscr})\in \HH^3(\Gamma,\Ext^{-1}_{\Xscr^{\dg}_\eta}(\tilde{\Tscr},\tilde{\Tscr}))
\]
which vanishes if $\tilde{\Tscr}$ is in the essential image of
\[
D(\Xscr^{\dg}_\eta\otimes_k \Gamma)\r D(\Xscr^{\dg}_\eta)_\Gamma
\]
\begin{lemma} \label{ref-11.4-165} One has
\[
o_3(\tilde{\Tscr})\neq 0
\]
\end{lemma}
\begin{proof} 
 If $o_3(\tilde{\Tscr})$ were to vanish, then the higher
obstructions $o_{3+i}(\tilde{\Tscr})$ (Lemma \ref{ref-7.3.1-98}), which lie in $\HH^{3+i}(\Gamma,\Ext^{-1-i}
_{\Xscr^{\dg}_\eta}(\tilde{\Tscr},\tilde{\Tscr}))$, would also vanish, since $\gldim \Gamma=3$ and hence
the Hochschild dimension of $\Gamma$ is $3$ as well.
  So $\tilde{\Tscr}$ may be viewed
as an object in $D(\Xscr^{\dg}_\eta\otimes_k \Gamma)\cong D_\infty(\Xscr_\eta\otimes_k \Gamma)$. By \eqref{ref-10.3-154} we get
a distinguished triangle in $D(\Xscr^{\dg}_\eta)$
\begin{equation}
\label{ref-11.4-166}
\Tscr \xrightarrow{\alpha} \tilde{\Tscr}\xrightarrow{\beta} \Sigma^{-4}\Mscr^{-1}\otimes_{\Xscr} \Tscr\r 
\end{equation}
and hence
\[
H^\ast(\tilde{\Tscr})=\Tscr\oplus \Sigma^{-4}(\Mscr^{-1}\otimes_{\Xscr} \Tscr),
\]
and moreover, by construction, this isomorphism is compatible with the $H^\ast(\Xscr^{\dg}_\eta)=H^\ast(\Xscr_\eta)$ and
$\Gamma$-actions. In the terminology of \S\ref{ref-6.4-79}, $\tilde{\Tscr}$ is a colift
of $\Tscr\in D(\Xscr\otimes_k \Gamma)$ to $D_\infty((\Xscr_\eta\otimes_k \Gamma)_{\eta\cup 1})$
(see \S\ref{ref-6.2-75}). 

The
obstruction against the existence of such a colift
is the image of $\eta\cup 1$ under the characteristic map
\[
\HH^6(\Xscr\otimes_k \Gamma, \Mscr\otimes_k\Gamma)\xrightarrow{c_{\Tscr}} \Ext^6_{\Xscr\otimes_k \Gamma}(\Tscr,\Mscr\otimes_{\Xscr} \Tscr)
\]
(see Lemma \ref{ref-6.4.1-80}, Lemma \ref{ref-6.3.1-77} and \S\ref{ref-6.2-75}).
Let $c_{\Tscr,\Gamma}$ be the composition 
\[
\HH^6(\Xscr,\Mscr)\xrightarrow{\eta\mapsto \eta\cup 1}\HH^6(\Xscr\otimes_k \Gamma, \Mscr\otimes_k\Gamma)\xrightarrow{c_{T}} \Ext^6_{\Xscr\otimes_k \Gamma}(\Tscr,\Mscr\otimes_{\Xscr} \Tscr)
\]
 By the $\Gamma$-equivariant version of \eqref{ref-8.11-117} we have a commutative diagram
\[
\xymatrix{
\HH^6(X,M)\ar[d]^{\cong}_W\ar[r]^-{c_{T,\Gamma}} & \Ext^6_{\Qch(X)_\Gamma}(T,M\Lotimes_X T)\ar[d]^{w}_{\cong}\\
\HH^6(\Xscr,\Mscr)\ar[r]_-{c_{\Tscr,\Gamma}} & \Ext^6_{\Xscr\otimes_k \Gamma}(\Tscr,\Mscr\Lotimes_\Xscr \Tscr)
}
\]
of $\Gamma$-equivariant characteristic maps.
The rightmost map is an isomorphism by \eqref{ref-8.10-116} and \eqref{eq:equivariant}. The leftmost map is an isomorphism by \eqref{eq:hochschild}.
By Proposition \ref{ref-8.8.2-125}, the upper horizontal map is also an isomorphism, finishing the proof that the colift does not exist (as $\eta\neq 0$).
\end{proof}
Applying $\RHom_{\Xscr^{\dg}_\eta}(-,\tilde{\Tscr})$ to \eqref{ref-11.4-166},  and using \eqref{ref-5.11-72}, we get
 a distinguished triangle of complexes
\[
\RHom_{\Xscr} (\Sigma^{-4}\Mscr^{-1}\otimes_{\Xscr} \Tscr,\Tscr) \r \RHom_{\Xscr^{\dg}_\eta}(\tilde{\Tscr},\tilde{\Tscr})\r \RHom_{\Xscr}(\Tscr,\Tscr)\r 
\]
and hence
\begin{equation}
\label{ref-11.5-167}
\Ext^3_{\Xscr}(\Mscr^{-1}\otimes_{\Xscr} \Tscr,\Tscr)\cong\Ext^{-1}_{\Xscr^{\dg}_\eta}(\tilde{\Tscr},\tilde{\Tscr})
\end{equation}
For use below, we note that this isomorphism sends a morphism $\phi:\Mscr^{-1}\otimes_{\Xscr} \Tscr\r \Tscr$ of degree three to the composition
\begin{equation}
\label{ref-11.6-168}
\tilde{\Tscr}\xrightarrow{\beta} \Sigma^{-4}\Mscr^{-1}\otimes_{\Xscr} \Tscr\xrightarrow{\phi} \Sigma^{-1}\Tscr\xrightarrow{\Sigma^{-1}\alpha} \Sigma^{-1}\tilde{\Tscr}.
\end{equation}
\begin{lemma}
There is a commutative diagram
\begin{equation}
\label{ref-11.7-169}
\xymatrix{
\Ext^3_{\Xscr}(\Mscr^{-1}\otimes_{\Xscr} \Tscr,\Tscr)\ar[d]_{f_\ast}\ar[r]^{\cong} &\Ext^{-1}_{\Xscr^{\dg}_\eta}(\tilde{\Tscr},\tilde{\Tscr})\ar[d]^{\tilde{f}_\ast\circ \psi_{\Xscr_\eta,\ast}}\\
\Ext^3_{\Yscr}(f_\ast(\Mscr^{-1}\otimes_{\Xscr} \Tscr),f_\ast(\Tscr))\ar[r]&\Ext^{-1}_{\Yscr}(\tilde{f}_\ast(\tilde{\Tscr}),\tilde{f}_\ast(\tilde{\Tscr}))
}
\end{equation}
where $\tilde{f}$ is as in \eqref{diag:lift}, the upper map is as in \eqref{ref-11.5-167} and the lower map is defined in a similar way as \eqref{ref-11.6-168}.
\end{lemma}
\begin{proof} This is a tautology starting from \eqref{ref-11.6-168}.
\end{proof}
\begin{corollary} \label{ref-11.6-170}
The map $\tilde{f}_\ast \circ \psi_{\Xscr_\eta,\ast} $ in \eqref{ref-11.7-169} is an isomorphism.
\end{corollary}
\begin{proof}
Since $\dim Y=4$ and $T$, $M^{-1}\otimes T$ are coherent sheaves on $X$, we have $\Ext^1_{\Yscr}(f_\ast(\Sigma^{-4}\Mscr^{-1}\otimes_{\Xscr} \Tscr),f_\ast(\Tscr))=
\Ext^1_{Y}(f_\ast(\Sigma^{-4}M^{-1}\otimes_{X} T),f_\ast(T))=0$ (the first equality follows from \eqref{eq:fullyfaithful} and
Lemmas \ref{ref-8.3.1-107}, \ref{ref-8.6.1-119}) and hence
\[
\tilde{f}_\ast(\tilde{\Tscr})=f_\ast(\Tscr)\oplus f_\ast(\Sigma^{-4}\Mscr^{-1}\otimes_{\Xscr} \Tscr),
\]
from which it follows right away that the lower arrow on \eqref{ref-11.7-169} is an isomorphism. So it
suffices to show that the left arrow is an isomorphism. Again using \eqref{eq:fullyfaithful} and  Lemmas  \ref{ref-8.3.1-107}, \ref{ref-8.6.1-119}, it is sufficient
to prove that
\[
\Ext^3_{X}(M^{-1}\otimes_X T,T)\xrightarrow{f_\ast} \Ext^3_{Y}(f_\ast(M^{-1}\otimes_{X} T),f_\ast(T))
\]
is an isomorphism.
We have a long exact sequence (see \eqref{ref-9.13-141})
\begin{multline*}
\Ext^1_{X}(T,(I/I^2)^{-1}\otimes_X M\otimes_X T))\r 
\Ext^3_{X}(T,M\otimes_X T)\xrightarrow{f_\ast}\\
\Ext^3_{Y}(f_\ast(T),f_\ast(M\otimes_X T))\r
\Ext^2_{X}(T,(I/I^2)^{-1}\otimes_X M\otimes_X T)).
\end{multline*}
Now by Lemma  \ref{ref-11.3-164} we have
\[
\Ext^1_{X}(T,(I/I^2)^{-1}\otimes_X M\otimes_X T)=\Ext^2_{X}(T,(I/I^2)^{-1}\otimes_X M\otimes_X T)=0,
\]
so that we have
\[
\Ext^3_{X}(M^{-1}\otimes_X T,T)\xrightarrow[\cong]{f_\ast}
\Ext^3_{Y}(f_\ast(M^{-1}\otimes_X T),f_\ast T)
\]
and we are done.
\end{proof}
\begin{proof}[Proof of Theorem \ref{ref-11.1-160}]
We follow the strategy exhibited in the introduction.
Since $\Psi$ is a functor, we obviously have 
\[
\Psi(T)=(\tilde{f}_\ast\circ \psi_{\Xscr_\eta,\ast})(\tilde{\Tscr})\in D(\Yscr)_\Gamma.
\]
If $\Psi$ is Fourier-Mukai then $\Psi(T)\in D(\Yscr\otimes_k \Gamma)$.  It now suffices to use Lemma \ref{nonvanishing obstruction} below to obtain a contradiction.
\end{proof}
\begin{lemma}\label{nonvanishing obstruction} The obstruction  (see Lemma \ref{ref-7.3.1-98})
\[
o_{3}(\Psi(T))\in \HH^3(\Gamma,\Ext^{-1}_{\Yscr}(\Psi(T),\Psi(T))
\]
is not vanishing. 
\end{lemma}
\begin{proof} By the naturality of obstructions (see Lemma \ref{ref-7.3.1-98}),
we have
\[
o_{3}(\Psi(T))=(\tilde{f}_\ast\circ \psi_{\Xscr_\eta,\ast})(o_3(\tilde{\Tscr})).
\]
By Corollary \ref{ref-11.6-170} we know that $\tilde{f}_\ast\circ \psi_{\Xscr_\eta,\ast}$ is
an isomorphism and by Lemma \ref{ref-11.4-165} we have $o_3(\tilde{\Tscr})\neq 0$.
\end{proof}
Done!
\appendix
\section{The virtual kernel cohomology}
\label{sec:virtual}
\begin{theoremdefinition} \label{thm:virtual} \cite[Thm 1.1, proof]{Rizzardo1} (see also \cite{CS2})
Let $X$, $Y$ be smooth projective $k$-varieties and let
$F:D^b(\coh(X))\r D^b(\coh(Y))$ be an exact functor. 
Choose an ample line bundle $\Oscr_X(1)$ on $X$, and
let $\Gamma_\ast(X)$ be the corresponding homogeneous coordinate ring.

Define 
\[
\widetilde{\Hscr}^i =\bigoplus_j H^i(F(\Oscr_X(j))).
\]
Then $\widetilde{\Hscr}^i$ is a noetherian $\Gamma_\ast(X)\otimes_k \Oscr_Y$-module. Let $\Hscr^i$ be the
corresponding sheaf of $\Oscr_{X\times_k Y}$-modules. 

If $F$ is a Fourier-Mukai functor with kernel $\Kscr\in D^b(\coh(X\times_k Y))$, then $H^i(\Kscr)\cong \Hscr^i$.
\end{theoremdefinition}
We will refer to $(\tilde{H}^i)_i$ as the \emph{virtual kernel
  cohomology} of $F$. 
\begin{proposition}
Let $F=\Psi$ where $\Psi$ is as in \eqref{ref-10.4-156}. Then
\begin{equation}
\label{ref-10.6-158}
\Hscr^\ast=i_\ast \Oscr_X\oplus \Sigma^{-n+2}i_\ast M^{-1}.
\end{equation}
where $i:X\r X\times_k Y$ is given by $i(x)=(x,f(x))$. 
\end{proposition}
\begin{proof}
According to Lemma \ref{ref-10.3-153} and
Lemma \ref{ref-8.3.1-107} we find
\[
H^\ast({L}(\Oscr_X(j)))=w (\Oscr_X(j)\oplus \Sigma^{-n+2}(M^{-1}\otimes_X \Oscr_X(j)))
\]
and since the last 3 morphisms in \eqref{ref-10.4-156} are essentially the pushforward by $f$ on the level of cohomology:
\begin{equation}
\label{ref-10.5-157}
\tilde{\Hscr}^\ast=f_\ast \Oscr_X(j)\oplus \Sigma^{-n+2}f_\ast(M^{-1}\otimes_X \Oscr_X(j))
\end{equation}
From
\eqref{ref-10.5-157} we easily deduce \eqref{ref-10.6-158}.
\end{proof}
\section{(Non-)existence of topological lifts}\label{stable infinity}
\def\Cat{\operatorname{Cat}}
\def\Sp{\operatorname{Sp}}
\def\dgcat{\operatorname{dgcat}}
\subsection{Spectral categories} 
Let $\Sp$
be the symmetric monoidal category of symmetric spectra
\cite{SchwedeBook}. It is useful to think of $\Sp$ as an ``absolute''  analogue
of the category of unbounded complexes of abelian groups. 
A \emph{spectral category/functor} is a category/functor enriched
in $\Sp$. These are absolute analogues of DG-categories and functors.

Given a spectral category $\Ascr$, one can form a category
$\pi_0\Ascr$ by keeping the same set of objects and  putting $(\pi_0\Ascr)(x,y)=\pi_0(\Ascr(x,y))$.
A spectral functor
 $F:\Ascr\r \Bscr$ is a \emph{quasi-equivalence} if $\pi_0 F:\pi_0\Ascr\r \pi_0\Bscr$ is an equivalence and moreover
$F$ induces isomorphisms $\pi_\ast(\Ascr(x,y))\r \pi_\ast(\Bscr(Fx,Fy))$. Let $\Cat_{\Sp}$ be the category
of spectral categories/functors and let $\Ho(\Cat_{\Sp})$ be
the corresponding homotopy category obtained by inverting quasi-equivalences.

\medskip

If $R$ is a commutative ring then
the Eilenberg Maclane spectrum $H\!R$ is a commutative monoid in $\Sp$ \cite[Example 5.25]{SchwedeBook}.
Let $\Cat_{\Sp}(H\!R)$ be the category of $H\!R$-linear spectral categories and let $\Ho(\Cat_{\Sp}(H\!R))$ be
the corresponding homotopy category. We use similar notations for DG-categories: $\dgcat(R)$ and $\Ho(\dgcat(R))$.
Let~$U:\Ho(\Cat_{\Sp}(HR))\r \Ho(\Cat_{\Sp})$ be the forgetful functor.
\begin{lemmas}\label{Q-linear}
The forgetful functor $U:\Ho(\Cat_{\Sp}(H\QQ))\r \Ho(\Cat_{\Sp})$ is fully faithful.
\end{lemmas}
\begin{proof}
The functor $U$ has a left adjoint $L$ given by smashing Hom-spaces with the Eilenberg-MacLane spectrum $H\QQ$. Because $H\QQ$ is a localization of the sphere spectrum \cite[Theorem 7.11]{Rudyak}, $LU$ is naturally isomorphic to the identity and hence $U$ is fully faithful.
\end{proof}
In \cite{Tabuada} (following \cite{Shipley}) Tabuada constructs an equivalence\footnote{The result is stated for $R=\ZZ$
but it works for  any $R$, see the comment in \cite[\S2.2]{Shipley}} ${\mathbb H}:\Ho(\dgcat(R))\r \Ho(\Cat_{\Sp}(H\!R))$. 
The construction of ${\mathbb H}$ is quite involved but it is not hard to verify that 
for a $R$-linear DG-category $\cc$ one has an $R$-linear equivalence between $\pi_0{\mathbb H}(\cc)$ and $H^0(\cc)$.

If $\cc$, $\dd$ are $R$-linear DG-categories and $\Psi:H^0(\cc)\r H^0(\dd)$ is an $R$-linear functor
then a \emph{spectral lift} of $\Psi$ is a morphism $\widetilde{\Psi}:U\mathbb{H}(\cc)\r U\mathbb{H}(\dd)$
in $\Ho(\Cat_{\Sp})$ such that $\pi_0\widetilde{\Psi}=\Psi$.

The following lemma shows that if $R=\QQ$ there is no difference between a spectral lift and a $\QQ$-linear DG-lift.
\begin{lemmas} \label{lem:QQ} If $R=\QQ$ then a spectral lift of $\Psi:H^0(\cc)\r H^0(\dd)$ exists if and only if there exists
a morphisms $\overline{\Psi}:\cc\r \dd$ in $\Ho(\dgcat(\QQ))$ such that $H^0(\overline{\Psi})=\Psi$.
\end{lemmas}
\begin{proof} Let $\widetilde{\Psi}:U\mathbb{H}(\cc)\r U\mathbb{H}(\dd)$ in $\Cat_{\Sp}$ be a spectral lift of $\Psi$.
Since $U$ is fully faithful by Lemma \ref{Q-linear} and $\mathbb{H}$
is an equivalence we may put $\overline{\Psi}=(U\mathbb{H})^{-1}\widetilde{\Psi}$.
\end{proof}
\subsection{Our functor}
For a noetherian scheme $X$ we let $D_{\dg}^b(\coh(X))$ be the standard DG-enhancement of $D^b(\coh(X))$ using injective resolutions.
\begin{propositions}\label{spectral lift}
If $k=\QQ$ then the functor $\Psi$ defined in Theorem \ref{ref-11.1-160} does not have a spectral lift.
\end{propositions}
\begin{proof}
If $\Psi$ has a spectral lift then by Lemma \ref{lem:QQ} it has a $\QQ$-linear DG-lift. 
But then
$\Psi_{\QQ}$ has to be a Fourier-Mukai functor by \cite[Thm\ 8.15]{Toen}
which contradicts Theorem \ref{ref-11.1-160} 
\end{proof}
\begin{remarks} We conjecture that Proposition \ref{spectral lift} holds for any field. 
\end{remarks}
\begin{remarks}
Spectral categories form a rigid model for the category of $\infty$-categories.
See  \cite[Thm.\ 4.23]{Blumberg} for a precise statement.
From  Proposition \ref{spectral lift} one may deduce that the functor $\Psi$
  does not lift to an exact $\infty$-functor in the sense of \cite[\S1.1.4]{LurieHA}.
\end{remarks}
\subsection{Vologodsky's functor}
Let us first remind the reader how Vologodsky's construction \cite{Vologodsky} works. Let $Y$ be a smooth projective scheme over $\ZZ_p$ and let~$X/\FF_p$ be its special fiber. Let $i:X\r Y$ be the corresponding embedding and put $\Phi=Li^\ast\circ i_\ast$. For a carefully chosen $Y$, Vologodsky shows that $\Phi$ is not a Fourier-Mukai functor over $\FF_p$.
\begin{observations} \label{lem:vg} The functor $\Phi$ has a $\ZZ$-linear DG-lift and hence a $H\ZZ$-linear spectral lift.
\end{observations}
\begin{proof} Let $i^!:\coh(Y)\r \coh(X)$ be the right adjoint to $i_\ast:\coh(X)\r \coh(Y)$. It is easy to see that $Li^\ast\cong \Sigma\circ Ri^!$. Hence it is sufficient to construct
a $\ZZ$-linear DG-lift for $\Phi':=Ri^!\circ i_\ast$. As $\Phi'$ is a composition of two right derived functors
it suffices to invoke Lemma \ref{lem:dglift} below.
\end{proof}
We now recall some standard facts. If $\mathfrak{a}$, $\mathfrak{b}$ are DG-categories then a \emph{co-quasi-functor}
$M:\mathfrak{b}\r \mathfrak{a}$ is a $\mathfrak{a}{-}\mathfrak{b}$-bimodule (i.e.\ a DG-functor $\bb^\circ\otimes \aa\r C(\Ab)$) such that for
 for every
$B\in \mathfrak{b}$ there is an $M(B)\in \mathfrak{a}$ as well as a morphism of DG-functors
$\mathfrak{a}(M(B),-)\r M(B,-)$ (by enriched Yoneda this is the same as an element of $\xi_B\in Z^0 M(B,A)$)
such that for all $A'\in\mathfrak{a}$ the induced map $\mathfrak{a}(M(B),A')\r M(B,A')$
is a quasi-isomorphism. 
A co-quasi-functor $M$  induces an honest functor $H^0(\mathfrak{b})\r H^0(\mathfrak{a})$
 by sending $B\mapsto M(B)$ and a corresponding construction for morphisms. We denote this functor by $H^0(M)$.
\begin{lemmas} \label{lem:coquasi}
 If $M$ is a co-quasi-functor then $H^0(M)$ has a DG-lift.  
\end{lemmas}
\begin{proof} Let $\cc=\aa\coprod_{M}\bb$ be the  gluing of $\aa$ and $\bb$ along $M$ as in \cite[\S3.2]{Orlov5}.
Let $\cc'$ be the full DG-subcategory of $\cc$ spanned by the objects $(M(B),B,\xi_B)$ for $B\in \bb$.
From the fact that $M$ is a co-quasi-functor one deduces that the projection functor $\pr_{\bb}:\cc'\r \bb$ is a quasi-isomorphism.
Let $\overline{M}:\bb\r \aa$ be the morphism in $\Ho(\dgcat(\ZZ))$ given by the composition $\bb\xrightarrow{\pr_{\bb}^{-1}}\cc'\xrightarrow{\pr_{\aa}}\aa$. It is easy to see that $H^0(\overline{M})\cong H^0(M)$.
\end{proof}
\begin{lemmas}
\label{lem:dglift} Let $F$ be a left exact functor between Grothendieck categories $\Cscr\r\Dscr$. Equip
$D^+(\Cscr)$, $D^+(\Dscr)$ with their standard enhancements $D^+_{\dg}(\Cscr)$, $D^+_{\dg}(\Dscr)$ given by injective resolutions.
Then  the right derived functor
$
RF:D^+(\Cscr)\r D^+(\Dscr)
$
of $F$ has a DG-lift.
\end{lemmas}
\begin{proof}
  For each $C\in \Ob(D^+(\Cscr))$, $D\in \Ob(D^+(\Dscr))$ fix injective resolutions $I_C$, $I_D$. Then we may define a co-quasi-fuctor
$RF^{\dg}:D^+_{\dg}(\Cscr)\r D^+_{\dg}(\Dscr)$ 
by putting
$
RF_{\dg}(C,D)=\underline{\Hom}_{C(\Dscr)}(FI_C,I_D)
$.
It is easy to see that $H^0(RF_{\dg})=RF$.
It now suffices to invoke Lemma \ref{lem:coquasi}.
\end{proof}
\section{Proof of Theorem  \ref{ref-1.2-2}}
\label{ref-B-182}
Concerning (1), again the most general proof follows from Appendix \ref{sec:amnon} as follows. Let $B_1=\tau_{\leq 0}B$ and let $B_2$ be the pullback of the diagram 
\[\xymatrix{
& R \ar[d] \\
B_1 \ar[r] & H^0(B_1).
}\]
Then $H^0(B_2)=R$ and $H^i(B_2)=0$ for $i=-1,\ldots,-m$ and $i>0$. Let $\Proj R$ the category of projective $R$-modules. We can then define a functor $L:R\to D(B_2)$ sending $R$ to $B_2$, and then extend to 
\begin{equation}\label{L for B_2}
L:\Proj R \to D(B_2). 
\end{equation}
Moreover since $B$ is concentrated in nonpositive degrees we have a functor $H$ given by $-\Lotimes_{B_2}H^0(B_2):D(B_2) \to D(H^0(B_2))=D(R)$ which sends $L(R)$ to $R$ and hence $L(P)$ to $P$ for $P\in \Proj R$. In the same way as in Proposition \ref{ref-5.3.1-69}, by considering the good couple $\Ascr=\{\Sigma^n P |P\in L(\Proj R), n\leq m\}$ and $\Bscr=\{\Sigma^n P|P\in L(\Proj R), n \geq 0\}$ we obtain a functor $L:D^b(R)\to D(B_2)$ extending \eqref{L for B_2}. The result follows by composing with $-\Lotimes_{B_2} B$.

We now concentrate on (2). So we have an exact functor
\[
{L}:D^b(R)\r D(B)
\]
which sends $R$ to $B$ in a way that is compatible with the right $R$-action on both sides
in $D(B)$.
Assume now that ${L}$ is of the form $U\Lotimes_R -$ for $U$ an object
in $D(B\otimes_k R^\circ)$. We assume that $U$ represented by a
cofibrant object in $\underline{\Mod}(B\otimes_k R^\circ)$ also
denoted by $U$. 
Since by construction the isomorphism $U={L}(R)\cong B$
is compatible with the right $R$-action on both sides in $D(B)$, it follows
that $B$ as an object in $D(B)_{R^\circ}$ lifts to an object in $D(B\otimes_k R^\circ)$. In other
words, the obstructions $o_i(B)$ exhibited in \S\ref{ref-7.3-97} vanish, which by the
proof of Lemma \ref{ref-7.3.1-98} is the same as saying that there is a $B$-linear right  $A_\infty$-$R$-action
on $B$ lifting the right $R$-action on $H^\ast(B)$, i.e.\ there is a  $A_\infty$-morphism
$R^\circ\r \End_B(B)=B^\circ$, finishing the proof.
\section{Proof of Proposition \ref{ref-1.3-5}}
\label{sec:ref-1.3-5}
\subsection{The unital DG-hull of a strictly unital $A_\infty$-category}
\label{ref-C.1-183}
In this section we temporarily drop our blanket convention that
$A_\infty$-notions are automatically strictly unital.

If $\aa$ is a strictly
unital $A_\infty$-category then there exists a universal strictly unital
$A_\infty$ morphism $\psi_\aa: \aa\r U^u(\aa)$ to a DG-algebra \cite[p127]{Lefevre}. 

Concretely, $U^u(\aa)$ is a suitable quotient of the non-unital DG-hull $\Omega\mathbb{B} \aa$ with 
 identity  morphisms adjoined. 
Taking the quotient is necessary to make the adjoined identity morphisms
compatible with the ones in $\aa$. From this explicit description 
one shows easily that $\psi_\aa$ is a quasi-isomorphism and in particular one has 
equivalences of categories
\[
D_\infty(\aa)\cong D_\infty(U^u(\aa))\cong D(U^u(\aa))
\]
where on the right we have the usual derived category of a DG-algebra. The second 
equivalence
is \cite[Cor. 4.1.3.11]{Lefevre}.

Since $\Omega\mathbb{B}
\aa=T^c(\Sigma^{-1} T(\Sigma \aa))$ (tensor (co)categories without (co)unit)
we find that $U^u(\aa)$ is concentrated in degrees $\le 0$.
\subsection{The proof} 
\label{ref-C.2-184}
By Theorem \ref{ref-1.2-2} we have to compute the obstruction against lifting the
natural map $R\r H^\ast(R^{\dg}_\eta)$. By construction, there is a 
$A_\infty$-quasi-isomorphism $R_\eta\r R^{\dg}_\eta$. Hence, by Corollary \ref{ref-7.2.2-89}, it is
sufficient to compute the obstructions for lifting the natural map $f:R\r H^0(R_\eta)$. 
By Remark \ref{ref-7.2.5-94}, the first possible non-vanishing obstruction
is $o_{n}(f)=\eta$, and it is indeed non-vanishing since we have assumed
$\eta\neq 0$. So lifting is not possible, finishing the proof.

\section{(by Amnon Neeman)}
\label{sec:amnon}

\hyphenation{mon-o-mor-phism mon-o-mor-phisms fi-nitely ap-pen-dex
man-u-script man-u-scripts co-lim-it co-lim-its homo-mor-phism
homo-mor-phisms epi-mor-phism epi-mor-phisms}


\newcommand{\ie}                {\emph{i.e.,}\xspace}
\newcommand{\angles}[1]         {{\langle #1 \rangle}}

\newcommand{\thought}[1]{}
\renewcommand{\thought}[1]{ \textbf{[#1]}}

\newenvironment{roenumerate}{\begin{enumerate}[\upshape (i)]}{\end{enumerate}}
\newcommand{\ulp}{\textup{(}}
\newcommand{\urp}{\textup{)}}
\newcommand{\uc}{\textup{:}}
\newcommand{\iref}[1]{(\ref{#1})}

\newcommand{\df}[1]{\emph{#1}}

\newcommand{\iso}{\cong}

\newcommand{\period}    {{\makebox[0pt][l]{\hspace{2pt} .}}}

\newcommand\nc {\newcommand}
\newcommand\rnc{\renewcommand}

\newcount\blopone
\newcount\xone
\newcount\xtwo
\newcount\ytwo

\theoremstyle{plain}
\let\theorem\undefined
\let\notation\undefined
\let\remark\undefined
\let\lemma\undefined
\let\corollary\undefined
\let\example\undefined
\let\conclusion\undefined
\let\case\undefined
\let\proposition\undefined
\let\definition\undefined
\let\hypothesis\undefined
\let\comment\undefined
\makeatletter
\let\c@case\undefined
\makeatother
\let\cor\undefined
\newtheorem{theorem}{Theorem}[section]
\newtheorem{prop}[theorem]{Proposition}
\newtheorem{com}[theorem]{Comment}
\newtheorem{redu}[theorem]{Reduction}
\newtheorem{refinement}[theorem]{Refinement}
\newtheorem{summary}[theorem]{Summary}
\newtheorem{importnota}[theorem]{Important Notation}
\newtheorem{prblm}[theorem]{Problem}
\newtheorem{notation}[theorem]{Notation}
\newtheorem{defin}[theorem]{Definition}
\newtheorem{caution}[theorem]{Caution}
\newtheorem{remark}[theorem]{Remark}
\newtheorem{reminder}[theorem]{Reminder}
\newtheorem{illustration}[theorem]{Illustration}
\newtheorem{lemma}[theorem]{Lemma}
\newtheorem{construction}[theorem]{Construction}
\newtheorem{corollary}[theorem]{Corollary}
\newtheorem{example}[theorem]{Example}
\newtheorem{conclusion}[theorem]{Conclusion}
\newtheorem{triviality}[theorem]{Triviality}
\newtheorem{proto}[theorem]{Prototype Quasifibration}
\newtheorem{cauex}[theorem]{Cautionary Example}
\newtheorem{hypo}[theorem]{Hypothesis}
\newtheorem{subth}{ }[theorem]
\newtheorem{case}{Case}[theorem]
\newtheorem{ssubth}{ }[subth]
\newtheorem{facts}[theorem]{Facts}

\nc\tri[1]{\begin{triviality}
\label{#1}}
\nc\fac[1]{\begin{facts}
\label{#1}
\begin{em}}
\nc\cas[1]{\begin{case}
\label{#1}
\begin{em}}
\nc\rfn[1]{\begin{refinement}
\label{#1}}
\nc\prt[1]{\begin{proto}
\label{#1}}
\nc\lem[1]{\begin{lemma}
\label{#1}}
\nc\pro[1]{\begin{prop}
\label{#1}}
\nc\thm[1]{\begin{theorem}
\label{#1}}
\nc\cor[1]{\begin{corollary}
\label{#1}}
\nc\dfn[1]{\begin{defin}
\label{#1}}
\nc\sthm[1]{\begin{subth}
\label{#1}}
\nc\exm[1]{\begin{example}
\label{#1}
\begin{em}}
\nc\plm[1]{\begin{prblm}
\label{#1}
\begin{em}}
\nc\rmk[1]{\begin{remark}
\label{#1}
\begin{em}}
\nc\rmd[1]{\begin{reminder}
\label{#1}
\begin{em}}
\nc\ntn[1]{\begin{notation}
\label{#1}
\begin{em}}
\nc\smr[1]{\begin{summary}
\label{#1}
\begin{em}}
\nc\cau[1]{\begin{caution}
\label{#1}
\begin{em}}
\nc\hyp[1]{\begin{hypo}
\label{#1}}
\nc\imn[1]{\begin{importnota}
\label{#1}
\begin{em}}
\nc\rdn[1]{\begin{redu}
\label{#1}
\begin{em}}
\nc\cax[1]{\begin{cauex}
\label{#1}
\begin{em}}
\nc\cmt[1]{\begin{com}
\label{#1}
\begin{em}}
\nc\con[1]{\begin{construction}
\label{#1}
\begin{em}}
\nc\ill[1]{\begin{illustration}
\label{#1}
\begin{em}}
\nc\ssthm[1]{\begin{ssubth}
\label{#1}
\begin{em}}
\nc\cnc[1]{\begin{conclusion}
\label{#1}
\begin{em}}

\nc\elem{\end{lemma}}
\nc\erdn{\end{em}\end{redu}}
\nc\erfn{\end{refinement}}
\nc\eprt{\end{proto}}
\nc\ethm{\end{theorem}}
\nc\ecor{\end{corollary}}
\nc\edfn{\end{defin}}
\nc\esthm{\end{subth}}
\nc\epro{\end{prop}}
\nc\etri{\end{triviality}}
\nc\eexm{\end{em}
\end{example}}
\nc\ecmt{\end{em}
\end{com}}
\nc\efac{\end{em}
\end{facts}}
\nc\ermk{\end{em}
\end{remark}}
\nc\ermd{\end{em}
\end{reminder}}
\nc\eill{\end{em}
\end{illustration}}
\nc\eplm{\end{em}
\end{prblm}}
\nc\ecas{\end{em}
\end{case}}
\nc\ecau{\end{em}
\end{caution}}
\nc\ecax{\end{em}
\end{cauex}}
\nc\eimn{\end{em}
\end{importnota}}
\nc\entn{\end{em}
\end{notation}}
\nc\econ{\end{em}
\end{construction}}
\nc\esmr{\end{em}
\end{summary}}
\nc\ehyp{
\end{hypo}}
\nc\ecnc{\end{em}
\end{conclusion}}
\nc\essthm{\end{em}
\end{ssubth}}

\newenvironment{beweis}{\noindent{\bf Proof}:\ \ }{\hfill{$\Box$}\vskip 2mm }
\newenvironment{rem}{{\bf Remark}:}{\vskip 5mm }

\makeatletter
\let\remarks\@undefined
\let\endremarks\@undefined
\makeatother
\newenvironment{remarks}{{\bf Remarks}:\begin{enumerate}}{\end{enumerate}}
\makeatletter
\let\examples\@undefined
\let\endexamples\@undefined
\makeatother
\newenvironment{examples}{{\bf Examples}:\begin{enumerate}}{\end{enumerate}}
\newtheorem{proposition}[theorem]{Proposition}
\newtheorem{definition}[theorem]{Definition}
\newtheorem{pretheorem}[theorem]{Pretheorem}
\newtheorem{hypothesis}[theorem]{Hypothesis}
\nc\sst{\scriptstyle}
\newcommand{\comment}[1]{}
\newcommand{\ri}{\longrightarrow}
\newcommand{\sr}{\rightarrow}
\newcommand{\slft}{\leftarrow}
\newcommand{\zz}{{\mathbb Z}}
\newcommand{\zq}{{\mathbb Z}_{qfh}}
\newcommand{\nn}{{\mathbb N}}
\newcommand{\K}{{\mathbf K}}
\newcommand{\D}{{\mathbf D}}
\newcommand{\qq}{{\mathbb Q}}
\newcommand{\rr}{{\mathbb R}}
\newcommand{\C}{{\mathbf C}}
\newcommand{\pp}{{\mathbb P}}
\newcommand{\cp}{{\mathbb{CP}}}
\newcommand{\nq}{{\mathbb N}_{qfh}}
\newcommand{\oo}{\otimes}
\newcommand{\uu}{\underline}
\newcommand{\ih}{\uu{Hom}}
\newcommand{\af}{{\mathbb A}^1}
\newcommand{\A}{{\mathbb A}^2}
\newcommand{\an}{{\mathbb A}^{n+1}}
\newcommand{\ak}[1]{{\mathbb A}^{#1}}
\nc\op{^{\hbox{\rm\tiny op}}}
\nc\mth{^{\hbox{\rm\tiny th}}}

\nc\script{\mathscr}
\nc\z{\zeta}
\nc\bc{{\mathbb{BC}}}
\nc\ct{{\script T}}
\nc\cl{{\script L}}
\nc\cv{{\script V}}
\nc\ce{{\script E}}
\nc\cs{{\script S}}
\nc\car{{\script R}}
\let\cd\undefined
\nc\cd{{\script D}}
\let\cc\undefined
\nc\cc{{\script C}}
\nc\ca{{\script A}}
\nc\ci{{\script I}}
\nc\co{{\script O}}
\nc\cx{{\script X}}
\nc\cz{{\script Z}}
\let\ch\undefined
\nc\ch{{\script H}}
\nc\bd{\begin{description}}
\nc\ed{\end{description}}
\nc\ctob{{\script C}at\big(\ci^{op},\ca\big)}
\nc\clim{{\ds\mathop{\rm lim}_{\ds\longleftarrow}}\,}
\nc\climi{\clim^{\!i}\,}
\nc\climn{\clim^{\!n}\,}
\nc\colim{{\ds\mathop{\rm colim}_{\ds\la}}}
\nc\oa{\overline{\ca}}
\nc\s{\sigma}
\nc\ta{\tau}
\nc\os{\overline\sigma}
\nc\ot{\overline\tau}
\nc\T{\Sigma}
\nc\Tm{\Sigma^{-1}}
\nc\de[1]{{\mathop{\rm deg(#1)}}}
\nc\Ad[1]{\mathop{\rm Ad}(#1)}
\nc\ad[1]{\mathop{\rm ad}(#1)}
\nc\kth{{\it K}--theory}

\def\der #1 {D\left(#1\right)}
\nc\prf{\begin{proof}}
\nc\eprf{\end{proof}}
\nc\ds{\displaystyle}
\let\Tor\undefined
\nc\Tor{\text{\rm Tor}}

\nc\cb{{\script B}}
\nc\ab{{\script A}b}

\nc\be{\begin{roenumerate}}
\nc\ee{\end{roenumerate}}

\nc\csab{{\script C}at\big(\cs^{op},\ab\big)}
\nc\ctab{{\script C}at\Big({\{\ct^\alpha\}}^{op},\ab\Big)}
\nc\csex{{\script E}x\big(\cs^{op},\ab\big)}
\nc\ctex{{\script E}x\Big({\{\ct^\alpha\}}^{op},\ab\Big)}
\nc\sub{\qquad\subset\qquad}
\nc\ctr[1]{{\left.\ct\left(-,#1\right)\right|}_{\cs}}
\nc\ctrf[2]{{\left.\ct\left(#1,#2\right)\right|}_{\cs}}
\nc\Ctr[1]{{\left.\ct\left(-,#1\right)\right|}_{\ct^\alpha}}
\nc\Ctrf[2]{{\left.\ct\left(#1,#2\right)\right|}_{\ct^\alpha}}

\nc\la{\longrightarrow}
\nc\nin{\noindent}
\nc\cad[1]{\text{card}(#1)}
\nc\eq{\quad=\quad}
\nc\BA{\begin{array}{c}}
\nc\EA{\end{array}}
\nc\barr{
\[
\begin{array}{cccccccccccccccc}
}
\nc\earr{
\end{array}
\]
}
\nc\as[1]{{\langle S\rangle}^{#1}}
\nc\sh{\text{\it shift}}

\nc\yy[1]{{\left.\ct\left(-,#1\right)\right|}_{\ct^c}}
\nc\vrep[2]{{\left.\ct\left(#1,#2\right)\right|}_{\ct^\alpha}}
\nc\da{\downarrow}
\rnc\Hom{{\mathop{\rm Hom}}}
\rnc\HHom{{\script H}{\mathop{\rm om}}}
\rnc\RHom{{\script {RH}}{\mathop{\rm om}}}

\let\End\undefined
\nc\End{{\mathop{\rm End}}}
\let\Ext\undefined
\nc\Ext{{\mathop{\rm Ext}}}
\nc\mr\Modtc
\nc\PExt{{\mathop{\rm PExt}}}
\nc\stm{\text{\rm stmod}(kG)}
\nc\stM{\text{\rm StMod}(kG)}
\nc\e{\varepsilon}
\nc\p{\varphi}
\newcommand{\m}{\mathfrak{m}}
\newcommand{\Dqc}{{\mathbf D_{\text{\bf qc}}}}

\nc\rs{\s^{-1}A}
\nc\br{{\{\s^{-1}A\}}}
\nc\ra\ri
\def\TT\undefined
\nc\TT{\mathbf{T}}
\nc\Ss{\mathbf{S}}
\nc\LL{\mathbf{L}}
\nc\y[1]{\mathbf{y}#1}
\nc\x[1]{\mathbf{z}#1}
\rnc\Mod[1]{#1\text{--\rm Mod}}
\nc\MMod[1]{\text{\rm Mod--}#1}
\nc\Md {\ensuremath{\mathop{\textup{Mod}}}}
\rnc\mod[1]{#1\text{--\rm mod}}
\nc\Modtc{\Mod{\ct^c}}
\nc\pgldim[1]{\mathop{\rm pgldim}\,#1}
\nc\tf{{\rm [TR5]}}
\nc\tfs{{\rm [TR5$^*$]}}
\nc\Fun{\text{\rm Funct}(F\op,\ab)}
\nc\sym{\text{\rm Sym}}
\nc\sgn{\text{\rm sgn}}
\rnc\Pro{\text{\rm Prod}^{}_\alpha(F\op,\ab)}
\nc\Yt[1]{{\left.\Hom_\ct^{}\left(-,#1\right)\right|}_F^{}}
\nc\dl{\delta}
\rnc\Proj[1]{#1\text{--\rm Proj}}
\rnc\proj[1]{#1\text{--\rm proj}}
\nc\Flat[1]{#1\text{--\rm Flat}}
\rnc\Inj[1]{#1\text{--\rm Inj}}
\nc\ov{\overline}

\nc\wt{\widetilde}
\nc\wh{\widehat}

\nc\hoco{
\begin{picture}(40,10)
\put(20,0){\makebox(0,0)[b]{\text{\rm Hocolim}}}
\put(5,-2){\vector(1,0){30}}
\end{picture}\,\,}

\nc\holim{
\begin{picture}(40,10)
\put(20,0){\makebox(0,0)[b]{\text{\rm Holim}}}
\put(35,-2){\vector(-1,0){30}}
\end{picture}}
\nc\ph{\varphi}
\nc\tstr{{\it t}--structure}
\rnc\id{\text{\rm id}}

\subsection{Some basic facts about \tstr{s}}
\label{SA1}

\lem{LA1.-1}
Let $\ct$ be a triangulated category with a \tstr, suppose we are given
in $\ct$ a morphism of triangles
\[
\xymatrix{
X\ar[r]^u \ar[d]^0 &Y\ar[r]^v \ar[d]^g &Z\ar[r]^w \ar[d]^0 &X[1]\ar[d]^{0} \\ 
X'\ar[r]^{u'} & Y'\ar[r]^{v'} &Z'\ar[r]^{w'} &X'[1]
}\]
and assume $X\in\ct^{\leq0}$ and $Z'\in\ct^{\geq0}$. Then there exists
$\theta:Z\la X'$ with $g=u'\theta v$. 
\elem

\prf
Because $gu=0$ the map $g$ must factor as $hv$ for some $h:Z\la Y'$. But then
$0=v'g=v'hv$, and $v'h$ must factor as $kw$ for some $k:X[1]\la Z'$. Since
$X[1]\in\ct^{\leq-1}$ and $Z'\in\ct^{\geq0}$ we conclude that the map $k$ must
vanish, hence $v'h=0$. Therefore $h$ must factor as $u'\theta$ for some
$\theta:Z\la X'$, and $g=hv=u'\theta v$.
\eprf

\lem{LA1.-2}
As in Lemma~\ref{LA1.-1} let $\ct$ be a triangulated category with a \tstr.
Assume we are given two triangles
\[
\xymatrix@R-20pt{
X\ar[r]^u  &Y\ar[r]^v  &Z\ar[r]^w  &X[1] \\
X'\ar[r]^{u'} & Y'\ar[r]^{v'} &Z'\ar[r]^{w'} &X'[1]
}\]
with $Y\in\ct^{\leq0}$ and $Z'\in\ct^{\geq0}$. If $\theta:Z\la X'$ is a map
such that $u'\theta v=0$ then there exists a morphism
$\s:X[1]\la X'$ with $\theta=\s w$.
\elem

\prf
We are given that $u'\theta v=0$, hence $\theta v$ must factor as
$\theta v=w'[-1]\rho$ for some $\rho:Y\la Z'[-1]$. But $Y\in\ct^{\leq0}$ and
$Z'[-1]\in\ct^{\geq1}$, hence $\rho$ must vanish. Therefore so does
$\theta v=w'[-1]\rho$, and we conclude that $\theta$ factors
as $\theta=\s w$ for some $\s:X[1]\la X'$.
\eprf

\subsection{Main results}
\label{S1}

\rmd{R1.0}
We adopt the notation first introduced in Be{\u\i}linson, Bernstein
and Deligne~\cite[1.3.9]{BeiBerDel82}. If $\ct$ is a triangulated
category and $\cx,\cz$ are full subcategories, then the full subcategory
$\cx*\cz$ has for objects all the $y\in\ct$ for which there exists a
triangle $x\la y\la z$ with $x\in\cx$ and $z\in\cz$.
\ermd

\dfn{D1.1}
Let $H:\car\la\ct$ be a triangulated functor between triagulated
categories. The pair of full subcategories $(\ca\subset\car,\cb\subset\car)$ is called a
\emph{good couple with respect to $H$} if
\be
\item
$\ca[-1]\subset\ca$ and $\cb[1]\subset\cb$.
\item
The map $\car(a,b)\la\ct(Ha,Hb)$ is an isomorphism if $a\in\ca$ and
$b\in\cb$, and is surjective if $a\in\ca$ and $b\in\cb[-1]$.
\setcounter{enumiv}{\value{enumi}}
\ee
The good couple $(\ca,\cb)$ is called \emph{excellent} if, in addition
to (i) and (ii) above, we have
\be
\setcounter{enumi}{\value{enumiv}}
\item
$\ca*\ca\subset\ca$ and $\cb*\cb\subset\cb$.
\ee
\edfn

\rmk{R1.2}
We note the easy facts
\be
\item
If $(\ca,\cb)$ is a good couple for $H$,
and $\ca'\subset\ca$ and $\cb'\subset\cb$ are full subcategories
satisfying $\ca'[-1]\subset\ca'$ and $\cb'[1]\subset\cb'$,
then  $(\ca',\cb')$ is also a
good couple for $H$. In this situation we will say that the good
couple $(\ca',\cb')$ is contained in the good couple $(\ca,\cb)$.
\item
 If $(\ca,\cb)$ is a good couple
 for $H$, then the restriction of $H$ to $\ca\cap\cb\subset\car$ is
 fully faithful.
\ee
\ermk

\lem{L1.3}
If $(\ca,\cb)$ is a good couple with respect to $H$ then so are
the couples $(\ca*\ca,\cb)$ and $(\ca,\cb*\cb)$.
\elem

\prf
It is enough to prove that $(\ca,\cb*\cb)$ is a good couple, the
statement about $(\ca*\ca,\cb)$ is obtained by applying this first
case
to the functor $H\op:\car\op\la\ct\op$ and the good couple
$(\cb\op\subset\car\op,\ca\op\subset\car\op)$.

The fact that $\ca[-1]\subset\ca$ is given, while
$(\cb*\cb)[1]=\cb[1]*\cb[1]\subset\cb*\cb$ is obvious.
Suppose $\wt b\in(\cb*\cb)[-1]$; then there exists a triangle in $b\la\wt
b\la b'\la $ with $b,b'\in\cb[-1]$. Let $a\in\ca$ be an object, then we
have a commutative diagram where the rows are exact
\[\xymatrix{
\car(a,b)\ar[d]^\beta \ar[r]&\car(a,\wt b)\ar[d]^\gamma \ar[r]&\car(a,b')\ar[d]^\delta\ar[r]
&\car(a,b[1])\ar[d]^\e \\
\ct(Ha,Hb) \ar[r] & \ct(Ha,H\wt b) \ar[r] & \ct(Ha,Hb') \ar[r] & \ct(Ha,Hb[1]) 
}\]
Since $b[1]$ belongs to $\cb$ the map
$\e$ is an isomorphism, while $\beta$ and $\delta$
are surjective. Hence $\gamma$ is surjective. 

Now suppose $\wt b\in\cb*\cb$. Then there exists a triangle
$b\la\wt
b\la b'\la $ with $b,b'\in\cb$. Let $a\in\ca$ be an object, then we
have a commutative diagram where the rows are exact 
\[\xymatrix{
\car(a,b'[-1])\ar[d]^\alpha \ar[r] &\car(a,b)\ar[d]^\beta \ar[r]&\car(a,\wt b)\ar[d]^\gamma \ar[r]&\car(a,b')\ar[d]^\delta\ar[r]
&\car(a,b[1])\ar[d]^\e \\
\ct(Ha,Hb'[-1]) \ar[r] & \ct(Ha,Hb) \ar[r] & \ct(Ha,H\wt b) \ar[r] & \ct(Ha,Hb') \ar[r] & \ct(Ha,Hb[1]) 
}\]
We know that $b,b'\in\cb$, and as $\cb[1]\subset\cb$ it follows that
also $b[1]\in\cb$. Therefore $\beta$, $\delta$ and $\e$ are
isomorphisms. Since $b[-1]\in\cb[-1]$ the map $\alpha$ is
surjective. The fine 5-lemma now tells us that $\gamma$ is an isomorphism.
\eprf

The following is now immediate

\cor{C1.5}
Every good couple $(\ca,\cb)$ is contained in an excellent couple.
In fact: the smallest excellent couple containing $(\ca,\cb)$ is
the pair $(\ca^*,\cb^*)$,
where $\cx^*$ is defined to be the union
\[
\cx^*=\bigcup_{n=1}^\infty \underbrace{\cx*\cx*\cdots*\cx}_{n \text{ times}}
\]
\ecor

\lem{L1.7}
Suppose $(\ca,\cb)$ is an excellent couple for $H$. Then the category
$\cc=H(\ca\cap\cb)$, the essential image of $\ca\cap\cb$
under $H$, satisfies $\cc*\cc\subset\cc$.
\elem

\prf
Let $c$ be an object in $\cc*\cc$. Then there exists in $\ct$ a triangle
$H(b[-1])\la c\la H(a) \stackrel g\la H(b)$ with $a,b[-1]$ both objects in
$\ca\cap\cb$.
In particular $a$ belongs to $\ca$ and $b[-1]$ belongs to $\cb$, but
as $\cb[1]\subset\cb$ we have that $b\in\cb$. Therefore
the map $\car(a,b)\la\ct(Ha,Hb)$ is an isomorphism, and hence
there exists a (unique) morphism $f:a\la b$ in $\car$ with $H(f)=g$.
Form in $\car$ the triangle $b[-1]\la \wt c\la a\stackrel f\la b$.
Then $\wt c$ belongs to $(\ca\cap\cb)*(\ca\cap\cb)\subset\ca\cap\cb$, and 
the functor  $H$ takes the triangle above to
$Hb[-1]\la H\wt c\la Ha\stackrel g\la Hb$. Hence $c\cong H\wt c$
with $\wt c\in\ca\cap\cb$.
\eprf

\cor{C22.5}
Suppose we are given a good couple $(\ca,\cb)$ and let $\cc$ be the essential
image of $\ca\cap\cb$ under the functor $H$.
Assume we are also given a subcategory
$\Dscr\subset\ct$, and suppose further that
every object in $\Dscr$ lies in $\cc^*$, where the notation $\cc^*$ is as
in Corollary~\ref{C1.5}.

Then there is an
excellent couple $(\ca',\cb')$, containing $(\ca,\cb)$, and such that
 $\Dscr$ lies in the essential image under $H$ of $\ca'\cap\cb'$.
\ecor

\prf
From Corollary~\ref{C1.5} it follows that the good couple $(\ca,\cb)$
may be included in an excellent couple $(\ca',\cb')$.
Let $\cc'\subset\ct$ be the essential image of $\ca'\cap\cb'$ under $H$;
then $\cc$ is clearly contained in $\cc'$, 
Lemma~\ref{L1.7} informs us that $\cc'*\cc'\subset\cc'$, and hence $\cc'$
contains $\cc^*$ which contains $\Dscr$.
\eprf

Now for the main result.

\thm{T1.13}
Let $H:\car\la\ct$ be a triangulated functor. Assume the
category $\car$ satisfies the axioms of 
the article \cite{Neeman91}. Suppose further that
$\ct$ has a non-degenerate
\tstr\ with heart $\ct^\heartsuit$, let $\ch:\ct\la\ct^\heartsuit$
be the
standard homological functor from $\ct$ to the heart, and
let $\ca\subset\ct^\heartsuit$ be a
full, abelian subcategory closed under extensions.
Assume $(\ca',\cb')$ is an excellent couple such that $\ca$ is contained
in the essential image of $\ca'\cap\cb'$.

Then there exists a triangulated functor $G:\ct^b_\ca\la\car$, where
$\ct^b_\ca\subset\ct$ is the full subcategory defined by
\[
\ct^b_\ca\eq\left\{t\in\ct\left|
\begin{array}{c}
\ch^i(t)=0\text{ for all but finitely many }i\in\zz\\
\ch^i(t)\in\ca\text{ for every }i\in\zz  
\end{array}
\right.
\right\}
\]
and such that the composite $HG:\ct^b_\ca\la\ct$ is naturally
isomorphic to the 
inclusion.

More precisely:
if we let $\ct^{[m,n]}_\ca=\ct^{\leq n}\cap\ct^{\geq m}\cap\ct^b_\ca$, our
construction will be such that $G(\ct^{[m,n]}_\ca)\subset\ca'[-m]\cap\cb'[-n]$.
\ethm

\prf
On the category $\ca=\ct_\ca^{[0,0]}$ we have little choice: we are looking for
an additive functor $G:\ca\la\ca'\cap\cb'$ so that the composite
$\ca\stackrel G\la \ca'\cap\cb'\hookrightarrow\car\stackrel H\la \ct$
is isomorphic
to the inclusion. But $H$ is fully faithful on $\ca'\cap\cb'$
by Remark~\ref{R1.2}(ii), and the essential image $H(\ca'\cap\cb')$ contains
$\ca$ by hypothesis. To define $G$ on $\ca$ we just choose a quasi-inverse,
and let $\ph:I\la HG$ be the natural isomorphism of $HG$ with the inclusion
functor.
Since we want $G$ and $\ph$ to be compatible with the shift this
determines $G$ and $\ph$ on $\ct_\ca^{[m,m]}=\ca[-m]$ for every integer $m$.

The strategy will be to prove, by induction on $n-m$, that the additive
functor $G$ and
the natural isomorphism $\ph$ may be
extended to $G:\ct_\ca^{[m,n]}\la\car$, compatibly with the shift. We have
proved the case $n-m=0$, and it remains to extend from $[m,n-1]$ to $[m,n]$.
By shifting we may assume $n=0$, that is for $m\leq-1$
we extend from $[m,-1]$ to
$[m,0]$.
It will be handy in the induction to note the following little fact.

\be
\item
Suppose the additive functor $G:\ct_\ca^{[m,n]}\la\ca'[-m]\cap\cb'[-n]$ has been
defined, as has the natural isomorphism $\ph:I\la HG$.
Then for $X$, $Y$ objects of $\ct_\ca^{[m,n]}$ we have that any morphism 
$\beta:HG(X)\la HG(Y)$ is equal to $HG(b)$ for some morphism $b:X\la
Y$.
If, for some integer $i$ with $m\leq i\leq n$, we have that $X$
belongs to $\ct_\ca^{[i,n]}$ and $Y$ belongs to $\ct_\ca^{[m,i]}$,
 then any morphism $\gamma:G(X)\la G(Y)$ is equal to
$G(g)$ for some $g:X\la Y$.
\setcounter{enumiv}{\value{enumi}}
\ee

\medskip

\nin
\emph{Proof of }(i).\ \ 
Because $\ph:I\la HG$ is a natural transformation we have, for any morphism
$f:X\la Y$ in $\ct_\ca^{[m,n]}$, the commutative square
\[\xymatrix@C+20pt{
  X \ar[r]^-{f} \ar[d]_{\ph_X^{}}& Y\ar[d]^{\ph_Y^{}} \\
  HG(X) \ar[r]^-{HG(f)} & HG(Y)
}\]
Applying this to the morphism $f=\ph_X^{}:X\la HG(X)$ we deduce the commutativity
of
\[\xymatrix@C+20pt{
  X \ar[r]^-{\ph_X^{}} \ar[d]_{\ph_X^{}}& HG(X)\ar[d]^{\ph_{HG(X)}^{}} \\
  HG(X) \ar[r]^-{HG(\ph_X^{})} & HGHG(X)
}\]
In other words the two composites
\[\xymatrix@C+20pt{
  X \ar[r]^-{\ph_X^{}}  & HG(X)\ar@<0.5ex>[r]^{\ph_{HG(X)}^{}} \ar@<-0.5ex>[r]_{HG(\ph_{X}^{})} & HGHG(X)
}\]
are equal. Since $\ph_X^{}:X\la HG(X)$ is an isomorphism we deduce
that $\ph_{HG(X)}^{}=HG(\ph_{X}^{})$ are equal maps $HG(X)\la HGHG(X)$.

In view of the above the commutative square
\[\xymatrix@C+20pt{
  HG(X) \ar[r]^-{\beta} \ar[d]_{\ph_{HG(X)}^{}}& HG(Y)\ar[d]^{\ph_{HG(Y)}^{}} \\
  HGHG(X) \ar[r]^-{HG(\beta)} & HGHG(Y)
}\]
may be rewritten as
\[\xymatrix@C+20pt{
  HG(X) \ar[r]^-{\beta} \ar[d]_{HG(\ph_X^{})}& HG(Y)\ar[d]^{HG(\ph_{Y}^{})} \\
  HGHG(X) \ar[r]^-{HG(\beta)} & HGHG(Y)
}\]
in other words if $b= \ph_Y^{-1}\beta\ph_X^{}$ then $\beta=HG(b)$.

Now suppose we are given a map $\gamma:G(X)\la G(Y)$.
Applying the previous assertion to $H(\gamma):HG(X)\la HG(Y)$ we learn
that
there is a map $g:X\la Y$ with $H(\gamma)=HG(g)$. Hence $H$ takes
the map $\gamma-G(g)$ to zero. But $\gamma-G(g)$ is a morphism
$G(X)\la G(Y)$, and as $X\in\ct_\ca^{[i,n]}$ we have $G(X)\in\ca'[-i]$
while $Y\in\ct_\ca^{[m,i]}$ implies that $G(Y)\in\cb'[-i]$. The
fact that $H$ annihilates $\gamma-G(g)$ therefore means $\gamma-G(g)=0$.
 \hfill{$\Box$}

\medskip

The preparatory result being proved, it's time to extend $G$ and $\ph$
from
$\ct_\ca^{[m,-1]}$ to $\ct_\ca^{[m,0]}$.
Let us begin with objects: assume $Y$ is an object in $\ct_\ca^{[m,0]}$. The
\tstr\ gives us a triangle
$Y^{\leq-1}\stackrel u\la Y\stackrel v\la Y^{\geq0}\stackrel w\la Y^{\leq-1}[1]$ in $\ct_\ca^b$,
with $Y^{\leq-1}\in\ct_\ca^{[m,-1]}$ and $Y^{\geq0}\in\ct_\ca^{[0,0]}$. By induction
we have already defined $G(Y^{\leq-1})\in\ca'[-m]\cap\cb'[1]$ and
$G(Y^{\geq0})\in\ca'\cap\cb'$.
The triangle gives a map $w:Y^{\geq0}\la Y^{\leq-1}[1]$, and induction gives
isomorphisms
$\ph_{Y^{\geq0}}^{}:Y^{\geq0}\cong HG(Y^{\geq0})$ and $\ph_{Y^{\leq-1}}^{}:Y^{\leq-1}\cong HG(Y^{\leq-1})$.
This permits us to form the commutative square
\[
\xymatrix@C+30pt{
Y^{\geq0}\ar[r]^-w \ar[d]_{\ph_{Y^{\geq0}}^{}} &Y^{\leq-1}[1]\ar[d]^{\ph_{Y^{\leq-1}}^{}[1]} \\ 
HG(Y^{\geq0})\ar[r]^-{\wh w} &HG(Y^{\leq-1})[1]
}\]
In other words we define $\wh w$ to be the composite making the square
commute.
On the other hand
$G(Y^{\geq0})\in\ca'$ and
$G(Y^{\leq-1})[1]\in\cb'[2]\subset\cb'$, and this implies that the map
\[\CD
\car\Big(G(Y^{\geq0}),G(Y^{\leq-1})[1]\Big)@>>>\ct\Big(HG(Y^{\geq0}),HG(Y^{\leq-1})[1]\Big)
\endCD\]
is an isomorphism. There is a unique
$\wt w:G(Y^{\geq0})\la G(Y^{\leq-1})[1]$ with $H(\wt w)=\wh w$.

Let $\cs$ be the category of triangles in $\car$
in sense of \cite[Axiom~3.4]{Neeman91}. By
\cite[Axiom~3.4(GTR4) and (GTR6)]{Neeman91}
we may choose an object $S\in\cs$
so that $F(S)$ is a candidate triangle where the third morphism is
$\wt w:G(Y^{\geq0})\la G(Y^{\leq-1})[1]$. Choose and fix such an $S=S(Y)$
for every object $Y\in\ct^{[m,0]}$, and declare $F(S(Y))$ to be
$G(Y^{\leq-1})\stackrel{\wt u}\la G(Y)\stackrel{\wt v}\la G(Y^{\geq0})\stackrel {\wt w}\la G(Y^{\leq-1})[1]$.
In other words we define $G(Y)$ to be the third edge of a triangle
on $\wt w$; but for the sake of future definitions we keep track,
in the enriched category of triangles, of the entire triangle $S(Y)$ defining
$G(Y)$.
Our first observation is that, since
$G(Y^{\leq-1})\in\ca'[-m]\cap\cb'[1]\subset\ca'[-m]\cap\cb'$ and
$G(Y^{\geq0})\in\ca'\cap\cb'\subset\ca'[-m]\cap\cb'$, we have
that $G(Y)$ belongs to
$(\ca'[-m]\cap\cb')*(\ca'[-m]\cap\cb')\subset \ca'[-m]\cap\cb'$.

We are also given, in the category $\ct_\ca^b$, a commutative diagram
where the rows are triangles
\[
\xymatrix@C+20pt{
Y^{\leq-1}\ar[r]^u  &Y\ar[r]^v  &Y^{\geq0}\ar[r]^-w \ar[d]^{\ph_{Y^{\geq0}}^{}} &Y^{\leq-1}[1]\ar[d]^{\ph_{Y^{\leq-1}}^{}[1]} \\ 
HG(Y^{\leq-1})\ar[r]^{H(\wt u)} & HG(Y)\ar[r]^{H(\wt v)} &HG(Y^{\geq0})\ar[r]^-{H(\wt w)} &HG(Y^{\leq-1})[1]
}\]
which may be extended, in the category $\ct_\ca^b$, to a morphism of triangles
\[
\xymatrix@C+20pt{
Y^{\leq-1}\ar[r]^u \ar[d]^{\ph_{Y^{\leq-1}}^{}} &Y\ar[r]^v \ar[d]^{\ph_Y^{}} &Y^{\geq0}\ar[r]^-w \ar[d]^{\ph_{Y^{\geq0}}^{}} &Y^{\leq-1}[1]\ar[d]^{\ph_{Y^{\leq-1}}^{}[1]} \\ 
HG(Y^{\leq-1})\ar[r]^{H(\wt u)} & HG(Y)\ar[r]^{H(\wt v)} &HG(Y^{\geq0})\ar[r]^-{H(\wt w)} &HG(Y^{\leq-1})[1]
}\]
Fix such a $\ph_Y^{}$. Since $\ph_{Y^{\leq-1}}^{}$ and $\ph_{Y^{\geq0}}^{}$ are
both isomorphisms so is $\ph_Y^{}$. For every object $Y\in\ct^{[m,0]}$
we have defined the object
$G(Y)$ and the isomorphism $\ph_Y^{}:Y\la HG(Y)$. It remains to define the
functor $G$ on morphisms. As we will see below we are done making choices,
the rest of the construction will be forced on us.

One note: 
If $Y\in\ct_\ca^{[m,-1]}\subset\ct_\ca^{[m,0]}$ then our choice of triangle
$Y^{\leq-1}\stackrel u\la Y\stackrel v\la Y^{\geq0}\stackrel w\la
Y^{\leq-1}[1]$ in $\ct_\ca^b$
will be 
$Y\stackrel \id\la Y\la 0\la
Y[1]$, and $S(Y)$ will be the
unique object in $\cs$ with $F(S(Y))$ being $G(Y)\stackrel\id\la G(Y) \la
0\la G(Y)[1]$. If $Y\in\ca=\ct_\ca^{[0,0]}\subset \ct_\ca^{[m,0]}$ then
our choice 
 of triangle
$Y^{\leq-1}\stackrel u\la Y\stackrel v\la Y^{\geq0}\stackrel w\la
Y^{\leq-1}[1]$ in $\ct_\ca^b$ will be 
$0\la Y\stackrel \id\la Y\la 0$, and 
$S(Y)$ is
the unique object in $\cs$ with $F(S(Y))$ being $0\la G(Y)\stackrel\id\la
G(Y)\la 0$. 
For the sake
of compatibility with earlier constructions, we also make sure that on the
category $\ct_\ca^{[m+1,0]}$ our choices are the same as they were when we were
dealing with extending from intervals of length $-m-2$ to
intervals of length $-m-1$. 

Suppose next that we are given in $\ct_\ca^{[m,0]}$ a morphism $g:Y\la Z$.
The construction above gives, in the category $\cs$, the
enriched triangles $S(Y)$ and $S(Z)$, with $F(S(Y))$ and $F(S(Z))$
being
\[
\xymatrix@R-20pt{
  G(Y^{\leq-1})\ar[r]^-{\wt u} &  G(Y)\ar[r]^-{\wt v} &G(Y^{\geq0})\ar[r]^-{\wt w} &G(Y^{\leq-1})[1] \\
  G(Z^{\leq-1})\ar[r]^-{\wt u'} &  G(Z)\ar[r]^-{\wt v'} &G(Z^{\geq0})\ar[r]^-{\wt w'} &G(Z^{\leq-1})[1]
}\]
Induction gives us the vertical maps in the square below
\[
\xymatrix@C+10pt{
  G(Y^{\geq0})\ar[r]^-{\wt w}\ar[d]_{G(g^{\geq0})}\ar@{}[dr]|-{\spadesuit} &G(Y^{\leq-1})[1]\ar[d]^{G(g^{\leq-1})[1]} \\
  G(Z^{\geq0})\ar[r]^-{\wt w'} &G(Z^{\leq-1})[1]
}\]
and we would like to show that the square $\spadesuit$ commutes.
But $G(Y^{\geq0})$ belongs
to $\ca'$ while $G(Z^{\leq-1})[1])$ belongs to
$\cb'[2]\subset\cb'$, hence it suffices to show that
the two composites become equal after applying the functor $H$. But
the squares
\[
\xymatrix@C+20pt{
  Y^{\geq0}\ar[r]^-{\ph_{Y^{\geq0}}^{}}\ar[d]^{g^{\geq0}} & HG(Y^{\geq0})\ar[d]^{HG(g^{\geq0})} &
  Y^{\leq-1}[1]\ar[r]^-{\ph_{Y^{\leq-1}}^{}[1]}\ar[d]^{g^{\leq-1}[1]} & HG(Y^{\leq-1})[1]\ar[d]^{HG(g^{\leq-1})[1]} \\
Z^{\geq0}\ar[r]^-{\ph_{Z^{\geq0}}^{}} &  HG(Z^{\geq0}) &Z^{\leq-1}[1]\ar[r]^-{\ph_{Z^{\leq-1}}^{}[1]} & HG(Z^{\leq-1})[1]
}\]
commute by induction, more precisely
by the naturality of $\ph$ on objects of length $<-m$.
And the squares 
\[
\xymatrix@C+10pt{
  Y^{\geq0}\ar[r]^-w \ar[d]^{\ph_{Y^{\geq0}}^{}} &Y^{\leq-1}[1]\ar[d]^{\ph_{Y^{\leq-1}}^{}[1]} &
  Z^{\geq0}\ar[r]^-{w'} \ar[d]^{\ph_{Z^{\geq0}}^{}} &Z^{\leq-1}[1]\ar[d]^{\ph_{Z^{\leq-1}}^{}[1]} \\ 
HG(Y^{\geq0})\ar[r]^-{H(\wt w)} &HG(Y^{\leq-1})[1] &
HG(Z^{\geq0})\ar[r]^-{H(\wt w')} &HG(Z^{\leq-1})[1]
}\]
commute by the construction of the maps $\wt w:G(Y^{\geq0})\la G(Y^{\leq-1})[1]$
and $\wt w':G(Z^{\geq0})\la G(Z^{\leq-1})[1]$ above. We deduce that $H(\spadesuit)$
is isomorphic to the obviously commutative square
\[
\xymatrix@C+10pt{
  Y^{\geq0}\ar[r]^-{w}\ar[d]_{g^{\geq0}} &Y^{\leq-1}[1]\ar[d]^{g^{\leq-1}[1]} \\
  Z^{\geq0}\ar[r]^-{w'} &Z^{\leq-1}[1]
}\]
Hence the square $\spadesuit$ does commute. Next we will prove
\be
\setcounter{enumi}{\value{enumiv}}
 \item
There is a unique morphism $\wt k:S(Y)\la S(Z)$, in the category $\cs$,
  so that $F(\wt k)$ is a morphism of candidate triangles
  in $\car$
  \[
\xymatrix@C+20pt{
  G(Y^{\leq-1})\ar[r]^-{\wt u}\ar[d]^{G(g^{\leq-1})} &  G(Y)\ar[r]^-{\wt v}\ar[d]^{k} &G(Y^{\geq0})\ar[d]_{G(g^{\geq0})}\ar[r]^-{\wt w}\ar@{}[dr]|{\spadesuit} &G(Y^{\leq-1})[1]\ar[d]^{G(g^{\leq-1})[1]} \\
  G(Z^{\leq-1})\ar[r]^-{\wt u'} &  G(Z)\ar[r]^-{\wt v'} &G(Z^{\geq0})\ar[r]^-{\wt w'} &G(Z^{\leq-1})[1]
}\]
and the square
  \[\xymatrix@C+20pt{
    Y\ar[r]^-g\ar[d]_{\ph_Y^{}} & Z\ar[d]^{\ph_Z^{}} \\
    HG(Y)\ar[r]^-{H(k)} & HG(Z)
  }\]
commutes in $\ct_\ca^b$.
\setcounter{enumiv}{\value{enumi}}
\ee

\medskip

\nin
\emph{Proof of\ }(ii).\ \
We begin with the proof of existence. 
From 
\cite[Axiom~3.4(GTR5)
  and (GTR6)]{Neeman91} we may extend the commutative square
$\spadesuit$ to a morphism of
triangles. That is
  there exists
in the category $\cs$ a morphism $\wt h:S(Y)\la SZ)$ so that
$F(\wt h)$ is a map
\[
\xymatrix{
  G(Y^{\leq-1})\ar[r]^-{\wt u}\ar[d]^{G(g^{\leq-1})} &  G(Y)\ar[r]^-{\wt v}\ar[d]^{h} &G(Y^{\geq0})\ar[d]^{G(g^{\geq0})}\ar[r]^-{\wt w} &G(Y^{\leq-1})[1]\ar[d]^{G(g^{\leq-1})[1]} \\
  G(Z^{\leq-1})\ar[r]^-{\wt u'} &  G(Z)\ar[r]^-{\wt v'} &G(Z^{\geq0})\ar[r]^-{\wt w'} &G(Z^{\leq-1})[1]
}\]
Applying the functor $H$ we obtain in $\ct_\ca^b$ the morphism of triangles
\[\xymatrix@C+20pt{
  HG(Y^{\leq-1})\ar[r]^-{H(\wt u)}\ar[d]^{HG(g^{\leq-1})} &  HG(Y)\ar[r]^-{H(\wt v)}\ar[d]^{H(h)} &HG(Y^{\geq0})\ar[d]^{HG(g^{\geq0})}\ar[r]^-{H(\wt w)} &HG(Y^{\leq-1})[1]\ar[d]^{HG(g^{\leq-1})[1]} \\
  HG(Z^{\leq-1})\ar[r]^-{H(\wt u')} &  HG(Z)\ar[r]^-{H(\wt v')} &HG(Z^{\geq0})\ar[r]^-{H(\wt w')} &HG(Z^{\leq-1})[1]
}\]
There is no reason for this morphism of triangles to agree with the composite
\[
\xymatrix@C+20pt{
  HG(Y^{\leq-1})\ar[r]^-{H(\wt u)} &  HG(Y)\ar[r]^-{H(\wt v)} 
  &HG(Y^{\geq0})\ar[r]^-{H(\wt w)} &HG(Y^{\leq-1})[1] \\
  Y^{\leq-1}\ar[r]^-u  \ar[u]_{\ph_{Y^{\leq-1}}^{}} \ar[d]^{g^{\leq-1}}
  &Y\ar[r]^-v \ar[d]^-{g} \ar[u]_{\ph_Y^{}}
  &Y^{\geq0}\ar[r]^-w \ar[d]^{g^{\geq0}}\ar[u]_{\ph_{Y^{\geq0}}^{}}
  &Y^{\leq-1}[1]\ar[d]^{g^{\leq-1}}[1]\ar[u]_{\ph_{Y^{\leq-1}}^{}[1]} \\ 
  Z^{\leq-1}\ar[r]^-{u'}  \ar[d]^{\ph_{Z^{\leq-1}}^{}}
  &Z\ar[r]^-{v'}  \ar[d]^{\ph_Z^{}} 
  &Z^{\geq0}\ar[r]^-{w'}\ar[d]^{\ph_{Z^{\geq0}}^{}}
  &Z^{\leq-1}[1]\ar[d]^{\ph_{Z^{\leq-1}}^{}}\\
    HG(Z^{\leq-1})\ar[r]^-{H(\wt u')} &  HG(Z)\ar[r]^-{H(\wt v')} &HG(Z^{\geq0})\ar[r]^-{H(\wt w')} &HG(Z^{\leq-1})[1]
}\]
but the difference is a morphism of triangles
\[\xymatrix@C+20pt{
  HG(Y^{\leq-1})\ar[r]^-{H(\wt u)}\ar[d]^{0} &
  HG(Y)\ar[r]^-{H(\wt v)}\ar[d]^{H(h)-\ph_Z^{}g\ph_Y^{-1}} &HG(Y^{\geq0})\ar[d]^{0}\ar[r]^-{H(\wt w)}
  &HG(Y^{\leq-1})[1]\ar[d]^{0} \\
  HG(Z^{\leq-1})\ar[r]^-{H(\wt u')} &  HG(Z)\ar[r]^-{H(\wt v')} &HG(Z^{\geq0})\ar[r]^-{H(\wt w')} &HG(Z^{\leq-1})[1]
}\]
Since $HG(Y^{\leq-1})\cong Y^{\leq-1}$ belongs to $\ct^{\leq-1}$ and
$HG(Z^{\geq0})\cong Z^{\geq0}$ belongs to $\ct^{\geq0}$ Lemma~\ref{LA1.-1} applies,
and tells us that there exists a morphism $\theta:HG(Y^{\geq0})\la HG(Z^{\leq-1})$
with $H(h)-\ph_Z^{}g\ph_Y^{-1}=H(\wt u')\theta H(\wt v)$. But now
$G(Y^{\geq0})$ belongs to $\ca'$ and $G(Z^{\leq-1})$ belongs
to $\cb'[1]\subset\cb'$, and hence the map
$\theta:HG(Y^{\geq0})\la HG(Z^{\leq-1})$ can be expressed (uniquely) as
$H(\rho)$ for a morphism $\rho:G(Y^{\leq0})\la G(Z^{\leq-1})$. By
\cite[Axiom~3.4(GTR2) and (GTR6)]{Neeman91} the morphisms of triangles 
$\wt h:S(Y)\la S(Z)$ in $\cs$, whose images under
the functor $F$ are of the form
\[
\xymatrix{
  G(Y^{\leq-1})\ar[r]^-{\wt u}\ar[d]^{G(g^{\leq-1})} &  G(Y)\ar[r]^-{\wt v}\ar[d]^{h} &G(Y^{\geq0})\ar[d]^{G(g^{\geq0})}\ar[r]^-{\wt w} &G(Y^{\leq-1})[1]\ar[d]^{G(g^{\leq-1})[1]} \\
  G(Z^{\leq-1})\ar[r]^-{\wt u'} &  G(Z)\ar[r]^-{\wt v'} &G(Z^{\geq0})\ar[r]^-{\wt w'} &G(Z^{\leq-1})[1]
}\]
are acted on transitively by the group
$\wt u'\circ\Hom\big(G(Y^{\geq0}),G(Z^{\leq-1})\big)\circ\wt v$.
Hence we may choose a morphism
of triangles $\wt k:S(Y)\la S(Z)$ whose image under $F$ is
\[
\xymatrix{
  G(Y^{\leq-1})\ar[r]^-{\wt u}\ar[d]^{G(g^{\leq-1})} &  G(Y)\ar[r]^-{\wt v}\ar[d]^{k=h-\wt u'\rho\wt v} &G(Y^{\geq0})\ar[d]^{G(g^{\geq0})}\ar[r]^-{\wt w} &G(Y^{\leq-1})[1]\ar[d]^{G(g^{\leq-1})[1]} \\
  G(Z^{\leq-1})\ar[r]^-{\wt u'} &  G(Z)\ar[r]^-{\wt v'} &G(Z^{\geq0})\ar[r]^-{\wt w'} &G(Z^{\leq-1})[1]
}\]
But
$H(k) =H(h-\wt u'\rho\wt v)=H(h)-H(\wt u')\theta H(\wt v)=\ph_Z^{}g\ph_Y^{-1}$.
Hence the diagram
  \[\xymatrix@C+20pt{
    Y\ar[r]^-g\ar[d]_{\ph_Y^{}} & Z\ar[d]^{\ph_Z^{}} \\
    HG(Y)\ar[r]^-{H(k)} & HG(Z)
  }\]
commutes, proving the existence.

Next we need to show the uniqueness.
Suppose $\wt h,\wt h'$ are two morphisms $S(Y)\la S(Z)$
as in (i): then $F$ takes $\wt h-\wt h'$ to
the morphism of triangles
\[
\xymatrix{
  G(Y^{\leq-1})\ar[r]^-{\wt u}\ar[d]^{0} &  G(Y)\ar[r]^-{\wt v}\ar[d]^{h-h'} &G(Y^{\geq0})\ar[d]^{0}\ar[r]^-{\wt w} &G(Y^{\leq-1})[1]\ar[d]^{0} \\
  G(Z^{\leq-1})\ar[r]^-{\wt u'} &  G(Z)\ar[r]^-{\wt v'} &G(Z^{\geq0})\ar[r]^-{\wt w'} &G(Z^{\leq-1})[1]
}\]
and \cite[Axioms~3.4 (GRT2) and (GTR6)]{Neeman91} guarantee that there exists
a $\rho:G(Y^{\geq0})\la G(Z^{\leq-1})$ with $h-h'=\wt u'\rho\wt v$. Since
$H(h)$ and $H(h')$ are both equal to
$\ph_Z^{}g\ph_Y^{-1}$ we have
$0=H(h-h')=H(\wt u')H(\rho)H(\wt v)$. But now the triangles in $\ct$
\[
\xymatrix@R-20pt{
  HG(Y^{\leq-1})\ar[r]^-{H(\wt u)} &  HG(Y)\ar[r]^-{H(\wt v)}
  &HG(Y^{\geq0})\ar[r]^-{H(\wt w)} &HG(Y^{\leq-1})[1] \\
  HG(Z^{\leq-1})\ar[r]^-{H(\wt u')} &  HG(Z)\ar[r]^-{H(\wt v')} &HG(Z^{\geq0})\ar[r]^-{H(\wt w')} &HG(Z^{\leq-1})[1]
}\]
are such that $HG(Y)\cong Y$ lies
in $\ct^{\leq0}$ and $HG(Z^{\geq0})\cong Z^{\geq0}$ belongs to $\ct^{\geq0}$.
From Lemma~\ref{LA1.-2} we learn that $H(\rho)$ must factor as
$\s H(\wt w)$, for some $\s:HG(Y^{\leq-1})[1]\la HG(Z^{\leq-1})$.
Let $\gamma:Z^{\leq-2}\la Z^{\leq-1}$ be the canonical \tstr\ map. It is a
morphism in $\ct_\ca^{[m,-1]}$, on which $G$ and the isomorphism $\ph:I\la
HG$ are already defined---hence $\gamma$ is isomorphic to
$HG(\gamma):HG(Z^{\leq-2})
\la HG(Z^{\leq-1})$. In particular $HG(\gamma)$ identifies as the map
from the \tstr\ truncation.
Therefore $\s:HG(Y^{\leq-1})[1]\la HG(Z^{\leq-1})$, which is a map
from $HG(Y^{\leq-1})[1]\cong Y^{\leq-1}[1]\in\ct^{\leq-2}$ to the
object $HG(Z^{\leq-1})$, must factor (uniquely) as
\[
\CD
HG(Y^{\leq-1})[1] @>\beta>> HG(Z^{\leq-2})
@>HG(\gamma)>>  HG(Z^{\leq-1})
\endCD
\]
The objects $Y^{\leq-1}[1]$ and $Z^{\leq-2}$ both belong to
$\ct_\ca^{[m-1,-2]}$, on which $G$ and the isomorphism $I\la HG$ are
already
defined. By (i) above there exists a morphism $b:Y^{\leq-1}[1]\la
Z^{\leq-2}$
with $\beta=HG(b)$. Define $\theta:G(Y^{\leq-1}[1])\la G(Z^{\leq-1})$
to be the composite
\[
\CD
G(Y^{\leq-1})[1] @>G(b)>> G(Z^{\leq-2})
@>G(\gamma)>>  G(Z^{\leq-1})
\endCD
\]
then $H(\theta)=\s$. Hence
$\rho-\theta\wt w$ satisfies the identity
$H(\rho-\theta\wt w)=H(\rho)-H(\theta)H(\wt w)=H(\rho)-\s H(\wt w)=0$.

But $\rho-\theta\wt w$ is a morphism $G(Y^{\geq0})\la G(Z^{\leq-1})$, with
$G(Y^{\geq0})\in\ca'$ and
$G(Z^{\leq-1})\in\cb'[1]\subset\cb'$. The fact that
$H(\rho-\theta\wt w)=0$ implies that $\rho-\theta\wt w=0$. It follows
that $h-h'=\wt u'\rho\wt v=\wt u'\theta\wt w\wt v=0$, where the
vanishing is because the composite $\wt w\wt v$ vanishes.\hfill{$\Box$}

\medskip

This completes the proof of (ii), and allows us to make the key definition
\be
\setcounter{enumi}{\value{enumiv}}
\item
  If $g:Y\la Z$ is a morphism in $\ct_\ca^{[m,0]}$, let $\wt g:S(Y)\la S(Z)$ be
  the unique morphism satisfying the conditions of (ii). We define
  $G(g):G(Y)\la G(Z)$ by letting $F(\wt g)$ be the morphism
  \[
\xymatrix@C+20pt{
  G(Y^{\leq-1})\ar[r]^-{\wt u}\ar[d]^{G(g^{\leq-1})} &  G(Y)\ar[r]^-{\wt v}\ar[d]^{G(g)} &G(Y^{\geq0})\ar[d]^{G(g^{\geq0})}\ar[r]^-{\wt w} &G(Y^{\leq-1})[1]\ar[d]^{G(g^{\leq-1})[1]} \\
  G(Z^{\leq-1})\ar[r]^-{\wt u'} &  G(Z)\ar[r]^-{\wt v'} &G(Z^{\geq0})\ar[r]^-{\wt w'} &G(Z^{\leq-1})[1]
}\]
The reader will note (ii) guarantees the commutativity of the square
 \[\xymatrix@C+20pt{
    Y\ar[r]^-g\ar[d]_{\ph_Y^{}} & Z\ar[d]^{\ph_Z^{}} \\
    HG(Y)\ar[r]^-{HG(g)} & HG(Z)
 }\]
which is precisely what we need to prove the naturality of $\ph$.
\setcounter{enumiv}{\value{enumi}}
\ee
The fact that $G(\id)=\id$, $G(g+h)=G(g)+G(h)$ and $G(hg)=G(h)G(g)$
are all immediate from the uniqueness proved in (ii).
We have extended the additive functor $G$
and the natural transformation $\ph:I\la HG$
from $\ct_\ca^{[m,-1]}$ to
$\ct_\ca^{[m,0]}$. It remains to show that
the functor $G$, which by induction has been extended to all of $\ct_\ca^b$,
is a triangulated functor. We need to prove that $G$ takes triangles to
triangles. The next little fact will help.

\be
\setcounter{enumi}{\value{enumiv}}
\item
  Suppose $X\stackrel u\la Y\stackrel v\la Z\stackrel w\la X[1]$ is
  a triangle in $\ct_\ca^b$. If we can exhibit in $\car$ some triangle of
  the form
  $G(X)\stackrel{G(u)}\la G(Y)\stackrel{G(v')}\la G(Z')\stackrel{G(w')}\la G(X)[1]$, then the sequence
  $G(X)\stackrel{G(u)}\la G(Y)\stackrel{G(v)}\la G(Z)\stackrel{G(w)}\la G(X)[1]$
  is also a triangle in $\car$.
\setcounter{enumiv}{\value{enumi}}
\ee

\medskip

\nin
\emph{Proof of }(iv).\ \ 
Consider the diagram
\[
\xymatrix@C+20pt{
  X\ar[r]^-{u}\ar[d]^{\ph_X^{}} & Y\ar[r]^-{v'}\ar[d]^{\ph_Y^{}} & Z'\ar[d]^{\ph_Z^{}}\ar[r]^-{w'} & X[1]\ar[d]^{\ph_X^{}[1]}\\
  HG(X)\ar[r]^-{HG(u)} & HG(Y)\ar[r]^-{HG(v')} &HG(Z') \ar[r]^-{HG(w')} & HG(X)[1]
}\]
which commutes by the naturality of $\ph$. The bottom row is obtained
by applying the triangulated functor $H$ to the triangle
$G(X)\stackrel{G(u)}\la G(Y)\stackrel{G(v')}\la G(Z')\stackrel{G(w')}\la G(X)[1]$
in the category $\car$. Hence the bottom row is a triangle. Since $\ph$
is an isomorphism the top row is isomorphic to the bottom row,
hence a triangle in $\ct_\ca^b$.
But in the diagram 
\[
\xymatrix@C+20pt{
  X\ar[r]^-{u}\ar@{=}[d] & Y\ar[r]^-{v'}\ar@{=}[d] & Z'\ar[r]^-{w'} & X[1]\\
X\ar[r]^-{u} & Y\ar[r]^-{v} & Z\ar[r]^-{w} & X[1]
}\]
the rows are triangles, hence we may extend to an isomorphism 
\[
\xymatrix@C+20pt{
  X\ar[r]^-{u}\ar@{=}[d] & Y\ar[r]^-{v'}\ar@{=}[d] & Z'\ar[d]^{\psi}\ar[r]^-{w'} & X[1]\ar@{=}[d]\\
 X\ar[r]^-{u} & Y\ar[r]^-{v} & Z\ar[r]^-{w} & X[1]
}\]
Applying the functor $G$ we have in $\car$ an isomorphism of the top and bottom rows
\[
\xymatrix@C+20pt{
  G(X)\ar[r]^-{G(u)}\ar@{=}[d] & G(Y)\ar[r]^-{G(v')}\ar@{=}[d] & G(Z')\ar[d]^{G(\psi)}\ar[r]^-{G(w)'} & G(X)[1]\ar@{=}[d]\\
 G(X)\ar[r]^-{G(u)} & G(Y)\ar[r]^-{G(v)} & G(Z)\ar[r]^-{G(w)} & G(X)[1]
}\]
Since the top row is a triangle so is the bottom.\hfill{$\Box$}

\medskip

The preliminaries are now out of the way, we have to prove
that $G$ takes triangles to triangles.
Let us begin with two easy special cases.

\be
\setcounter{enumi}{\value{enumiv}}
\item
  Let $Y$ be an object in $\ct^{\leq0}\cap\ct^b_\ca$ and consider the triangle
  $Y^{\leq-1}\stackrel u\la Y\stackrel v\la Y^{\geq0}\stackrel w\la Y^{\leq-1}[1]$.
  Then $G$ takes it to a triangle.
\setcounter{enumiv}{\value{enumi}}
\ee

\medskip

\nin
\emph{Proof of }(v).\ \ 
Choose an $m\leq0$ with $Y\in\ct_\ca^{[m,0]}$; then $Y^{\leq-1}$ and $Y^{\geq0}$ also
belong to $\ct_\ca^{[m,0]}$, hence we can work inside $\ct_\ca^{[m,0]}$ to compute
what $G$ does. Recalling
the construction that allowed us to extend
from $\ct_\ca^{[m,-1]}$ to $\ct_\ca^{[m,0]}$,
the triangles $S(Y^{\leq-1})$, $S(Y)$ and $S(Y^{\geq0})$ are such
that $F(S(Y^{\leq-1}))$, $F(S(Y))$ and $F(S(Y^{\geq0}))$ are the rows
in the diagram below
\[\xymatrix@C+10pt{
  G(Y^{\leq-1})\ar[r]^-\id & G(Y^{\leq-1})\ar[r] & 0 \ar[r]\ar[d] & G(Y^{\leq-1})[1]\ar[d]^\id\\
  G(Y^{\leq-1})\ar[r]^{\wt u} & G(Y)\ar[r]^{\wt v} & G(Y^{\geq0})\ar[d]^\id \ar[r]^{\wt w} & G(Y^{\leq-1})[1]\ar[d]\\
  0 \ar[r] & G(Y^{\geq0}) \ar[r]^-\id &G(Y^{\geq0}) \ar[r] & 0
}\]
and there is only one way to extend to morphisms of triangles. Thus
the complicated definition of (iii) specializes to
\[\xymatrix@C+10pt{
  G(Y^{\leq-1})\ar[r]^-\id\ar[d]^\id & G(Y^{\leq-1})\ar[r]\ar[d]^{G(u)=\wt u} & 0 \ar[r]\ar[d] & G(Y^{\leq-1})[1]\ar[d]^\id\\
  G(Y^{\leq-1})\ar[r]^{\wt u}\ar[d] & G(Y)\ar[r]^{\wt v}\ar[d]^{G(v)=\wt v} & G(Y^{\geq0})\ar[d]^\id \ar[r]^{\wt w} & G(Y^{\leq-1})[1]\ar[d]\\
  0 \ar[r] & G(Y^{\geq0}) \ar[r]^-\id &G(Y^{\geq0}) \ar[r] & 0
}\]
The morphisms $\wt w$ and $G(w)$ have the property
that $H(\wt w)=\ph_{Y^{\leq-1}}^{}[1]w\ph_{Y^{\geq0}}^{-1}=HG(w)$. Thus
$\wt w, G(w)$ are morphisms from $G(Y^{\geq0})\in\ca'$ to
$G(Y^{\leq-1})[1]\in\cb'$, whose images under $H$ agree---we
conclude that $G(w)=\wt w$. Therefore
$G$ takes the triangle 
$Y^{\leq-1}\stackrel u\la Y\stackrel v\la Y^{\geq0}\stackrel w\la Y^{\leq-1}[1]$
to $G(Y^{\leq-1})\stackrel{\wt u}\la G(Y)\stackrel{\wt v}\la G(Y^{\geq0})\stackrel {\wt w}\la G(Y^{\leq-1})[1]$, which is a triangle by construction.
\hfill{$\Box$}

\be
\setcounter{enumi}{\value{enumiv}}
\item
  Suppose we are given in $\ct_\ca^b$ a triangle
  $A\stackrel u\la B\stackrel v\la C\stackrel w\la A[1]$ with
  $A,B,C\in\ca=\ct_\ca^{[0,0]}$. Then $G$ takes it to a triangle.
\setcounter{enumiv}{\value{enumi}}
\ee

\medskip

\nin
\emph{Proof of }(vi).\ \ 
Complete $G(w):G(C)\la G(A)[1]$ to a triangle
$G(A)\stackrel{u'}\la Y\stackrel{v'}\la G(C)\stackrel{G(w)}\la G(A)[1]$.
Then $Y$ belongs to $(\ca'\cap\cb')*(\ca'\cap\cb')\subset\ca'\cap\cb'$. Now
consider the diagram
\[
\xymatrix{
  HG(A) \ar[r]^-{HG(u)} & HG(B)\ar[r]^-{HG(v)} &
  HG(C)\ar@{=}[d]\ar[r]^-{HG(w)} & HG(A)[1]\ar@{=}[d] \\
  HG(A)\ar[r]^-{H(u')} & H(Y) \ar[r]^-{H(v')} & HG(C)\ar[r]^-{HG(w)} &
  HG(A)[1]
}\]
The top row is isomorphic to
$A\stackrel u\la B\stackrel v\la C\stackrel w\la A[1]$,
hence a triangle in $\ct_\ca^b$. And the bottom row is obtained
by applying the functor $H$ to the triangle
$G(A)\stackrel{u'}\la Y\stackrel{v'}\la G(C)\stackrel{G(w)}\la G(A)[1]$.
Hence both rows are triangles and we may complete to a morphism
of triangles
\[
\xymatrix{
  HG(A)\ar@{=}[d] \ar[r]^-{HG(u)} & HG(B)\ar[r]^-{HG(v)} \ar[d]^\psi&
  HG(C)\ar@{=}[d]\ar[r]^-{HG(w)} & HG(A)[1]\ar@{=}[d] \\
  HG(A)\ar[r]^-{H(u')} & H(Y) \ar[r]^-{H(v')} & HG(C)\ar[r]^-{HG(w)} &
  HG(A)[1]
}\]
with $\psi$ an isomorphism.
Since the functor $H$ is fully faithful on $\ca'\cap\cb'$ we learn
first that the isomorphism $\psi:HG(B)\la H(Y)$ must be $H(\rho)$
for some isomorphism $\rho:G(B)\la Y$. But the
diagram
\[
\xymatrix{
  G(A)\ar@{=}[d] \ar[r]^-{G(u)} & G(B)\ar[r]^-{G(v)} \ar[d]^\rho&
  G(C)\ar@{=}[d] \\
  G(A)\ar[r]^-{u'} & Y \ar[r]^-{v'} & G(C)
}\]
is a diagram in $\ca'\cap\cb'$ whose image under $H$ commutes, hence the
diagram commutes. We conclude that 
$G(A)\stackrel {G(u)}\la G(B)\stackrel {G(v)}\la G(C)\stackrel{G(w)}\la G(A)[1]$
is
isomorphic in $\car$ to the triangle
$G(A)\stackrel{u'}\la Y\stackrel{v'}\la G(C)\stackrel{G(w)}\la G(A)[1]$,
hence is a triangle.\hfill{$\Box$}

\medskip

\be
\setcounter{enumi}{\value{enumiv}}
\item
  Suppose we are given in $\ct_\ca^b$ a triangle
  $X\stackrel u\la Y\stackrel v\la Z\stackrel w\la X[1]$ with
  $X,Y\in\ct^{\leq0}\cap\ct^b_\ca$ and $Z\in\ca=\ct_\ca^{[0,0]}$.
  Then $G$ takes it to a triangle.
\setcounter{enumiv}{\value{enumi}}
\ee

\medskip

\emph{Proof of }(vii).\ \ 
In the category $\ct_\ca^b$ factor the map
$v$ canonically as $Y\stackrel{v'}\la Y^{\geq0}\stackrel {\wt v}\la Z$.
Complete to an octahedron
\[
\xymatrix@C+30pt{
 &  X^{\leq-1}\ar[d]^{u''} \ar[r]^-{u^{\leq-1}} & Y^{\leq-1}\ar[d]^{u'} &  \\
  Z[-1]\ar@{=}[d]\ar[r]^-{-w[-1]} & X\ar[d]^{v''}\ar[r]^-{u} & Y \ar[d]^{v'}\ar[r]^-{v} & Z\ar@{=}[d]\\
  Z[-1]\ar[r]^-{-\wt w[-1]} & X^{\geq0}\ar[r]^-{u^{\geq0}}\ar[d]^{w''} &Y^{\geq0}\ar[d]^{w'} \ar[r]^-{\wt v} & Z\\
 &  X^{\leq-1}[1] \ar[r]^-{u^{\leq-1}[1]} & Y^{\leq-1}[1] &  
}\]
In the category $\car$ complete to an octahedron the composable morphisms
$G(X^{\geq0})\stackrel{G(u^{\geq0})}\la G(Y^{\geq0})\stackrel{G(w')}\la G(Y^{\leq-1}[1])$.
We obtain a diagram where the rows and columns are triangles
\[
\xymatrix@C+30pt{
 &  G(X^{\leq-1})\ar[d]^{G(u'')} \ar[r]^-{G(u^{\leq-1})} & G(Y^{\leq-1})\ar[d]^{G(u')} &  \\
  G(Z)[-1]\ar@{=}[d]\ar[r]^-\alpha & G(X)\ar[d]^{G(v'')}\ar[r]^-\beta &G(Y) \ar[d]^{G(v')}\ar[r]^-{G(v)} & G(Z)\ar@{=}[d]\\
  G(Z)[-1]\ar[r]^-{-G(\wt w)[-1]} & G(X^{\geq0})\ar[r]^-{G(u^{\geq0})}\ar[d]^{G(w'')} &G(Y^{\geq0}) \ar[r]^-{G(\wt v)}\ar[d]^{G(w')} & G(Z)\\
&  G(X^{\leq-1}[1]) \ar[r]^-{G(u^{\leq-1}[1])} & G(Y^{\leq-1}[1]) & }\]
The third row is the triangle in $\car$ obtained by applying (vi) to
the triangle
$X^{\geq0}\stackrel{u^\geq0}\la Y^{\geq0}\stackrel{\wt v}\la Z \stackrel{\wt w}\la X^{\geq0}[1]$. The second and third column are by applying (v) to
the triangles $X^{\leq -1}\stackrel{u''}\la X\stackrel{v''}\la X^{\geq0}\stackrel{w''}\la X^{\leq-1}[1]$ and
$Y^{\leq -1}\stackrel{u'}\la Y\stackrel{v'}\la
Y^{\geq0}\stackrel{w'}\la Y^{\leq-1}[1]$.
The map $\alpha$ is a morphism $G(Z[-1])\la G(X)$ with
$Z[-1]\in\ct^{\geq0}\cap\ct^b_\ca$ and $X\in\ct^{\leq0}\cap\ct^b_\ca$, and
(i) tells us that $\alpha=G(a)$ with $a:Z[-1]\la X$ a morphism in
$\ct_\ca^b$. 
Now the fact that $G$ is a functor gives us a commutative square
\[
\xymatrix@C+20pt{
 G(X^{\leq-1})\ar[d]^{G(u'')} \ar[r]^-{G(u^{\leq-1})} & G(Y^{\leq-1})\ar[d]^{G(u')} \\
 G(X)\ar[r]^-{G(u)} &G(Y) 
}\]
and hence $[\beta-G(u)]G(u'')=0$. It follows that $\beta-G(u)$ factors
as $\gamma G(v'')$ for some $\gamma:G(X^{\geq0})\la G(Y)$. Since $X^{\geq0}$
belongs to $\ct^{\geq 0}\cap\ct^b_\ca$ and
$Y$ belongs to $\ct^{\leq0}\cap\ct^b_\ca$, part (i)
tells us that there must
be a map $g:X^{\geq0}\la Y$ with $G(g)=\gamma$. Thus $\beta=G(u)+\gamma
G(v'')=G(u) +G(g)G(v'')=G(u+gv'')$. Putting $b=u+gv''$ we have that 
 $G(Z)[-1]\stackrel{G(a)}\la G(X)\stackrel{G(b)}\la G(Y)\stackrel{G(v)}\la
G(Z)$ is a triangle in $\car$, and by (iv) so is 
$G(Z)[-1]\stackrel{G(-w[1]))}\la G(X)\stackrel{G(u)}\la G(Y)\stackrel{G(v)}\la
G(Z)$.\hfill{$\Box$}

\medskip

It remains to conclude the proof of Theorem~\ref{T1.13}, we must show
that $G$ takes any triangle in $\ct_\ca^b$ to a triangle in $\car$. The
proof will be by induction on the length of the cohomology sequence. 
We have a homological functor $\ch:\ct_\ca^b\la \ca$,
from $\ct_\ca^b$ to the heart of
its \tstr. Every triangle in $\ct_\ca^b$ maps under this functor to a long
exact
sequence in $\ca$ which vanishes outside a bounded interval. We let
$\ell$ be the smallest integer so that the long exact sequence has 
nonzero terms all contained in an interval of length $\ell$.

If $\ell\leq3$ then a triangle $X\la Y\la Z\la X[1]$ of length $\ell$
may
be rotated to have $X,Y,Z\in\ca$, and (vi) tells us that $G$ takes it
to a triangle. It remains to prove the induction step: we must show
that
if all triangles of length $<\ell$ map under $G$ to triangles, then so
do triangles of length $\ell$. Let $X\stackrel u\la Y\stackrel v\la
Z\stackrel w\la X[1]$ be a triangle of length $\ell$. By rotating we
may assume that $X,Y,Z$ all belong to $\ct^{\leq0}\cap\ct^b_\ca$ and
$Z^{\geq0}\neq0$.
In the category $\ct_\ca^b$ complete the composable maps $Y\la Z\la
Z^{\geq0}$ to an octahedron
\[\xymatrix{
 & Z^{\geq0}[-1] \ar@{=}[r]\ar[d]^{w'}  & Z^{\geq0}[-1]\ar[d]^{w''}  & \\
X \ar[r]^-{\wt u}\ar@{=}[d] & \wh Y \ar[r]^-{\wt v}\ar[d]^{u'}  & Z^{\leq-1}\ar[d]^{u''} \ar[r]^-{\wt w}
& X[1] \ar@{=}[d] \\
X \ar[r]^-{u} & Y \ar[r]^-{v}\ar[d]^{v'}  & Z \ar[r]^-{w}\ar[d]^{v''}
& X[1] \\
& Z^{\geq0} \ar@{=}[r]  & Z^{\geq0} & 
}\]
Now complete the composable maps $G(Z^{\geq0}[-1] )\stackrel {G(w')}\la
G(\wh Y) \stackrel {G(\wt v)}\la G(Z^{\leq-1})$
to an octahedron in
$\car$
\[\xymatrix{
 & G(Z^{\geq0}[-1] )\ar@{=}[r]\ar[d]^{G(w')}  & G(Z^{\geq0}[-1])\ar[d]^{G(w'')}  & \\
G(X) \ar[r]^-{G(\wt u)}\ar@{=}[d] & G(\wh Y) \ar[r]^-{G(\wt v)}\ar[d]^{G(u')}& G(Z^{\leq-1})\ar[d]^{G(u'')} \ar[r]^-{G(\wt w)}
& G(X[1]) \ar@{=}[d] \\
G(X) \ar[r]^-{G(u)} & G(Y) \ar[r]^-{\beta}\ar[d]^{G(v')}  & G(Z) \ar[r]^-{\gamma}\ar[d]^{G(v'')}
& G(X[1]) \\
& G(Z^{\geq0}) \ar@{=}[r]  & G(Z^{\geq0}) & 
}\]
The fact that $G$ takes $Z^{\geq0}[-1]\la Z^{\leq-1}\la Z\la
Z^{\geq0}$ to a triangle is by (v). The fact that $G$ takes
$Z^{\geq0}[-1]\la \wh Y\la Y\la
Z^{\geq0}$ to a triangle is by (vii). And the fact that $G$ takes
$X\la \wh Y\la Z^{\leq -1}\la X[1]$ to a triangle is by induction on
$\ell$.

Applying the functor $G$ to the octahedron in $\ct_\ca^b$ gives a commutative
diagram, in particular the two squares in
\[\xymatrix{
 G(\wh Y) \ar[r]^-{G(\wt v)}\ar[d]^{G(u')}& G(Z^{\leq-1})\ar[d]^{G(u'')} \ar[r]^-{G(\wt w)}
& G(X[1]) \ar@{=}[d] \\
 G(Y) \ar[r]^-{G(v)}  & G(Z) \ar[r]^-{G(w)}
& G(X[1])
}\]
both commute. Hence $[\beta-G(v)]G(u')=0$ and $[\gamma-G(w)]G(u'')=0$.
It follows that there exist maps $\beta':G(Z^{\geq0})\la G(Z)$ and
$\gamma':G(Z^{\geq0})\la G(X[1])$ with $\beta-G(v)=\beta'G(v')$
and $\gamma-G(w)=\gamma'G(v'')$. But $Z^{\geq0}$
belongs to $\ct^{\geq0}\cap\ct^b_\ca$ and
$Z,X[1]$ are both in $\ct^{\leq0}\cap\ct^b_\ca$,
and (i) tells us that there exist
morphisms $b:Z^{\geq0}\la Z$ and $g:Z^{\geq0}\la X[1]$ with $\beta'=G(b)$
and $\gamma'=G(g)$. Therefore $\beta=G(v+bv')$ and $\gamma=G(w+gv'')$.
Thus we have produced a triangle
\[
\CD
G(X) @>G(u)>> G(Y) @>G(v+bv')>> G(Z) @>G(w+gv'')>> G(X[1])
\endCD
\]
and (iv) tells us that $G$ takes $X\stackrel u\la Y\stackrel v\la Z
\stackrel w\la X[1]$
to a triangle.
\eprf




\def\cprime{$'$} \def\cprime{$'$} \def\cprime{$'$}
\providecommand{\bysame}{\leavevmode\hbox to3em{\hrulefill}\thinspace}
\providecommand{\MR}{\relax\ifhmode\unskip\space\fi MR }
\providecommand{\MRhref}[2]{%
  \href{http://www.ams.org/mathscinet-getitem?mr=#1}{#2}
}
\providecommand{\href}[2]{#2}

\end{document}